\documentclass[twoside]{irmaems}
\usepackage{amsmath}
\usepackage{amssymb}
\usepackage{latexsym}

\usepackage{amsbsy,amsfonts,amscd}
\usepackage{graphicx}

\usepackage{xy} 
\usepackage{makeidx}
\xyoption{all} 

\usepackage{latexsym}

\newcommand{\bb}[1]{{\mathbb #1}}    

\newcommand{\bbR}{{\bb R}}
\newcommand{\bbA}{{\bb A}}
\newcommand{\bbZ}{{\bb Z}}

\newcommand{\bbC}{{\bb C}}

\newcommand{\lmod}{\backslash}
\newcommand{\PGL}{{\rm PGL}}
\newcommand{\SL}{{\rm SL}}
\newcommand{\SO}{{\rm SO}}
\renewcommand{\O}{{\rm O}} 

\newcommand{\Aff}{{\rm Aff}}
\newcommand{\Aut}{{\rm Aut}}
 \newcommand{\PSL}{{\rm PSL}}

\newcommand{\Out}{{\rm Out}}
\newcommand{\Inn}{{\rm Inn}}

\newcommand{\Diff}{{\rm Diff}}
\newcommand{\Ad}{{\rm Ad}}

\newcommand{\N}{{\rm N}}  
\newcommand{\Z}{{\rm Z}} 

\newcommand{\E}{{\rm E}}
\newcommand{\A}{{\rm A}}

\newcommand{\cA}{{\mathcal{A}}}

\newcommand{\cC}{{\mathcal{C}}}

\newcommand{\cO}{{\mathcal{O}}}
\newcommand{\cU}{{\mathcal{U}}}
\newcommand{\cT}{{\mathcal{T}}}
\newcommand{\cZ}{{\mathcal{Z}}}

\newcommand{\tensor}[1]{\otimes_{#1}}

\newcommand{\id}{{\rm id}}

\renewenvironment{matrix}[1] {\left( \begin{array}{#1}}{\end{array}\right)}

{ \par  \medskip  \noindent  {\bf Definition} \hspace{0.5em} }%
{\par \medskip } 

\renewcommand{\S}{Section~}
\numberwithin{equation}{section}

\newcommand{\Id}{{\mathrm{Id}}}  
\newcommand{\R}{{\mathrm{R}}}

\newcommand{\Dev}{{\mathrm{Dev}}}
\newcommand{\ra}{\rightarrow}

\newcommand{\lev}{{\mathop{\mathrm{lev\, }}}}
\newcommand{\Mid}{{\, \Big\vert \;}}
\newcommand{\Isom}{{\rm Isom}}
\newcommand{\Vect}{{\rm Vect}}

\newcommand{\cH}{{\mathcal H}}
\newcommand{\cQ}{{\mathcal Q}} 
\newcommand{\cF}{{\mathcal F}}
\newcommand{\cK}{{\mathcal K}}
\newcommand{\cE}{{\mathcal E}}
\newcommand{\cP}{{\mathcal P}}
\newcommand{\cV}{{\mathcal V}}
\newcommand{\cM}{{\mathcal M}}
\newcommand{\mT}{{\mathfrak T}}
\newcommand{\mH}{{\mathfrak H}}
\newcommand{\mC}{{\mathfrak C}}
\newcommand{\mB}{{\mathfrak B}}
\newcommand{\mA}{{\mathfrak A}}
\newcommand{\mU}{{\mathfrak U}}
\newcommand{\mD}{{\mathfrak D}}
\newcommand{\lra}{{\, \longrightarrow \, }}

\newcommand{\tGL}{\widetilde \GL}

\renewcommand\theequation{\arabic{section}.\arabic{equation}}

\theoremstyle{definition} 

 \newtheorem{definition}{Definition}[section]
 \newtheorem{remark}[definition]{Remark}
 \newtheorem{example}[definition]{Example}
 \newtheorem{fact}[definition]{Fact}

\theoremstyle{plain}      

 \newtheorem{proposition}[definition]{Proposition}
 \newtheorem{theorem}[definition]{Theorem}
 \newtheorem{corollary}[definition]{Corollary}
 \newtheorem{lemma}[definition]{Lemma}

\newcommand{\dis}{\operatorname{dis}}

\newcommand{\Hom}{\operatorname{Hom}}
\newcommand{\GL}{\operatorname{GL}}

\renewcommand{\PSL}{\operatorname{PSL}}
\renewcommand{\SL}{\operatorname{SL}}

\newcommand{\Hyp}{{\mathbb H}}
\newcommand{\Teich}{{\mathfrak T}}
\newcommand{\Str}{{\mathfrak S}} 
\newcommand{\Mod}{{\mathfrak M}}
\newcommand{\Def}{{\mathfrak D}}
\newcommand{\StrM}{{\mathfrak{DM}}} 
\newcommand{\StrMf}{{\mathfrak{DM}_{f}}} 
\newcommand{\Map}{{\mathrm{Map}}} 
\newcommand{\bbAo}{{\bbA^2 \! - 0}}  
\newcommand{\tbbAo}{\widetilde{\bbA^2 \! - 0}} 
\newcommand{\PbbAo}{{\mathrm P}({\bbA^2 \! - 0})} 
\author{Oliver Baues \thanks{e-mail: baues@math.uni-karlsruhe.de} \\
Institut f\"ur Algebra und Geometrie\\ 
Karlsruher Institut f\"ur Technologie \\ 
D-76128 Karlsruhe
}

\date{December 12, 2011}     

\bigskip
\bigskip
\bigskip
\bigskip

\makeindex

\begin{document}

\setcounter{tocdepth}{1}
\title{The deformations of flat affine structures on the two-torus}

\author{Oliver Baues \thanks{\today; 
} 
}

\address{
Institut f\"ur Algebra und Geometrie\\ 
Karlsruher Institut f\"ur Technologie (KIT)\\ 
D-76128 Karlsruhe\\ 
email:\,\tt{baues@math.uni-karlsruhe.de} 
}

\maketitle

\begin{abstract}  The group action which defines the 
moduli problem for the deformation space of flat affine structures on 
the two-torus is  the action of the affine group $\Aff(2)$ on $\bbR^2$.
Since this action has non-compact stabiliser $\GL(2,\bbR)$, the underlying locally 
homogeneous geometry is highly non-Riemannian. 
In this chapter, 
we describe the deformation space of all flat affine structures on the two-torus. 
In this context 
interesting 
phenomena arise in the topology 
of the deformation space, which, for example,
is \emph{not} a Hausdorff space.  
This contrasts with  
the case of constant curvature metrics, or 
conformal structures on surfaces, which are encountered in classical Teichm\"uller theory. As our main result on the space of deformations of flat affine structures on the two-torus we prove that the holonomy map from the  deformation space to the variety of conjugacy classes of homomorphisms 
from the fundamental group of the two-torus to the affine group is a local homeomorphism. 
%
\end{abstract}



\tableofcontents   

\section{Introduction}  \label{sect:intro}
A flat affine structure on a smooth manifold is specified by an 
atlas with coordinate changes in the group of affine transformations of $\bbR^n$. 
A manifold together with such an atlas  
is called a \emph{flat affine manifold}. 
\index{flat affine manifold} \index{manifold!flat affine}
Equivalently, 
a flat affine manifold is a smooth manifold
which has a flat and torsion-free connection on the tangent bundle.
A particular class of examples is furnished by Riemannian flat manifolds,
but  the class of flat affine manifolds is much larger.  
The study of flat affine manifolds has a long history which can be 
traced back to the local theory of hypersurfaces and Cartan's projective 
connections. Global questions were first studied in the context 
of Bieberbach's theory of crystallographic groups, and they
have gained renewed interest in the more general setting by 
Ehresmann's theory of locally homogeneous spaces, and more 
recently  in Thurston's geometrisation program which shows the importance 
of locally homogeneous structures in the classification of manifolds. 
¥
Flat affine manifolds are \emph{affinely diffeomorphic}
if they are diffeomorphic by a diffeomorphism which looks like an 
affine map in the coordinate charts.  
The universal covering space of a flat affine 
manifold admits a local affine diffeomorphism into affine space 
$\bbA^n = \bbR^n$ which is called the \emph{development map}; 
\index{development maps}
its image is an open subset of $\bbR^n$, called the \emph{development image}. \index{development image}
The development map and image provide rough 
invariants for the classification of flat affine manifolds. 



Benz\'ecri \cite{Benzecri} showed that a closed oriented surface which supports a flat affine structure must be diffeomorphic to a two-torus
,  thereby confirming in dimension two a conjecture of Chern that the Euler characteristic of a compact flat affine manifold must be zero.   
The flat affine structures on the two-torus and their development images 
were partially classified by Kuiper \cite{Kuiper} in 1953. The classification
was completed by independent work of Furness-Arrowsmith \cite{FurArr}
and Nagano-Yagi \cite{NaganoYagi} around 1972. Their works show that the flat affine structures on the two-torus fall into 
four main classes which have
development image the plane $\bbR^2$, the halfspace, the sector, 
or the once punctured plane, respectively. 

The \emph{moduli space} \index{moduli space!of flat affine structures}
of flat affine structures is by definition the set of flat affine structures up to affine diffeomorphism. 
More precisely, the group $\Diff(T^2)$  of all diffeomorphisms of the two-torus $T^2$ acts naturally on the set of flat affine structures on $T^2$. 
The set of  orbits classifies flat affine two-tori
up to affine diffeomorphism; it is called the {moduli space}. 
The \emph{deformation space}  \index{deformation space!of flat affine structures}
 is  the set 
of all flat affine structures divided by the action of the group  $\Diff_0(T^2)$ 
of diffeomorphisms  which are isotopic to the identity. This action classifies 
flat affine structures on $T^2$ up to isotopy, or equivalently affine 
two-tori with a marking. 

The deformation space has a natural topology, which it
inherits from the $C^\infty$-topology on the space of development maps. 
In this chapter, our aim is to describe the topology of the deformation 
space  $\Def(T^2,\bbA^2)$ of all flat affine structures on the two-torus. 
The development process gives, for each flat affine two-torus,  a natural 
homomorphism $h: \bbZ^2= \pi_{1}(T^2) \ra \Aff(2)$ of the fundamental group of the torus to the plane affine group
$\Aff(2) = \Aff(\bbR^2)$. This homomorphism is called
the \emph{holonomy homomorphism} \index{holonomy homomorphism}  and it is defined up to conjugacy
with an affine map. The holonomy thus gives rise to a continuous 
open map $${hol}: \Def(T^2,\bbA^2) \ra \Hom(\bbZ^2, \Aff(2)) / \Aff(2) \;  $$ 
from the deformation space to the space of conjugacy classes of homomorphisms, which is called the \emph{holonomy map}. 
\index{holonomy map} \index{deformation space!holonomy map}
By a general theorem of Thurston and Weil concerning deformations of \emph{locally
homogeneous structures} on manifolds, 
this map has an \emph{open} image.

As such, the deformation space of flat affine structures is the natural 
analogue of the Teichm\"uller space of conformal structures, or, 
equivalently, constant curvature Riemannian metrics, on surfaces. 
Its  construction 
is completely analogous to the definition of the 
Teichm\"uller space for flat Riemannian metrics on the two-torus, 
or hyperbolic constant curvature $-1$ metrics on surfaces $M_{g}$, $g \geq 2$.
In these classic situations, both the Teichm\"uller space $\Teich_{g}$ 
and its quotient the moduli 
space are Hausdorff spaces. 
The Teichm\"uller space  of flat metrics
$\Teich_{1}$ is diffeomorphic to $\bbR^2$, and  the Teichm\"uller space of 
hyperbolic metrics $\Teich_{g}$, $g \ge 2$, is diffeomorphic
to $\bbR^{6g-6}$. Moreover, 
the corresponding holonomy map topologically identifies $\Teich_{g}$ with an open subset of the quotient space  $\Hom(\bbZ^2, \Isom(\bbR^2)) /
 \Isom(\bbR^2)$, for $g=1$, or, respectively, 
 a component of the space 
 $\Hom(\Gamma_{g}, \PSL(2,\bbR)) / \PSL(2,\bbR)$, $g \geq 2$. 
 
Here the analogy with the classical theory breaks down, and
neither of these facts are true for the deformation space of flat affine structures. 
In fact, the group action, which defines the 
moduli problem for the deformation space of flat affine structures on the two-torus, namely the action of the affine group 
$\Aff(2)$ on the homogeneous space $$ X = \bbR^2 =\Aff(2) / \GL(2,\bbR)$$ 
has non-compact stabiliser $\GL(2,\bbR)$, and therefore 
the underlying geometry on $X$ is highly non-Riemannian. 
This is illustrated by the fact that 
various kinds of flat affine structures, 
with sometimes strikingly distinct geometric properties, are supported on the two-torus. A fact which can be seen already from the various possible development images for flat affine structures, and which is also reflected in the structure and  
topology of the deformation space. 
Here phenomena arise which 
are completely different from the case of constant curvature metrics or conformal structures on surfaces.  

Another salient difference stems from the fact that the local model of the deformation space of flat affine structures, namely the
\emph{character variety}  $\Hom(\bbZ^2, \Aff(2)) / \Aff(2)$ arises 
as a quotient space of an algebraic variety by a \emph{non-reductive} group action. The properties of such actions
and their invariant theory are generally poorly understood. 

The case of deformation of \emph{complete} flat affine structures bears the \index{flat affine manifold!complete}
closest resemblance to the classical situation. A flat affine structure is
called \emph{complete} if the development map is a diffeomorphism, a property
which in the Riemannian situation is always guaranteed. 
A flat affine  two-torus is complete if and only if its development
image is the affine plane $\bbA^2$. The 
deformation space  of  complete affine structures
on the two-torus was studied recently in \cite{BauesG, BG}. 
It is shown there,
for example, that the holonomy map identifies the space of complete structures $\Def_{c}(T^2,\bbA^2)$ with 
a locally closed subspace of the space of homomorphisms $\Hom(\bbZ^2, \Aff(\bbR^2))$, and, moreover,  the space $\Def_{c}(T^2,\bbA^2)$ is 
homeomorphic to $\bbR^2$. However, the topology of the moduli space of complete flat affine structures, 
which, with respect to appropiately chosen coordinates for $\Def_{c}(T^2,\bbA^2)$, is homeomorphic to the quotient space of $\bbR^2$ by the natural action of $\GL(2,\bbZ)$,  is highly singular. 

This chapter is devoted to the study of the global and local structure of the space of deformations of \emph{all} flat affine structures on the two-torus. The deformation space of all flat affine
structures  is much larger than the deformation space of complete flat affine structures.  Indeed, the deformation space 
$\Def_{c}(T^2,\bbA^2)$ of complete flat affine structures on the two-torus forms a closed  two-dimensional subspace in the deformation space of all structures $\Def(T^2,\bbA^2)$, which itself is a space of dimension four.  In the general situation the holonomy map ${hol}: \Def(T^2,\bbA^2) \ra \Hom(\bbZ^2, \Aff(2)) / \Aff(2)$ for the deformation space of flat affine structures is no longer a homeomorphism onto its image. That is, there exist flat affine structures on the two-torus, which have the same holonomy group and which have dramatically different geometry. (Compare, in particular, Example \ref{ex:hol_notinj} in this chapter.)
Moreover, the holonomy image in $\Hom(\bbZ^2, \Aff(2))$ contains singular orbits for the affine group $\Aff(2)$, which in turn  give rise to non-closed points in the deformation space $\Def(T^2,\bbA^2)$. This also shows 
that the deformation
space is \emph{not} a Hausdorff space. It is a four-dimensional and connected space which has an intricate topology and it supports various substructures arising from the different types of affine flat geometries on  the two-torus. 

As our main result on the local structure of the space of deformations of flat affine structures 
we prove  in this chapter that \emph{the holonomy map $hol$
 is a local homeomorphism onto its image}. That is, at least locally the topology of the deformation
space $\Def(T^2,\bbA^2)$ is fully controlled by the character variety. We remark that this is \emph{not} a general phenomenon for 
deformation spaces of locally homogeneous structures, \emph{not even on surfaces}. Indeed, in Appendix B of this chapter, we specify a two-dimensional
homogeneous geometry whose deformation space of structures on the two-torus has a holonomy map $hol$ which locally near certain structures is a branched covering. Examples of deformation spaces of flat conformal structures on three-dimensional manifolds where the holonomy map $hol$ is not locally injective at exceptional points were found previously by Kapovich and are discussed in \cite{Kapovich}.\\

The chapter is organized as follows. In \S 2 we give a self-contained proof of Benz\'ecri's theorem which states that 
a closed orientable flat affine surface is diffeomorphic 
to the two-torus. In Section 3 we describe the deformation theory of compact locally homogeneous manifolds, including its
foundational results and give basic examples. Section 4 discusses
several methods to construct flat affine surfaces and introduces
the main classes of flat affine structures on the two-torus. 
In \S 5 we prove the main classification theorem for flat affine structures on the two-torus in detail, including the crucial and nontrivial fact that the development map of such a structure is always a covering map. Finally, in \S 6 we put the pieces
together in order to prove that the holonomy map for the deformation space of flat affine structures on the two-torus is a local homeomorphism to the character variety. In addition, Appendix A gives an account on conjugacy classes in $\GL(2,\bbR)$ and in its universal covering group. In Appendix B we describe a two-dimensional homogeneous geometry such that the holonomy map for  its deformation space of structures on the two-torus is not everywhere a local homeomorphism.

\subsubsection*{Acknowledgement} 
The author wishes to thank Wolfgang Globke, Bill Goldman and Athanase
Papadopoulos for their interest, advice and support 
during the long gestation of this article. I thank Athanase Papadopoulos especially for inviting this project as a contribution to Volume III of the ``Handbook of \mbox{Teichm\"uller} theory'', and Bill Goldman for sharing his 
insight on the subject. Most pictures
in the article were created by Wolfgang Globke with the software Omnigraffle for Macintosh.

\newpage

\begin{figure}[htbp]

\vspace{1.5cm}
\fbox{\begin{minipage}{0.9\textwidth} The following and other similar pictures
illustrate convergence of development maps in the deformation space of flat affine structures on the two-torus. Each development map gives rise to a tiling of an open domain in affine space which is deformed with the change of development maps.
\end{minipage}}
\vspace{1.5cm}

  \hspace{-3.1cm}
  \fbox{
    \includegraphics[width=21cm]{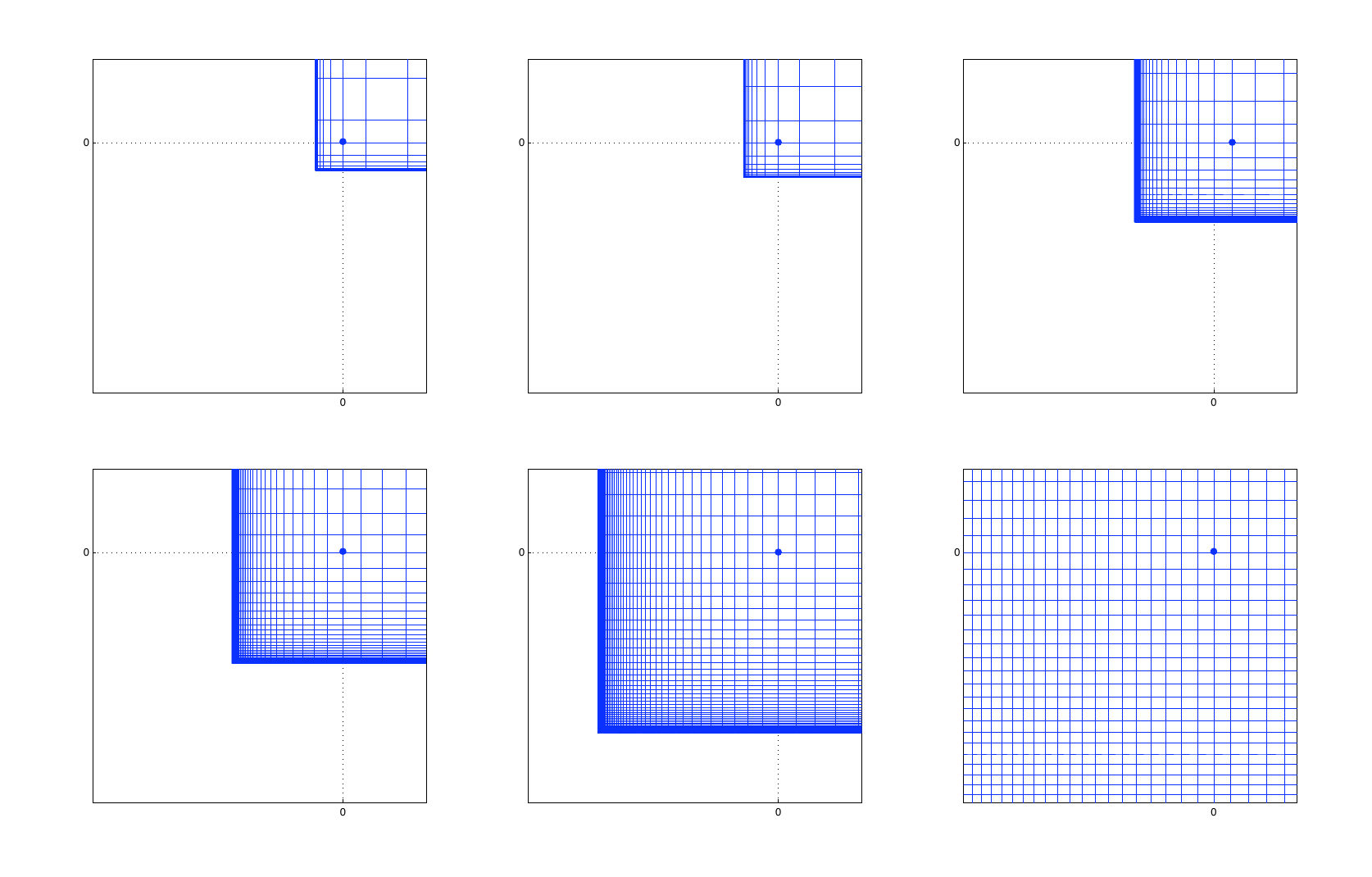}
  }
  \caption{Tiled sectors approaching the standard plane.}
  \label{fig:sect-plane}
\end{figure}

\begin{figure}[htbp]
  \hspace{-3.1cm}
  \fbox{
    \includegraphics[width=21cm]{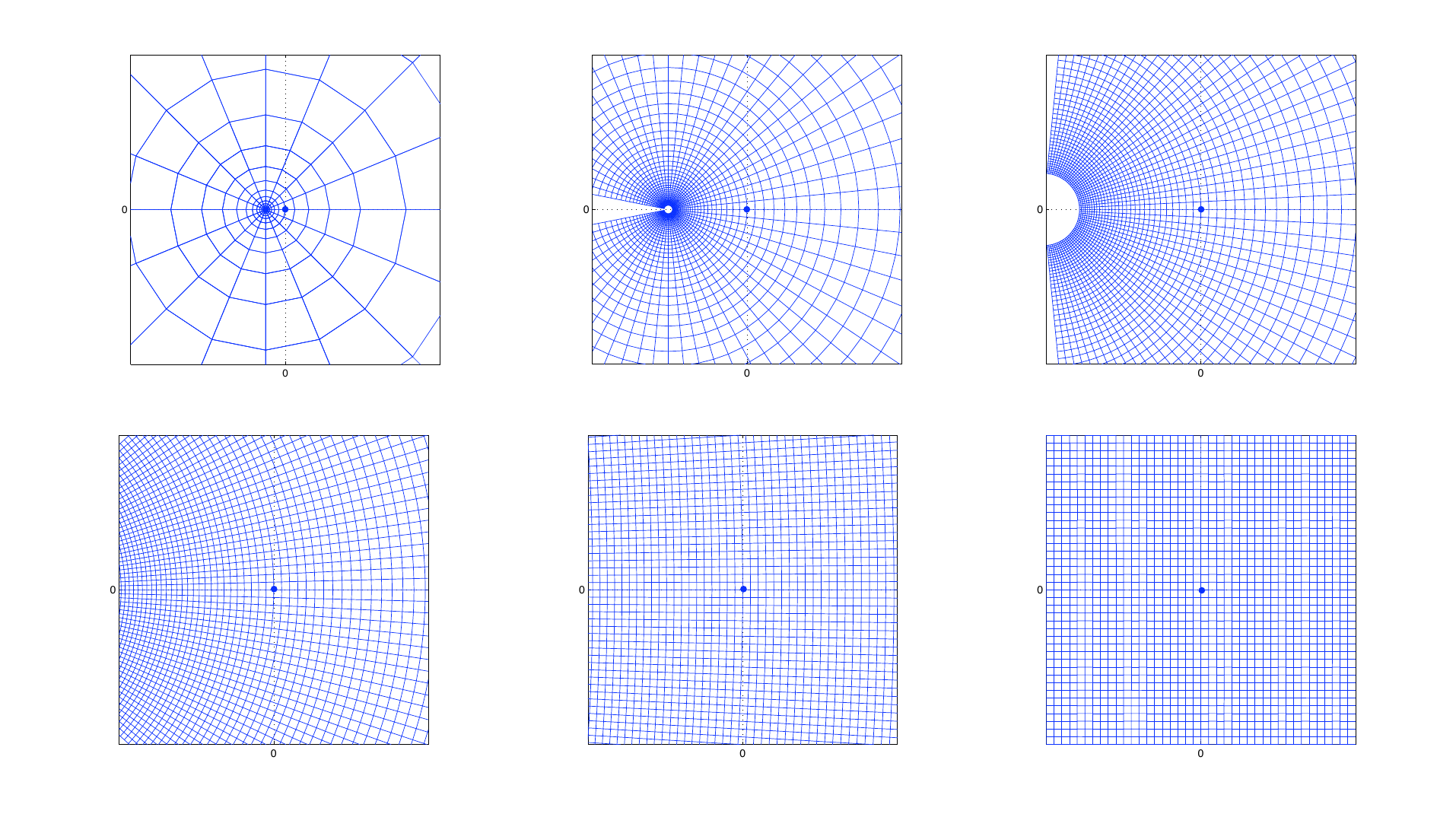}
  }
  \caption{Tiled punctured planes approaching the standard plane.}
  \label{fig:pplane-plane}
\end{figure}

\begin{figure}[htbp]
  \hspace{-3.1cm}
  \fbox{
    \includegraphics[width=21cm]{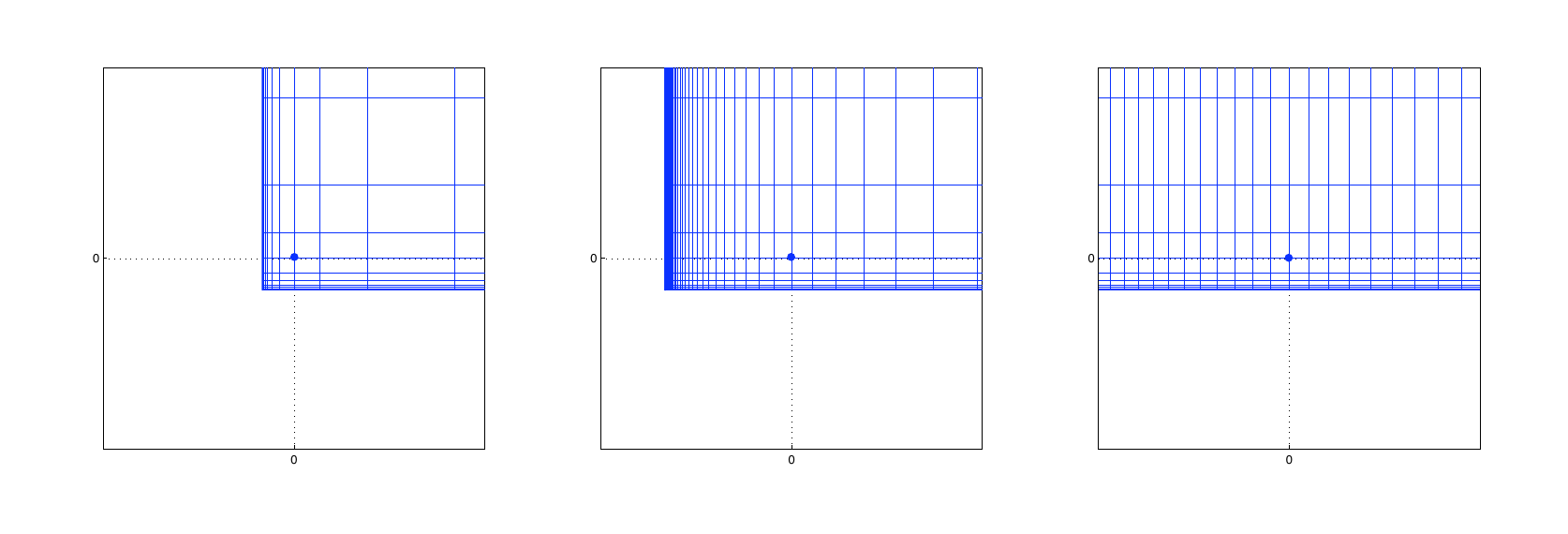}
  }
  \caption{Tiled sectors approaching a halfplane of type $\mathsf{C}_{2}$.}
  \label{fig:sect-hplane}
\end{figure}

\begin{figure}[htbp]
  \hspace{-3.1cm}
  \fbox{
    \includegraphics[width=21cm]{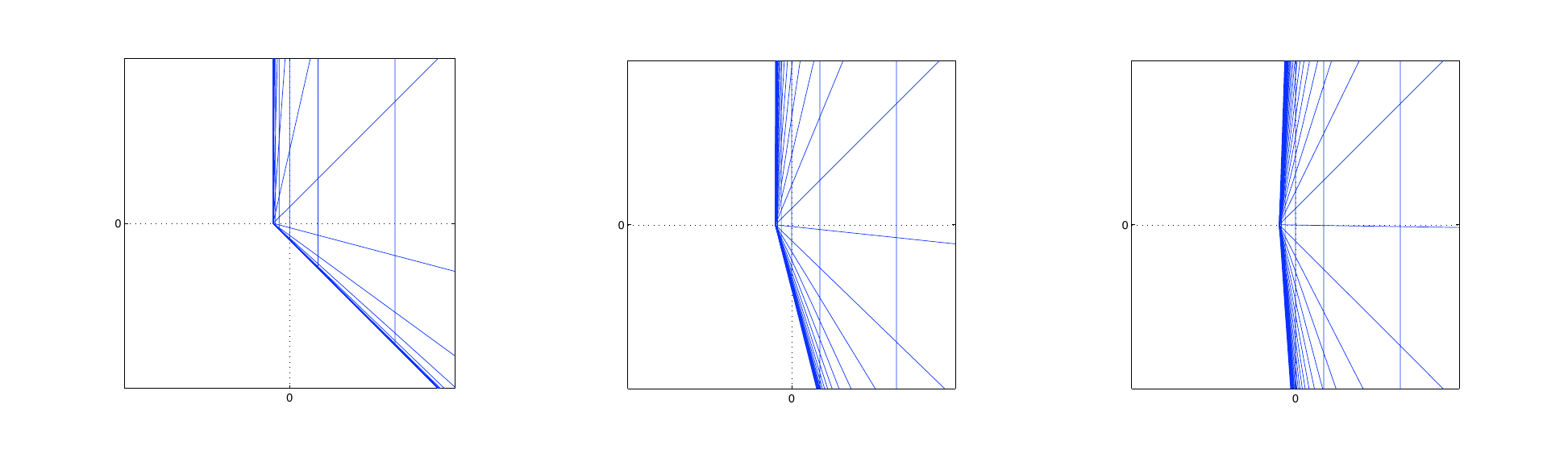}
  }
  \caption{Tiled sectors approaching a halfplane of type $\mathsf{C}_{1}$.}
  \label{figure:BC1}
\end{figure}

\begin{figure}[htbp]
  \hspace{-3.1cm}
  \fbox{
    \includegraphics[width=21cm]{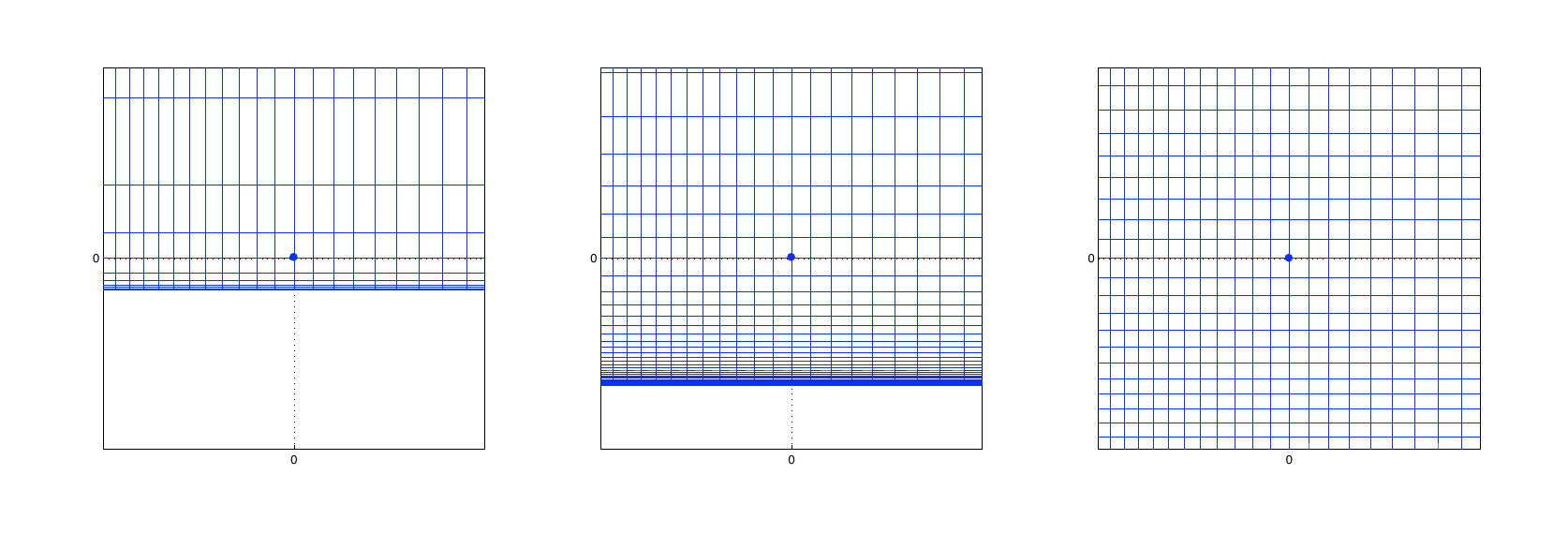}
  }
  \caption{Tiled halfplanes approaching the standard plane.}
  \label{fig:hplane-plane}
\end{figure}


\section{The theorem of Benz\'ecri } 
Let $M$ be a closed oriented  surface of genus $g$. 
Then the  Gau\ss-Bonnet theorem \cite{Hopf2} expresses the Euler characteristic
$$ \chi(M)= 2 -2g$$  as an integral over the Gau\ss\ curvature
of any Riemannian metric on $M$. In particular, a flat Riemannian closed
surface $M$ has Euler characteristic zero, and therefore it is 
diffeomorphic to a two-torus. If $M$ is a closed flat affine surface, the
Gau\ss-Bonnet theorem does not apply, since the corresponding flat
connection is possibly non-Riemannian. However, the strong 
topological restriction applies to flat affine surfaces, as well:
 \index{surface!flat affine} \index{Benz\'ecri's theorem} \index{Gau\ss-Bonnet theorem} \index{Euler characteristic} 
 \index{flat affine torus} \index{flat affine surface}

\begin{theorem}[Benz\'ecri, \cite{Benzecri}] \label{thm:Benzecri}
Let $M$ be a closed flat affine surface. 
Then $M$ has Euler characteristic zero.
\end{theorem} 
\begin{proof} First we remark that the sphere $S^2$ does not admit a flat affine 
structure. In fact, since $S^2$ is simply connected,
the development image of a flat affine structure on $S^2$ 
would be compact and open in $\bbR^2$, which is absurd. 
 
Now we assume that $M$ has genus $g$, $g \geq 1$. Let 
$\mathsf{p}: \tilde M \ra M$ be the universal covering. 
Then $\tilde M$ is a flat affine manifold which is diffeomorphic to
$\bbR^2$. Moreover, $M$ is obtained by gluing 
a $4g$-gon $P \subset \tilde M$ 
along its consecutive  sides $a_{1}, b_{1}, a_{1}^{-}, b_{1}^-, \ldots,
a_{g}, b_{g}, a_{g}^{-}, b_{g}^-$,
with side pairing transformations  $g_{a_{i}}, g_{b_{i}}$
such that $g_{a_{i}} a_{i}^- = a_{i}$ and $g_{b_{i}}  b_{i}^- = b_{i}$. 
These transformations are subject to the single cycle relation  
$$  \prod_{i= 1,\ldots ,g} \; g_{a_{i}} g_{b_{i}}
 g_{a_{i}}^{-1} g_{b_{i}}^{-1} = \id_{\tilde M} $$
and generate the discontinuous group of deck transformations 
of the covering $\mathsf{p}: \tilde M \ra M$. 
In particular, the side pairing transformations are affine maps of $\tilde M$. 
Note however that the polygon $P \subset \tilde M$ 
is a closed oriented topological disc with piecewise 
smooth boundary. (The construction 
may be carried out,  in Euclidean geometry if $g=0$,
respectively hyperbolic geometry, for $g \geq 2$,  such that the edges of $P$ 
are geodesic segments. See, for example,  \cite{Ratcliffe}.)

Let  $\tilde x_{0}$ denote the vertex of $P$ belonging to the sides $a_{1}$ and
$b^-_{g}$. Let $v \neq 0$ be a tangent vector at $\tilde x_{0}$. 
Now choose a \emph{non-vanishing} vector field
$V$ along the boundary of $P$, such that $V(\tilde x_{0}) = v$, and, 
furthermore, such that $V$ restricted to $a_{i}$ (resp.\ $b_{i}$)
is related to $V$ restricted to  $a^-_{i}$ (resp.\ $b^-_{i}$)
by the corresponding side pairing transformation. (We can obtain such a $V$ 
by constructing vector fields along the closed curves $\mathsf{p} \,   a_{i}$,  $\mathsf{p} \, b_{i}$ which coincide at $x_{0} = \mathsf{p}(\tilde x_{0})$.)  
Next we extend $V$ to a vector field $X$ on $P$, which has an 
isolated singularity in the interior of $P$. 

The \emph{index} of the vector field $X$ at the singularity may be calculated as the 
\emph{turning number} of the restriction of $X$ to the positively traversed 
boundary of $P$, see \cite{Hopf2}. For this, recall that the turning 
number  $\tau(V) \in \bbZ$ 
of a closed non-vanishing vector field 
$V:  I \ra \,  \bbAo$ is defined by the equation 
$$\tau(V) \, 2\pi = \theta(1) - \theta(0) , $$ where
$\theta: I \ra \bbR$ is any lift of of the map $I \ra S^1$,  $t \mapsto V(t) / |V(t)|$. 
Now, since the flat affine manifold $\tilde M$ is simply connected, $\tilde M$ has
a global parallelism which identifies each tangent space $T_{x} \tilde M$
with $T_{\tilde x_{0}} \tilde M$.  Therefore, we can choose any scalar product in $T_{\tilde x_{0}} \tilde M$ to compute the index of $X$ by the above formula.  


Since the side pairing maps $g_{a_{i}}$ and $g_{b_{i}}$ are affine 
transformations of $\tilde M$, they preserve antipodality of any two vectors 
$V(s)$ and $V(t)$.  This implies that the turn of $V$ restricted to
the positively traversed curve $ a_{1} b_{1}  a_{1}^{-} b_{1}^-$ 
is less than $2 \pi$. And consequently,  $|\tau( V )| <  g$. 

Our construction implies that the vector field $X$ on $P$ projects to
a vector field  on $M$. Therefore, by the Poincar\'e-Hopf theorem \cite{Hopf2, MilnorTop}, the index of $X$ equals the Euler 
characteristic $\chi(M)$ of $M$.  We thus obtain the 
estimate $$  | \chi(M) | =  | 2-2g | \,  < g \; . $$
This implies  $g = 1$. 
\end{proof}

Benz\'ecri's theorem was generalised by Milnor \cite{Milnor}  to 
the more general 
\begin{theorem} Let $E$ be a flat rank
two vector-bundle over a closed orientable surface $M_{g}$, $g \geq 1$,
then $  |\chi(E) | <  g  $.  
\end{theorem} 

Here, $\chi(E)$ denotes the evaluation of the Euler-class $e(E) \in H^2(M,\bbZ)$
on the fundamental homology class of $M$. In case of  the tangent bundle 
$E =TM_{g}$, the equality  $$ \chi(TM_{g}) = \chi(M_{g})$$  (see
\cite[\S 11]{MilnorStasheff}) implies Benz\'ecri 's theorem. Wood  \cite{Wood}
interpreted Milnor's result in the context of circle-bundles. See  \cite{Goldman5}
for a recent survey on Benz\'ecri 's theorem, Milnor's inequality and related topics.

\paragraph{Generalizations to higher dimensions}
Weak analogues of the Milnor-Wood estimate for the Euler-class of higher-dimensional manifolds were
subsequently given by Sullivan \cite{Sullivan} and Smillie in his doctoral thesis \cite{Smillie}. 
See also \cite{BucherGelander} for a recent contribution in this realm. 

The Chern conjecture 
asserts that any compact flat affine \index{Chern conjecture}
manifold should have Euler characteristic zero. 
Kostant and Sullivan \cite{KS} observed that every compact \emph{and} complete
flat affine manifold has Euler characteristic zero.  
There are some additional affirmative results under
various assumptions on the holonomy group, see
for example \cite{Goldman}. 
The original conjecture, however, remains a difficult  
open problem. 

Another fruitful generalization of the Milnor-Wood inequality 
concerns the representation theory 
of surface groups into higher-dimensional simple Lie groups,
see \cite{BIW} for a survey. 

\section{Locally homogeneous structures and their deformation 
spaces} 

Let $M$ be a smooth manifold and fix a universal 
covering space $\mathsf{p}:\tilde{M} \rightarrow M$. 
Let $X$ be a homogeneous space for the
Lie group $G$, on which $G$ acts effectively. 
The manifold $M$ is said to be 
locally modeled on $(X,G)$ \index{manifold!$(X,G)$-}
if $M$ admits an atlas of charts with range in $X$ such that the
coordinate changes are locally restrictions of elements of $G$. 
A maximal atlas with this property
is called an {\em $(X,G)$-structure} on $M$. 
The manifold $M$ together with an $(X,G)$-structure is called an
{\em $(X,G)$-manifold}, or \emph{locally homogeneous space} 
\index{locally homogeneous space}
modeled on $(X,G)$. A map between two $(X,G)$-manifolds is called an \emph{$(X,G)$-map} if it coincides with the action of an element of 
$G$ in the local charts. If the $(X,G)$-map is a diffeomorphism it
is called an {\em $(X,G)$-equivalence} and accordingly the two 
manifolds are called \emph{$(X,G)$-equivalent}.

\subsection{$(X,G)$-manifolds, development map and holonomy}
Every $(X,G)$-manifold comes equipped with some extra structure, called the development and the holonomy. 
Via the covering projection $\mathsf{p}:\tilde{M} \rightarrow M$ the universal covering space of 
the $(X,G)$-manifold $M$ inherits a unique $(X,G)$-structure
from $M$.  We fix $x_0 \in M$, and 
a local $(X,G)$-chart at $x_0$.  
The corresponding 
{\em development map\/} of the $(X,G)$-structure 
is the $(X,G)$-map  
$$ D:  \tilde{M}  \rightarrow   X$$ which is obtained by 
analytic continuation of the local chart. 
\index{development maps!for $(X,G)$-structures}

For every  $(X,G)$-equivalence $\Phi$ of $\tilde{M}$,
there exists a unique element $h(\Phi) \in G$ such that
\begin{equation}  D \circ \Phi= h(\Phi) \circ D \; .  \label{eq:devel1}
\end{equation} 
The fundamental group $\pi_1(M) = \pi_1(M,x_0)$ acts 
on  $\tilde{M}$ via deck transformations. 
This induces the {\em holonomy homomorphism\/} \index{holonomy homomorphism} 
$$h: \pi_1(M,x_0) \longrightarrow G$$ which satisfies 
\begin{equation}  \label{eq:hol} D \circ \gamma =   h(\gamma)  \circ D  \; ,  \;  
\text{ for all $\gamma \in \pi_1(M,x_0)$.}
\end{equation}  
Note that, after the choice of the development map (which corresponds to a choice of a germ of an $(X,G)$-chart in $x_0$ and also the choice 
of a lift $\tilde{x}_{0} \in \tilde{M}$ of $x_{0}$), the holonomy homomorphism 
$h$ is well defined.  
Therefore, the $(X,G)$-structure on $M$ determines 
the \emph{development pair} \index{development pair} 
$(D,h)$ up to the action of $G$, where
$G$ acts by left-composition on $D$, and by conjugation on $h$.
%
Specifying a development pair is equivalent to constructing an 
$(X,G)$-structure on $M$:

\begin{proposition} \label{prop:devpair}
Every local diffeomorphism $ D:  \tilde{M}  \rightarrow  X$ which satisfies \eqref{eq:hol},
for some $h: \pi_1(M,x_0) \rightarrow G$,
defines a unique $(X,G)$-structure on $M$,
and every $(X,G)$-structure on $M$ arises in this way.
\end{proposition}

\subsubsection{Compactness and completeness of 
$(X,G)$-manifolds}
An important special case arises if the development map
is a diffeomorphism. Recall the following definition:

\begin{definition}[\textbf{Proper actions}] 
A discrete group $\Gamma$ is said to
act \emph{properly discontinuously}  on $X$ if, for all compact
subsets $\kappa \subseteq X$, the set $$\Gamma_{\kappa} = 
\{ \gamma \in \Gamma \mid \gamma
\kappa \cap \kappa \neq \emptyset \} $$ is finite. 
More generally, if $\Gamma$ is a locally compact group, and
 $\Gamma_{\kappa}$ is required to be compact, 
then the action is called \emph{proper}. 
\end{definition}

\begin{example}[\textbf{$(X,G)$-space forms}] \index{space form}
Let $\Gamma$ be a group of $(X,G)$-equivalences
acting properly discontinuously and freely on $X$. Then 
$X/\, \Gamma$ is a manifold which inherits an $(X,G)$-structure
from $X$.
If $X$ is simply connected the identity map
of $X$ is a development map for $X/\, \Gamma$. 
\index{$(X,G)$-space form}
\end{example}

In general, if the development map is a covering map 
onto  $X$, the $(X,G)$-manifold $M$ will be called  \emph{complete}. \index{manifold!$(X,G)$-} \index{$(X,G)$-structures!complete}
 \\

Simple examples (cf.\ the Hopf tori in Example \ref{ex:hopftori})
show that compactness of $M$ 
does not imply completeness.  

\begin{example}[\textbf{Compactness and completeness}] \label{ex:cac}
If $G$ acts properly on $X$ then
every \emph{compact}  $(X,G)$-manifold is complete.
\end{example}

In general, the relation between the properties of the $G$-action on $X$, and the completeness properties of compact 
$(X,G)$-manifolds is only vaguely understood. 
See \cite{Carriere} for a striking contribution in this direction
in the context of flat affine manifolds. 
Further discussion of $(X,G)$-geometries and the properties of the development process may be found in \cite{Epstein,Thurston}. \\

It may well happen that an $(X,G)$-geometry does not
admit (non-finite) proper actions (see \cite{Kob,Kulkarni,IozziWitte}) or no compact $(X,G)$-manifolds 
at all \cite{Benoist_Labourie}.

\begin{example}[\textbf{Calabi-Marcus phenomenon}] 
\label{ex:bbAo1}
Let $\bbAo$ be the once-~punc\-tured affine plane. 
It is easily observed that every discrete subgroup of $\SL(2,\bbR)$
which acts properly on $\bbAo$ must be finite, see Figure \ref{fig:CalabiMarcus}.  This is called the \emph{Calabi-Markus phenomenon}. \index{Calabi-Markus phenomenon}
\end{example}


\begin{figure}[htbp] 
  \centering
    \includegraphics[clip=true, height= 2cm, trim= 32cm 12cm 10cm 12cm]{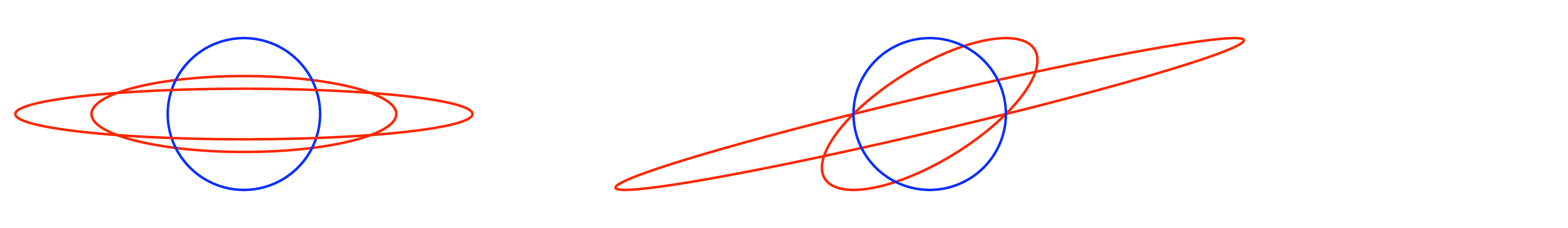}
  \caption{Dynamics of a hyperbolic rotation and a shearing acting on $\bbAo$}
  \label{fig:CalabiMarcus}
\end{figure}

It follows that the homogeneous space $(\bbAo, \SL(2,\bbR))$ has only quotients by
finite groups. Therefore, a complete space modeled on $(\bbAo, \SL(2,\bbR))$ cannot
be compact. In fact, we will remark in Example \ref{ex:bbAo2} below that there do not exist compact $(\bbAo, \SL(2,\bbR))$-manifolds at all. 

\paragraph{Prominent $(X,G)$-structures on surfaces}
Let $M_{g}$ denote an orientable  surface of genus $g$.
In the context of this paper,  the following $(X,G)$-structures play a prominent  role. \index{$(X,G)$-structures!on surfaces}
\begin{example} 
\hspace{1cm}
\begin{enumerate} 
\item $(\mathbb{S}^2, \O(2))$, spherical geometry, $g=0$.
\item $(\bbR^2, \E(2))$, plane Euclidean geometry, $g=1$.
\item $({\mathbb H}^2,  \PSL(2,\bbR))$, plane hyperbolic geometry, $g \geq 2$.
\item $(\bbR^2, \Aff(2))$, plane affine geometry, $g=1$.
\item $({\mathbb P}^2(\bbR), \PSL(3,\bbR))$, plane projective geometry, $g \geq 0$.
\end{enumerate} 
Every compact orientable surface of genus $g \geq 2$ supports hyperbolic structures. 
Also every compact surface supports a projective structure, see \cite{ChoiGoldman}. 
By Benz\'ecri's theorem the only compact surfaces which support a flat affine structure are the two-torus and the Klein bottle. The classification of flat affine structures on the two-torus 
was completed in the 1970's, see \S \ref{sect:classification} of this article. Subsequently, Bill Goldman in his undergraduate thesis \cite{Goldman0} classified projective structures on the two-torus in 1977. 
\end{example}


Note that
every Euclidean or hyperbolic compact surface is complete
(compare Example \ref{ex:cac}). 
The majority of flat affine structures on the two-torus are \emph{not}  complete but  the development map of a flat affine structure on 
the two-torus is always a covering onto its image (see Theorem 
\ref{thm:classification}). The 
development map of a projective structure 
on a surface may not even be a covering \cite{ChoiGoldman}. 

\subsubsection{$(X,G)$-subgeometries} \label{sect:sub_geom}
We may relate different locally homogeneous geometries by inclusion
as follows.
\begin{definition} \label{def:subgeometry}
Let  $(X,G)$ and  $(X',G')$ be homogeneous spaces and 
$\rho: G' \ra G$ a homomorphism 
together with a $\rho$-equivariant local diffeomorphism 
$o:  X' \ra X$. 
Then we say that \emph{$(X',G')$ is subjacent to} or 
\emph{a subgeometry of  $(X,G)$}. The subgeometry is 
called \emph{full} if the map $o$ is surjective onto $X$.
The subgeometry is called a \emph{covering of geometries} 
if $o:  X' \ra X$
is a regular covering map with group of deck transformations
precisely the kernel of $\rho$. 
\end{definition}  
\noindent 
If  $(X',G')$ is a subgeometry of $(X,G)$ then $X'$ is an
$(X,G)$-manifold with development map $o: X' \ra X$.
Note also that $o$ is a covering map onto its image,  
since $o$ is an equivariant map of homogeneous spaces.
The group $G'$ then acts as a group of $(X,G)$-equivalences of $X'$, 
so that $X'$ is, in fact,  a homogeneous $(X,G)$-manifold. 

\begin{example}
Let $\mathsf{p}: \tbbAo \,\ra \, \bbAo$ be the universal covering 
of the once-punctured affine plane $\bbAo$, and $\widetilde \GL(2,\bbR) \ra  \GL(2,\bbR)$
the universal covering group of $\GL(2,\bbR)$. Then the subgeometry $$ 
\mathsf{p}: (\widetilde \bbAo, 
\widetilde \GL(2,\bbR)) \;  \lra \; (\bbAo, \GL(2,\bbR))$$
 is a full subgeometry and, indeed, it is a covering of geometries.  
\end{example}

If  $(X',G')$ is a subgeometry of $(X,G)$ then, in particular, \emph{every $(X',G')$-manifold with development
map $D'$ inherits naturally an $(X,G)$-manifold structure 
with development map $o \circ D'$}. \\

This  observation provides a useful tool to construct $(X,G)$-manifolds. 
Assume, for instance,  that $G'$ acts properly on $X'$ and $\Gamma' \leq G'$ is a discrete subgroup. Then $\Gamma' \lmod X'$ is an $(X',G')$-manifold which inherits an $(X,G)$-structure via $o$. 
The following special case is of particular
importance:

\begin{definition}[\textbf{\'Etale $(X,G)$-representations}] \label{def:etale}
%
\index{etale@\'etale representation} \index{representation!\'etale}  
If $G'$ acts on $X'$ with finite stabilizer then an inclusion of geometries $\rho: G' \ra G$ as above is  called an \emph{\'etale representation} of 
$G'$ into $(X,G)$. 
\end{definition}

If $\rho$ is \'etale with open orbit $\rho(G') x_{0}$ the group manifold $G'$ inherits via the 
orbit map $$ o: G' \ra X \, , \; g' \mapsto \rho(g') x_{0}$$ a natural 
$(X,G)$-structure which is invariant by left-multiplication of $G'$. In particular, if $\Gamma' \leq G'$ is a discrete subgroup then the coset space $\Gamma' \lmod G'$ inherits an $(X,G)$-manifold structure.   


\begin{example}[\textbf{Geometries subjacent 
to the punctured plane}]
The affine automorphism group of the once punctured 
affine plane $\bbAo$ is the linear group $\GL(2,\bbR)$.
The homogeneous geometry $\left(\bbAo, \GL(2,\bbR)\right)$ has full subgeometries $\left(\bbAo, \GL(1,\bbC)\right)$ and $(\bbAo, \SL(2,\bbR))$. Note that 
the first one arises from an \'etale affine representation of the abelian Lie group $\bbR^2$. 
Further subgeometries, which are not full,  
are defined by the abelian 
\'etale Lie subgroups $\mathsf{C}_{1}$ and $\mathsf{B}$, which  are listed in (2) and (3) of Example \ref{ex:domains}. 
Of course, all these homogeneous spaces define particular subgeometries of plane affine geometry, as well.
\end{example}

\subsubsection{Existence of compact forms}
A compact manifold $M$, which is locally modeled on $(X,G)$, 
will be called a \emph{compact form} for $(X,G)$. Given a
homogeneous space $(X,G)$, it is possibly a difficult problem to decide if it has compact form. 

\begin{example}[\textbf{$(\bbAo, \, \SL(2,\bbR))$ has no compact form}] 
\label{ex:bbAo2}
By the Calabi-Markus phenomenon (see Example \ref{ex:bbAo1}),  
$(\bbAo, \, \SL(2,\bbR))$ has only quotients by finite groups. 
Since $(\bbAo, \, \SL(2,\bbR))$ is a subgeometry of plane affine geometry, Benz\'ecri's theorem (Theorem \ref{thm:Benzecri}) and
the classification of flat affine structures with development image 
$\bbAo$ (see Theorem \ref{thm:classification} )
imply the stronger result that \emph{there is no compact locally homogeneous surface modeled on the homogeneous space $(\bbAo, \SL(2,\bbR))$}.
\end{example}  

On the contrary, the 
spaces $(\bbAo,\GL(2,\bbR))$ and $\left(\bbAo, \GL(1,\bbC)\right)$ evidently 
have complete compact forms. For example, every lattice 
$\Gamma \leq \GL(1,\bbC)$ acts properly discontinuously and
freely on  $\bbAo$, which thus gives rise to  a compact flat affine manifold $\bbAo\, \big/ \, \Gamma$.\\ 

Benz\'ecri's theorem implies that 
every orientable compact form of the space $(\bbAo,\GL(2,\bbR))$ 
is diffeomorphic to the two-torus, and in particular it has abelian
fundamental group $\bbZ^2$. 

\begin{example}[\textbf{Compact forms of $(\bbAo,\GL(2,\bbR))$}]
The classification theorem asserts that the development 
map of a flat affine structure on the two-torus is a covering map onto the development image (cf.\ Proposition \ref{prop:Deviscovering}). 
In particular, 
every compact locally homogeneous surface modeled on the homogeneous spaces $(\bbAo, \GL(1,\bbC))$ or $(\bbAo, \GL(2,\bbR))$ is 
\emph{diffeomorphic to the two-torus} and 
either it is complete (which is always true for $(\bbAo, \GL(1,\bbC))$-structures)
or its development image is a sector of halfspace in $\bbA^2$
(see \S \ref{sect:etale_affine}).  
\end{example}

In \S \ref{sect:tori_bbao}, we describe the construction of
all compact $(\bbAo, \GL(2,\bbR))$- manifolds which are
\emph{complete}. The classification theorem for all structures, including
the non-complete case, is stated in \S \ref{sect:bbaoclass}.

\subsection{Convergence of development maps}
The space of  $(X,G)$-\-de\-ve\-lop\-ment maps for the manifold $M$  \index{development maps!space of}  \index{development maps!convergence of}
is the set $$ \Dev(M) = \Dev(M,X,G) $$ of all local 
 $C^\infty$-diffeomorphisms 
$$ D: \tilde{M} \rightarrow X $$ which, 
for some $h \in \Hom(\pi_1(M),G)$,  and, for all $\gamma \in  \pi_{1}(M)$,  satisfy   
$$  D \circ \gamma =   h(\gamma)  \circ D . $$
We endow the space of development maps with the compact $C^{\infty}$-topology. \index{compact $C^{\infty}$-topology}
\index{topology!compact $C^{\infty}$-}
In this topology, a sequence of smooth maps converges if and only if  \emph{it and all its derivatives} (computed in  local coordinate charts) converge uniformly on the compact subsets of $\tilde M$. In particular, $\Dev(M)$ thus becomes a Hausdorff second countable topological space.


\subsubsection{Convergence of holonomy} \label{sect:convhol}
Let $M$ be compact. Then $\pi_{1}(M)$ is finitely generated, and we equip $\Hom(\pi_1(M),G)$ with 
the topology of pointwise convergence.  \index{holonomy!convergence of}
Then the map $$\mathsf{hol}: \Dev(M) \ra  \Hom(\pi_1(M),G) \;  , \; D \mapsto h$$ 
is continuous, since $G$ has the
$C^\infty$-topology of maps on $X$.  The main theorem
on deformations of $(X,G)$-structures  
(see Theorem \ref{thm:Deformations} below) 
asserts that a small deformation of holonomy 
induces a deformation of development
maps. That is, the map $\mathsf{hol}$ admits local sections.
By compactness of $M$, the convergence of
development maps is controlled 
on a fundamental domain and by the holonomy. 
\begin{fact}[Holonomy determines convergence] \label{fact:convergence}
Let $U \subset \tilde{M}$ be 
an open subset with compact closure such that 
$\mathsf{p}(U) = M$, where $\mathsf{p}: \tilde M \ra M$ is the universal covering. 
Then a sequence  of development maps $D_{i} \in \Dev(M)$ with 
holonomy $h_{i}$ converges to a development map $D$ 
if and only if the restrictions of $D_{i}$ to $U$ converge to 
$D$ and the homomorphisms $h_{i}$ converge
to the holonomy $h$ of $D$. 
\end{fact}

A particular property of the $C^{\infty}$-topology is that it does not control the behavior  of maps outside compact sets. This allows
for possibly unexpected phenomena:  
\begin{example}[{\textbf{Openness of embeddings fails}}] \label{ex:embeddings}
Let $\Dev_{e}(M)$ be the subset of development maps which are
injective. Let $\cK \subset \tilde M$ be a compact fundamental
domain for the action of $\pi_{1}(M)$.  By \cite[Chapter 2, Lemma 1.3]{Hirsch}, the set of development maps which  are injective on $\cK$ is open with respect to the $C^1$-topology. In particular, it is open with respect to the $C^\infty$-topology. However, the global behavior of development maps is controlled by the holonomy. Therefore, even if  $M$ is compact $\Dev_{e}(M)$ may not be an open subset in $\Dev(M)$.
On the two-torus there are injective development maps in $\Dev(T^2, \bbA^2, \Aff(2))$ which contain a non-trivial covering map in every small neighborhood. This is even true
for the development of the 
standard translation structure, cf.\  \S \ref{sect:orbit_closures} 
and Figure \ref{figure:qHopfTo trans}.
\end{example}
\begin{figure}
\begin{center}
\includegraphics[scale=0.5, clip=true, trim=  0cm 0.35cm 0cm 0.7cm ]{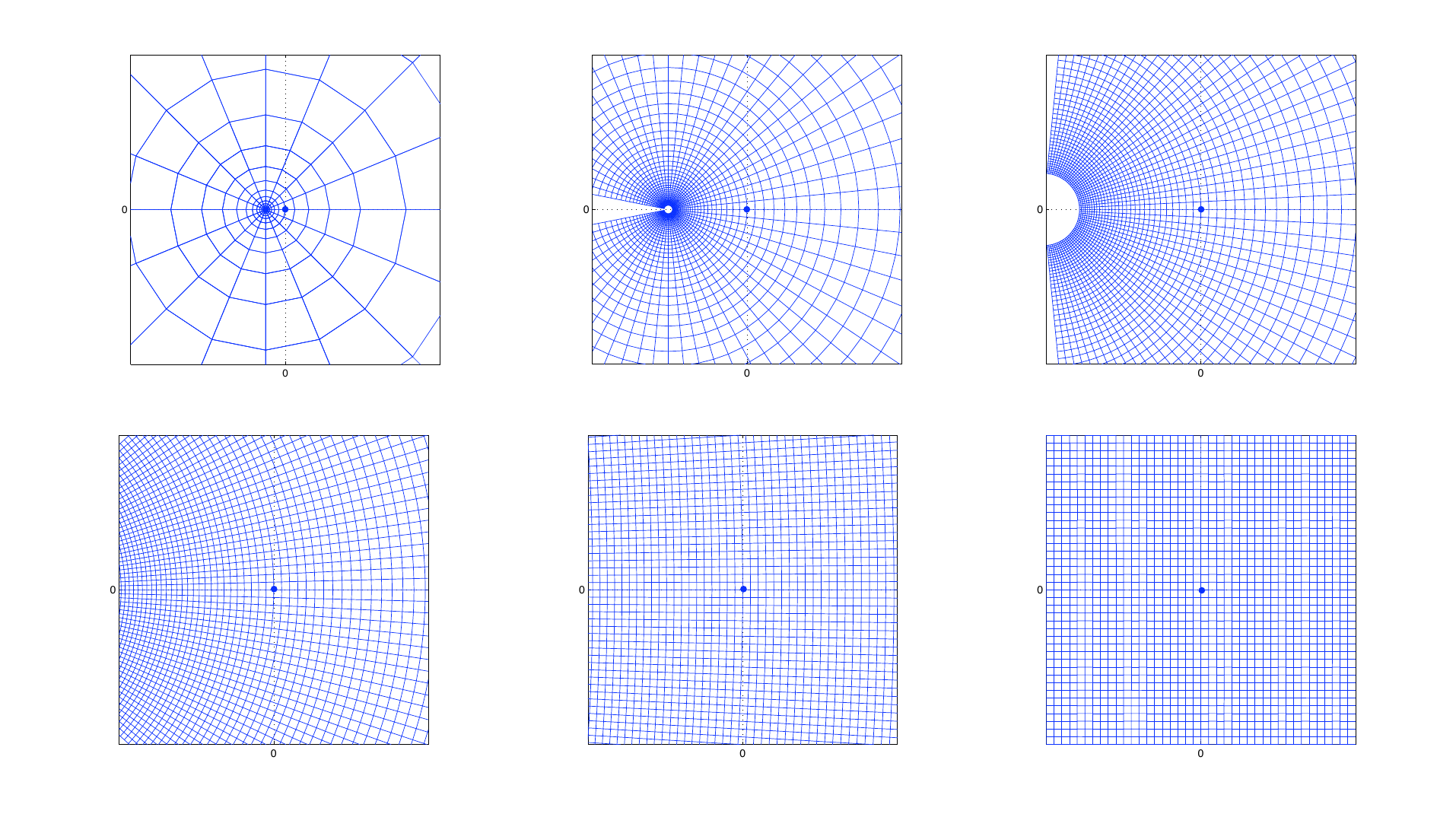}
\caption{The quotient of a Hopf torus deforms to a translation torus.} 
\label{figure:qHopfTo trans}
\end{center}
\end{figure}
\subsubsection{Deformation 
of development maps}
If $M$ is compact then, as observed by Thurston \cite{Thurston1}, 
building on earlier work of Weil \cite{Weil}, a small deformation of holonomy in the space of homomorphisms 
$\Hom(\pi_1(M) , G)$  induces a deformation of $(X,G)$-development maps. 
Before stating the theorem precisely, we discuss the   \index{development maps!deformation of}

\paragraph{Action of
diffeomorphisms of $M$  on development pairs.} \index{development pair}
Let $x_{0} \in M$ and $\tilde x_{0} \in \tilde M$, $\mathsf{p}(\tilde x_{0}) = x_{0}$, be basepoints and $\Phi \in \Diff(M,x_{0})$ a basepoint preserving diffeomorphism  
with lift  $\tilde{\Phi} \in \Diff(\tilde{M}, \tilde x_{0})$. 
The group $\Diff(M,x_{0})$ of basepoint preserving diffeomorphisms then acts on
development pairs, by mapping $D \in \Dev(M)$ to
$D \circ \tilde {\Phi}$.  
We let $\Diff_{1}(M,x_{0})$ denote the subgroup of $\Diff(M,x_{0})$ consisting of diffeomorphisms which
are homotopic to the identity by a basepoint preserving
homotopy,  and $\Diff_{0}(M,x_{0})$ the identity component of
$\Diff(M,x_{0})$ (that is, the subgroup of elements 
which are isotopic to the identity).
Then the action of $\Diff_{1}(M,x_{0})$ and its subgroup 
$\Diff_{0}(M,x_{0})$ on the set of development
maps $\Dev(M)$ leaves the holonomy invariant, since
$\Diff_{1}(M,x_{0})$ acts trivially on $\pi_{1}(M,x_{0})$. 
\\

See \cite{Epstein,Goldman,Lok,BeGe} for more detailed discussion 
of the following: \index{deformation theorem}
 \index{Thurston deformation theorem}

\begin{theorem}[Deformation theorem, Thurston et al.] \label{thm:Deformations}
Let $M$ be a compact manifold. Then the induced map 
\begin{equation}  \mathsf{hol}:  
\; \Diff_{0}(M,x_{0}) \backslash \,\Dev(M) \longrightarrow \Hom(\pi_1(M) , G)  \label{sfholmap} \end{equation}
which associates to a development map its holonomy homomorphism is a local homeomorphism. 
\end{theorem} 
The theorem states that the map $\mathsf{hol}:  \Dev(M) \longrightarrow \Hom(\pi_1(M) , G)$
is continuous and open. 
In addition, 
$\mathsf{hol}$ locally admits  \emph{continuous} sections.
Such a section is called a {\em development section}. 
More specifically, it is proved (see below) that 
\emph{every convergent sequence of holonomy maps lifts to a convergent sequence of
development maps}, and two nearby development maps with identical holonomy are isotopic by a basepoint preserving diffeomorphism. 
Therefore, a sequence of points  in the quotient space $\Diff_{0}(M,x_{0}) \backslash \,\Dev(M)$ 
is convergent if and only if there exists a corresponding lifted
sequence of development maps which converges. \\

The main idea in the proof of Theorem \ref{thm:Deformations} due to Weil \cite{Weil} is easy to grasp. Here we sketch the construction of the development section in the particular case of flat affine two-tori. In addition, we consider only tori which are obtained by gluing polygons in the 
plane (cf.\ \S \ref{sect:agluing}). 
A similar approach is also 
valid for non-homogeneous tori which 
are obtained as quotients of the universal covering affine 
manifold of $\bbAo$ 
(cf.\ \S \ref{sect:tori_bbao}), and, 
in fact, in the general case of arbitrary $(X,G)$-manifolds, compare \cite{Lok,Weil}. A somewhat different approach 
to this result is explained in \cite{Goldman}
and the recent survey \cite{Goldman4} on locally homogeneous manifolds. 

\begin{proof}
[Proof of Theorem \ref{thm:Deformations}]
Let $M$ be a flat affine two-torus and $D_{o}: \bbR^2 \ra \bbA^2$,
$h_{o}: \bbZ^2 \ra \Aff(2)$ a development pair for $M$. Let $h_{\epsilon}
\in \Hom(\bbZ^2, \Aff(2))$, $\epsilon \geq 0$, be a small deformation of $h_{o}$. To
obtain the development section, we construct a curve of development maps
$D_{\epsilon}$ with holonomy $h_{\epsilon}$, which converges to $D_{o}$
in the compact $C^\infty$-topology.    

For this, we assume that the development pair of $M$ is represented 
as the identification space of a polygon $\mathcal P$ in affine space.
In fact, as explained in \cite[\S 2]{BauesG},  $\mathcal{P}$ can be chosen to be a quadrilateral contained in $\bbA^2$, 
which  is glued along its sides by the generators $\gamma_{1}$, $\gamma_{2}$   of $\pi_{1}(T^2) =  \bbZ^2$ using the holonomy images $h_{o}(\gamma_{i}) \in \Aff(2)$.
The generators satisfy cycle relations 
and certain gluing conditions.
Next we fix a diffeomorphism of the standard unit square in $\bbR^2$
with $\mathcal{P}$. Using $h$, this extends $\pi_{1}(T^2)$-equivariantly
to a smooth covering 
$$ \bbR^2 \; \ra \; \bar{X}= (P \times \Gamma) \big/ \sim \; \;  ,$$ 
where the identification space $\bar{X}$ is a flat affine manifold which is obtained as
the disjoint union of
the polygons $\gamma P$,  $\gamma \in \Gamma$, 
glued along their edges as determined by the
side pairings $h_{o}(\gamma_{i})$.  Here  $\Gamma = h_{o}(\bbZ^2)$
is the holonomy group of $M$.
Moreover, the inclusion $\mathcal P \rightarrow \bbA^2$
extends to a development map $$ \bar D: \bar X \ra \bbA^2 \; .$$ The composition
of both maps yields the desired 
development map $D: \bbR^2 \ra  \bbA^2$ with
holonomy $h_{o}$. The space $\bar X$ is the holonomy covering space of
$M$, see \cite[Proposition 2.1]{BauesG} for a detailed  account.  

The development section $D_{\epsilon}$  may now be obtained in a similar 
manner. In fact, for small $\epsilon>0$,  $\mathcal{P}$ 
can be deformed continuously to a quadrilateral $\mathcal{P}_{\epsilon}$, which
satisfies the gluing conditions with respect to $h_{\epsilon}$. (See Figure 
\ref{figure:deform1} for an illustration.) This gives
rise to a series of identification spaces 
$\bar X_{\epsilon} = (P \times \Gamma_{\epsilon})  / \sim_{\epsilon}$, and corresponding development maps $D_{\epsilon}:  \bbR^2 \ra  \bbA^2$ with
holonomy $h_{\epsilon}$. By the above  Fact \ref{fact:convergence}, the developments maps
$D_{\epsilon}$ converge to $D_{o}$ in the $C^\infty$-topology.   
\end{proof}

The above construction of the development section is illustrated in Figures \ref{figure:deform1} and \ref{figure:deform2}.  

\begin{figure}[htbp]  
\begin{center}

\includegraphics[scale=0.5, clip=true, trim=  0cm 0.35cm 0cm 0.7cm ]{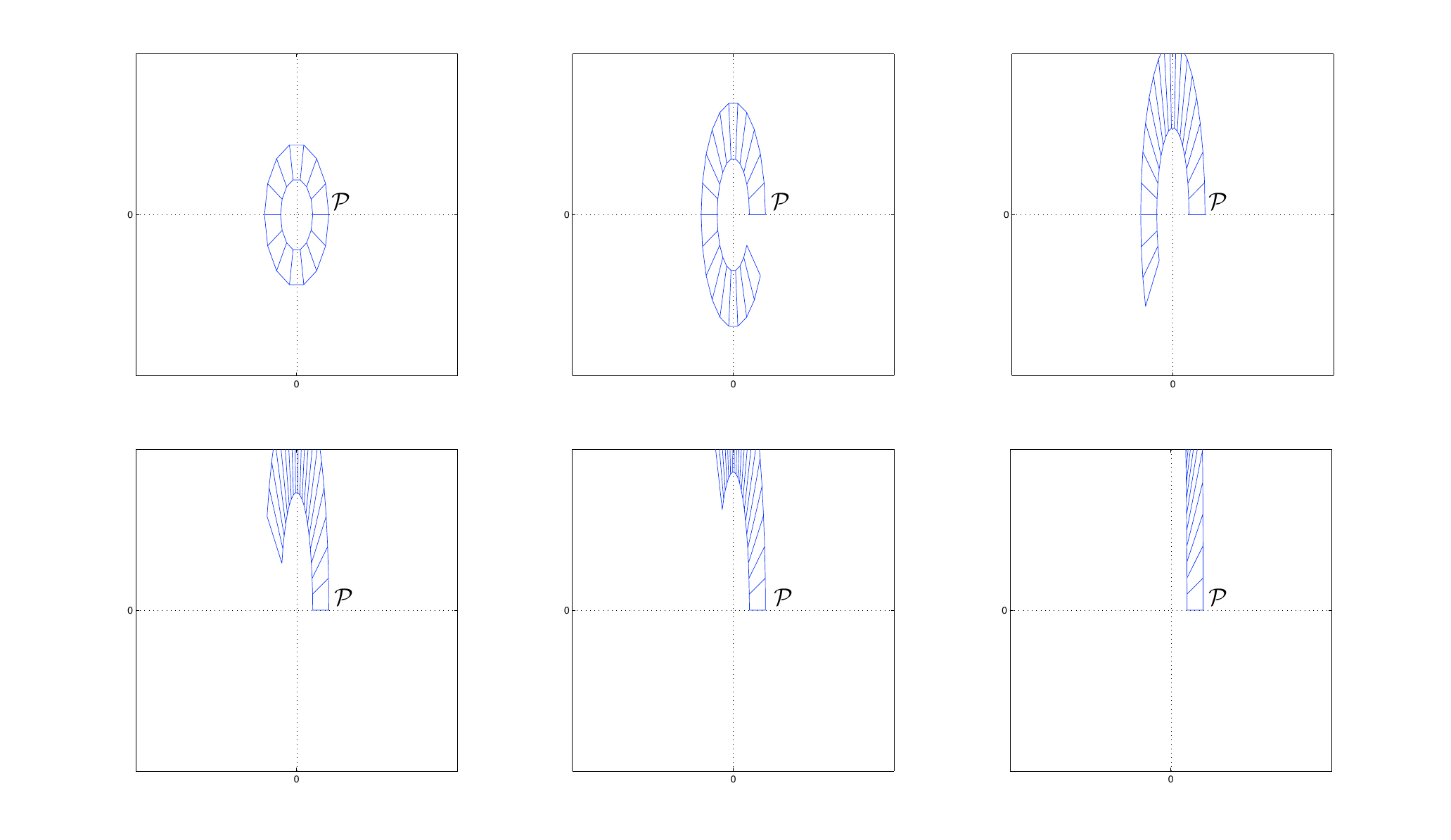}

\caption{The fundamental polygon $\mathcal{P}$ and its development deform with the holonomy.}
\label{figure:deform1}
\end{center}
\end{figure}

\begin{figure}[htbp]  
\begin{center}
\includegraphics[scale=0.5, clip=true, trim=  0cm 0.35cm 0cm 0.7cm ]{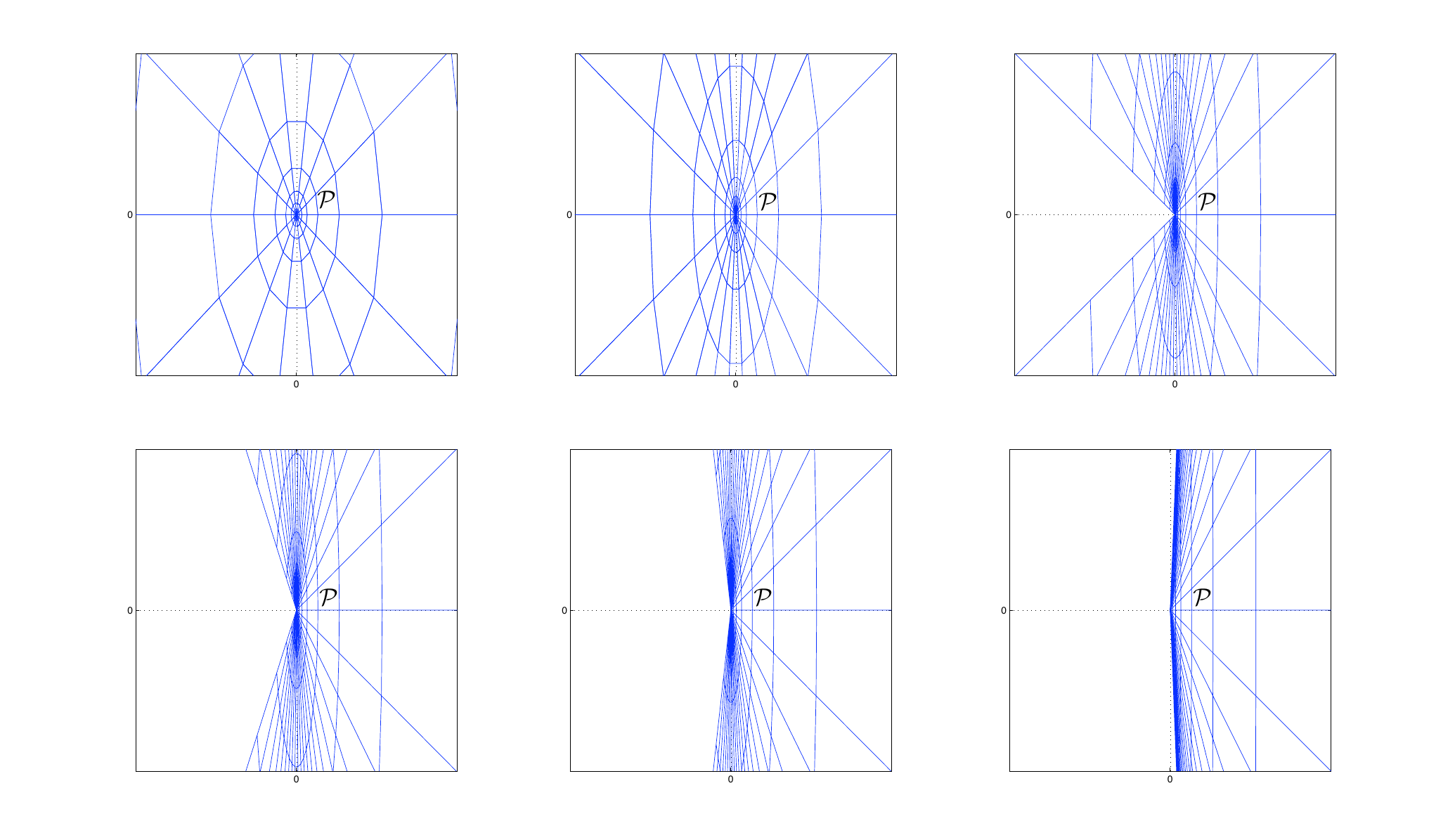}
\caption{A family of development maps for the 
 once-punctured plane $\bbAo$ collapses to 
 the development process of an affine half-plane.
}
\label{figure:deform2}
\end{center}
\end{figure}

\subsubsection{Topological rigidity of development maps}
\index{development maps!topological rigidity of}
\index{rigidity!of development maps}
Although local rigidity holds by the deformation theorem, 
it may fail globally. 
If the map $\mathsf{hol}$ in \eqref{sfholmap} is \emph{not} injective (as happens in the case of flat affine two-tori, see the basic Examples \ref{ex:finhopftori}, \ref{ex:gentorusnh} and also \S \ref{sect:thedefspace} for further discussion), there do 
exist non-isomorphic  $(X,G)$-manifolds with the same holonomy
homomorphism $h$.  On the contrary,  if the domain of discontinuity $\Omega$ 
for the holonomy group $h(\Gamma)$, $\Gamma = \pi_{1}(M)$, 
on $X$ is large then the development is uniquely determined. This is the case, for example, if $\Omega = X$ and $h(\Gamma)$ is the holonomy of a compact complete $(X,G)$-manifold.  

\begin{example}[\textbf{Discontinuous holonomy}] \label{ex:toprig}
\index{holonomy!discontinuous} \index{domain of discontinuity}
\index{holonomy homomorphism} 
Let $D: \tilde M  \ra X$ be the development map for an 
$(X,G)$-structure on the compact manifold $M$ with holonomy homomorphism 
$h$. If $h(\Gamma)$ acts 
properly discontinuously and freely with compact quotient 
on $X$ then
$D$ is a covering map onto $X$. (In fact, 
$D$ is a covering, since the local diffeomorphism
on compact manifolds 
$M \ra  X/h(\Gamma)$ induced by $D$ is a covering map.)
It follows that every 
other development map $D': \tilde M  \ra X$ with holonomy homomorphism $h$
is of the form $D' = D \circ \Phi$, where $\Phi \in \Diff(\tilde M)$
is a diffeomorphism which centralizes the deck transformation 
group $\Gamma$.
\end{example}

A more involved argument allows to show that  
$D$ is determined by $h(\Gamma)$ 
if the Hausdorff dimension of $X -\Omega$ is small, see \cite{GoKa}.  

\begin{example} The  development maps of compact $(\tbbAo, \widetilde \GL^+ \! (2,\bbR))$-forms are rigid,  see \S \ref{sect:bbaoclass}, Theorem \ref{thm:bbao_rigid}.
We remark that
the domain of discontinuity for the holonomy group of such a manifold can be  a proper open 
subset of $\tbbAo$.  
\end{example}

\subsection{Deformation spaces of $(X,G)$-structures}


Let $\Str(M) = \Str(M,X,G)$ denote the set of
all $(X,G)$-structures on $M$. The group $\Diff(M)$ of all diffeomorphisms of $M$ \index{diffeomorphism group!of manifold}
\index{manifold!diffeomorphism group of}
acts naturally on this set such that two $(X,G)$-structures are in the same orbit if and only if they are $(X,G)$-equivalent. The  set of all $(X,G)$-structures on $M$ up to $(X,G)$-equivalence is called the {\em moduli space} $\Mod(M) = \Mod(M,X, G)$ of $(X,G)$-structures. 

\begin{definition} \index{deformation space!of $(X,G)$-structures}
The {\em deformation space\/} for 
$(X,G)$-structures on $M$ is the quotient space 
$$ \Def(M) = \Def(M,X,G) =  \Str(M,X,G) / \Diff_{1}(M)$$
of equivalence classes of $(X,G)$-structures up to {\em homotopy\/}. 
\end{definition}
\noindent 
Thus, two $(X,G)$-structures define the same point 
in $\Def(M)$ if they are equivalent by an $(X,G)$-equivalence which is homotopic to the identity of $M$.  The moduli space 
$ \Mod(M)$ is the quotient space of the deformation space $\Def(M)$ by the group of homotopy  \index{moduli space!of $(X,G)$-structures}
classes of diffeomorphisms of $M$.

\begin{remark} There is some inconsistency in the literature about the definition
of the deformation space. Many authors define $\Def(M)$
to be the space of structures up to {\em isotopy}.  
If $M$ is a surface (two-dimensional manifold)  two homotopic diffeomorphisms are isotopic, by classical results of Dehn, Nielsen, and Baer  (see
for example \cite{Stillwell}). 
Therefore, in this case, these two definitions coincide.
The corresponding fact fails in higher dimensions, even for tori,
see \cite{HS}.
\end{remark}   

We observe that the Lie group $G$ acts by left-composition
on the space of development maps. 
This action is continuous  \emph{and  free}, and the set of 
 $(X,G)$-structures  naturally identifies with the quotient by the action of $G$, that is,  
 $$ \Str(M,X,G) =  G \,  \backslash \Dev(M,X,G)  \; . $$ 
Indeed, if $g \in G$ and $D \in \Dev(M,X,G)$ is a development map then $g \circ D$ is another development map for the same 
$(X,G)$-structure on $M$. This exhibits the deformation space
as a double quotient space 
$$ \Def(M,X,G) =   G \,  \backslash \Dev(M,X,G)  / \,  \Diff_{1}(M) \; . $$
The $C^\infty$-topology on the set of $(X,G)$-structures is the quotient topology inherited from $\Dev(M,X,G)$. 
(Thurston \cite[Chapter 5]{Thurston} also gives a direct
description of the topology on $\Str(M)$ in terms of 
convergence of sets of local charts which define 
the elements of $\Str(M)$, see \cite[1.5.1]{Epstein}.) 
The deformation space and the moduli space carry   
the quotient topology inherited from the 
set of $(X,G)$-structures.

\subsubsection{Orientation components of the deformation space}
Let $X$ be a $G$-space which is orientable. We let $G^+$ denote the normal subgroup
of orientation preserving elements of $G$. 
Now assume that $M$ is an $(X,G)$-manifold which is orientable. We fix an orientation for $M$.
Then there is a  disjoint decomposition 
\begin{equation} \label{eq:devdec}
\Dev(M,X,G) = \Dev^+(M,X,G) \cup  \Dev^-(M,X,G) \, ,
\end{equation}
where $\Dev^+(M,X,G)$ and $\Dev^-(M,X,G)$ denote the
closed (and open) subspaces 
which consist of orientation preserving and of orientation reversing, development maps, respectively.

Since $M$ is orientable, the components of the decomposition \eqref{eq:devdec} are preserved by the action of $\Diff_{1}(M,x_{0})$ on development maps. Furthermore, the action of
$G^+$ on development maps preserves the components.
%
Therefore, the deformation space $\Def(M,X,G^+)$ decomposes into two disjoint open and closed subsets, the \emph{orientation components}, \index{deformation space!orientation components of}
\begin{equation} \label{eq:defdec}
\Def(M,X,G^+) = \Def^+(M,X,G) \cup  \Def^-(M,X,G) \; .
\end{equation}
Note that every orientation reversing element of $G$ exchanges the orientation components of  $\Dev(M,X,G)$ and therefore also 
of $\Def(M,X,G^+)$. Hence, if
$G$ contains orientation reversing elements 
then the subgeometry $(X,G^+) \ra (X,G)$ induces a homeomorphism $$ \Def^+(M,X,G)  \approx  \Def(M,X,G)  \; . $$

\subsubsection{The topology of the deformation space}
The following classical and fundamental example gives a role model for the investigation of the properties of deformation
spaces for locally homogeneous structures. 
\begin{example}[Teichm\"uller space $\Teich_{g}$] \index{Teichm\"uller space}
Let $G^+= \PSL(2,\bbR)$ be the group of orientation preserving isometries of the 
hyperbolic plane $\Hyp_{2}$ and $M = M_{g}$ a surface of
genus $g$, $g \geq 2$. 
By the uniformization theorem, the Teichm\"uller space $\Teich_{g}$ of conformal structures on a surface $M_{g}$, $g \geq 2$,
may be considered as the deformation space of constant curvature $-1$ metrics, that is, $$\Teich_{g} = \Def^+(M_{g}, \Hyp_{2}, \PSL(2,\bbR)) \, . $$ 
The space $\Teich_{g}$ is homeomorphic to $\bbR^{6g-6}$.
Recall that the 
\emph{mapping class group} \index{mapping class group}
$$ \Map_{g}= \Diff^+(M_{g})/ \Diff_{0}(M_{g}) \cong \Out^+(\Gamma_{g}) $$  
is the group of isotopy classes of orientation 
preserving diffeomorphisms of a surface. 
This group 
acts properly discontinuously on $\Teich_{g}$, and the 
moduli space of conformal structures $$ \Mod(M_{g}) = 
\Teich_{g}/ \, \Map_{g} $$ is a Hausdorff space. 
(See, for example, \cite{Abikoff,EE,Ratcliffe} and 
other chapters of this handbook \cite{BIW, Goldman3}).
\end{example}

In general, however, the topology on the moduli space and the deformation space can be highly singular, as we can see, in particular,  from Examples \ref{ex:completeas} and  \ref{ex:defissing} below. The local properties of the 
deformation space are reflected in the character variety 
$\Hom(\pi_1(M) , G)/G$, which is the space of 
conjugacy classes of representations of 
$\pi_1(M)$ into $G$. \index{character variety}

\paragraph{The holonomy map on the deformation space}
Since  $\Str(M,X,G)$ has the quotient topology from development maps,  the  holonomy \eqref{sfholmap} 
induces a continuous map 
\begin{equation*} \overline{\mathsf{hol}}:  \Str(M,X,G) \longrightarrow \Hom(\pi_1(M) , G)/G  \; ,   \label{sfholmap2}
\end{equation*} which 
gives rise  to the map 
\begin{equation} \label{holmap} 
 hol:  \Def(M) \longrightarrow \Hom(\pi_1(M) , G)/G  \; .
\end{equation} 
The continuous map $hol$ associates to a homotopy class of $(X,G)$-structures on $M$ the corresponding conjugacy class of its holonomy homomorphism $h$. By the deformation theorem  (Theorem \ref{thm:Deformations}), 
$hol$ is furthermore an open map. \\

The map $hol$ thus encodes  
a good picture of the topology on $\Def(M)$:  

\begin{example}[Teichm\"uller space $\Teich_{g}$ is a cell] 
\label{ex:Teichm}  \index{Teichm\"uller space} \index{surface!hyperbolic}
The holonomy image of hyperbolic structures on 
$M_{g}$,  $g \geq 2$,  
is the subspace $\Hom_{c}(\Gamma_{g}, \PSL(2,\bbR))$ 
of the space $\Hom(\Gamma_{g}, \PSL(2,\bbR))$ 
which consists of injective homomorphisms with
discrete image. The space $\Hom_{c}(\Gamma_{g}, \PSL(2,\bbR))$ has two connected components \cite{Go88}. The  
components $\Hom_{c}^+(\Gamma_{g}, \PSL(2,\bbR))$ and
 $\Hom_{c}^-(\Gamma_{g}, \PSL(2,\bbR))$ arise
from the orientation of development maps.  
The group $\PSL(2,\bbR)$ acts freely and properly 
(by conjugation) on $$ \Hom_{c}^+(\Gamma_{g}, \PSL(2,\bbR)) \; ,$$ the quotient
space being homeomorphic to $\bbR^{6g-6}$. 
(See, for example,  \cite[Theorem 9.7.4]{Ratcliffe}). 
By completeness of hyperbolic structures on $M_{g}$, every
development map is a diffeomorphism. The topological 
rigidity of development maps (cf.\ Example \ref{ex:toprig})
implies that the induced map 
$$\Teich_{g}= \Def^+(M_{g}, \Hyp_{2}) \, \;  \stackrel{hol}{\longrightarrow} \, \; \Hom_{c}^+(\Gamma_{g}, \PSL(2,\bbR))/  \PSL(2,\bbR)) $$
is a homeomorphism. 
\end{example}

In general, it seems difficult to decide if the map 
\index{holonomy map!is a local homeomorphism}
$hol$ is a local homeomorphism, as well. Indeed, Kapovich
\cite{Kapovich} constructed examples of deformation spaces such 
that the map $hol$ is \emph{not} everywhere a local homeomorphism. 
We construct such a counterexample  for the deformation space of 
a two-dimensional geometric structure on the two-torus in Appendix B. \index{holonomy map!is not a local homeomorphism}
In the case of flat affine two-tori though, we shall show that $hol$ is a local homeomorphism (see \S \ref{sect:holislh}). 
\index{holonomy map!for flat affine structures}

\paragraph{The induced map of a subgeometry}
Let $o: (X',G') \ra (X,G)$ be a subgeometry with $\rho: G' \ra G$ the associated homomorphism  (see  
\S \ref{sect:sub_geom}). There is an associated  map 
\begin{equation} \label{eq:subgeometry0}
\Dev(M, X') \, \lra \, \Dev(M,X) \, , \; D' \mapsto D = o \circ D'
\end{equation}
and a map on homomorphisms 
$$  \Hom(\pi_{1}(M), G')  \, \lra \,  \Hom(\pi_{1}(M), G) \, , \; h' \mapsto h = \rho \circ h'  \; , $$
where $h = \mathsf{hol}(D)$.
These maps allow to relate 
the deformation spaces in  a commutative diagram 
of the form  
\begin{align} \label{eq:subgeometry1}
 \xymatrix{
\; \; \Def(M, X') \; \; \ar[d] \ar[r]^(0.37){hol}   & \; \; \Hom(\pi_{1}(M), G')/  G'   \; \; \ar[d]  \\
\; \; \Def(M, X) \; \; \ar[r]^(0.37){hol}  &  \; \; 
\Hom(\pi_{1}(M), G)/  G  \; \; . }    
\end{align}

Note that the properties of the induced map $\Def(M, X') \ra \Def(M,X)$ can vary wildy with various types of subgeometries. 
In general, the induced map need not be injective nor surjective.\\

Recall the notion of covering of geometries from Definition
\ref{def:subgeometry}. We shall require the following lemma:
\begin{lemma} \label{lem:covering_geoms}
If $o: (X',G') \ra (X,G)$ is a covering of geometries then 
the induced map on deformation spaces 
$$ \Def(M, X') \lra \Def(M,X) $$
is a homeomorphism.
\end{lemma}
\begin{proof} Indeed, since $o$ is a covering the above map \eqref{eq:subgeometry0}, $D' \mapsto D$,  
on development maps 
descends to a $\Diff(M)$-equivariant map on the sets of structures
$$ \Str(M,X') \lra \Str(M,X) $$ 
which is a homeomorphism. 
\end{proof}

\subsubsection{The topology on the space of $(X,G)$-structures}
\index{$(X,G)$-structures!space of}
The topology on the space $\Str(M,X,G)$ is rather well behaved. 
In fact,  $\Str(M,X,G)$ is  a Hausdorff and  metrizable   
topological space. This can be seen by representing an 
$(X,G)$-structure on $M$ as an integrable higher order structure
in the sense of Ehresmann (cf.\ \cite[\S I.8]{SKobayashi}). 
We discuss two important examples now:

\begin{example}[$\Str(M,  \Hyp_{2}, \PSL(2,\bbR))$]
The space of hyperbolic structures on a surface $M$ is homeomorphic \index{$(X,G)$-structures!hyperbolic}
to the space of hyperbolic (constant curvature $-1$) Riemannian metrics
with the $C^\infty$-topology on the space of Riemannian metrics.
It can also be equipped with the structure of a contractible Fr\'echet manifold, see \cite{EE}. Similarly, the space $\Str(M, \bbR^2, \E(2))$ of flat Euclidean structures is homeomorphic to
the space of flat Riemannian metrics on $M$.
\end{example}

In the case of flat affine structures, the action of the affine 
group on development maps admits a global slice: 
\index{$(X,G)$-structures!flat affine}

\begin{example}[$\Str(M,\bbA^n)$ is Hausdorff]  
\label{ex:framepreserving}
 Let $\Dev(M, \bbA^n)$ be the \index{development maps!space of}
set of development maps for flat affine structures on $M$. 
We choose a base frame $E_{x_{0}}$ on $\bbA^n$ and a 
frame $F_{\tilde m_{0}}$ 
on $\tilde M$, respectively, 
and let $$ \Dev_{f}(M,\bbA^n) = \Dev_{f}(M,F_{{\tilde m}_{0}}, E_{x_{0}})$$  denote the set of frame preserving
development maps. Since $\Aff(n)$ acts simply transitively on
the frame bundle of $\bbA^n$, there is a well defined continuous
retraction $\Dev(M, \bbA^n) \ra \Dev_{f}(M,\bbA^n)$, and, in fact, there is 
a homeomorphism $$ \Dev(M, \bbA^n) \approx \Aff(n) \times  \Dev_{f}(M,\bbA^n) \; .$$
This shows 
that the quotient  $\Dev(M, \bbA^n)/ \Aff(n)$ is homeomorphic to 
the subspace $\Dev_{f}(M, \bbA^n)$ and the affine group $\Aff(n)$ acts properly on the set
of development maps. In particular, the space of flat affine structures 
$\Str(M, \bbA^n)$ is a Hausdorff space. 
\end{example}
Another way to understand the topology on $\Str(M,\bbA^n)$
is to identify flat affine structures with 
\emph{flat torsion free
connections} on the tangent bundle of $M$. These form 
a space of sections of a quotient of the bundle of $2$-frames
over $M$, see
\cite[Proposition IV.7.1]{SKobayashi}. 
In \S \ref{sect:affine_conns} 
of this chapter we employ this approach to study flat affine structures on the two-torus.

\subsubsection{The subspace of complete $(X,G)$-structures}
\index{deformation space!of complete $(X,G)$-structures}
\index{$(X,G)$-structures!complete}
Let $\Def_{c}(M)$ denote the subset of the deformation space
$\Def(M)$ which consists of complete $(X,G)$-space forms 
(that is, the subspace corresponding to development maps which are diffeomorphisms).   
We denote with $\Hom_{c}(\pi_1(M), G)$ 
the set of all injective homomorphisms $\pi_1(M) \ra G$, 
such that the image acts properly discontinuously on $X$.
We call $\Hom_{c}(\pi_1(M), G)$ the \emph{set of discontinuous homomorphisms}. 
The holonomy homomorphisms  belonging to
the elements of $\Def_{c}(M)$ form an open subset 
of $\Hom_{c}(\pi_1(M), G)$. 
In fact, by the rigidity of development maps belonging to discontinuous 
holonomy homomorphisms (cf.\ Example \ref{ex:toprig}),
 a small deformation of holonomy, 
which remains in the domain of discontinuous homomorphisms,  lifts to a 
deformation of complete $(X,G)$-manifold structures on $M$. Therefore, Theorem
\ref{thm:Deformations} implies that
the restricted map 
\begin{equation*}  \mathsf{hol}:  
\; \Diff_{0}(M,x_{0}) \backslash \,\Dev_{c}(M) \longrightarrow \Hom_{c}(\pi_1(M) , G)  
\end{equation*}
is a local homeomorphism. 
Then the following result is easily observed (see also \cite{BauesV}):

 \begin{theorem} \label{thm:cdefspace}
 Let $M$ be a smooth compact manifold such that
 the natural homomorphism $\Diff(M)/\Diff_{1}(M) \ra \Out(\pi_{1}(M))$ is injective.  Then the induced map
 $$   hol:  {\Def}_{c}(M) \longrightarrow \Hom_{c}(\pi_1(M), G)/G  $$
 is a homeomorphism onto its image.
 \end{theorem} 

Note that the assumptions of the theorem are satisfied, for example, if $X$ is contractible.

\begin{example}[Complete flat affine structures on $T^2$]   \label{ex:completeas} \index{flat affine torus!complete} \index{flat affine manifold!complete} \index{$(X,G)$-structures!flat affine}
The holonomy image of development maps 
for complete flat affine structures on the two-torus 
is $\Hom_{c}(\bbZ^2, \Aff(2))$, that is, it
consists of all injective homomorphisms with
properly discontinuous image.
As is shown in \cite[\S 4.4]{BauesG}, 
this is a locally closed 
subset of $\Hom(\bbZ^2, \Aff(2))$, defined by algebraic
equalities and inequalities, and it has two connected
components.  Moreover, the conjugation action of the 
group $\Aff(2)$ on $\Hom_{c}(\bbZ^2, \Aff(2))$ is orbit
equivalent to its restriction to the subgroup $\GL(2,\bbR)$. 
The latter group acts freely and properly 
on $\Hom_{c}(\bbZ^2, \Aff(2))$ and the quotient
space is homeomorphic to $\bbR^2$.  Since 
$$ \Def_{c}(T^2, \bbA^2) \, \;  \stackrel{hol}{\longrightarrow} \, \; \Hom_{c}(\bbZ^2, \Aff(2))/\Aff(2)  $$
is a homeomorphism, the
deformation space of complete flat affine structures 
$\Def_{c}(T^2,  \bbA^2)$  is homeomorphic to $\bbR^2$. 
As is shown in \cite{BauesG,BG}, natural 
coordinates can be chosen such that the action of $\Map^{+}(T^2) = \SL(2,\bbZ)$
on $\Def_{c}(T^2, \bbA^2)$ corresponds to the standard representation of  $\SL(2,\bbZ)$
on $\bbR^2$. \index{deformation space!of complete flat affine structures}
\index{deformation space!of flat affine structures}
\end{example}

\subsubsection{Deformation of lattices (A. Weil, 1962)}  \label{sect:lattices} \index{lattice!deformation of} \index{deformation of lattice} \index{Weil A.}
Let $G$ be a simply connected Lie group and $\Gamma_{o} \leq G$ a cocompact lattice. We put $$ M_{o} = G/ \Gamma_{o} \; , $$ where $\Gamma_{o}$ acts by left-multiplication on the universal cover $\tilde M_{o} =G$ of $M_{o}$. Let $(X,G_{L}) = (G,G_{L
})$ be the homogeneous
geometry which is defined by the action of $G$ on
itself by left-\-mul\-tipli\-cation. Since the action of $G$ on itself is proper, every $(G,G_{L})$-manifold is complete. Hence 
$$  \Def(M_{o}, G) = \Def_{c}(M_{o}, G) \;  $$
and the holonomy image of $ \Def(M_{o}, G) $ is contained in 
the space of lattice homomorphisms
$$ \Hom_{L}(\Gamma_{o},G) = \{ \rho: \Gamma_{o} \hookrightarrow G \mid \rho(\Gamma_{o}) \text{ is a lattice in } G \} \; . $$
We call  the space of conjugacy classes of lattice homomorphisms
$$  \Def_{L}(\Gamma_{o}, G) =  \Hom_{L}(\Gamma_{o},G)/G $$
the deformation space of the lattice $\Gamma_{o}$. \index{deformation space!of lattice}¥
The holonomy map  
\begin{equation} \label{eq:hol_lat}  hol: \Def(M_{o}, G) \,  \longrightarrow \,  \Def_{L}(\Gamma_{o}, G)
\end{equation}
therefore locally embeds $\Def(M_{o}, G)$ as an open (and closed) subspace of the deformation space of $\Gamma_{o}$. 
This is the original setup which is studied in the seminal paper  \cite{Weil} by Andr\'e Weil. Fundamental results on the nature of the involved spaces  $\Hom_{L}(\Gamma_{o},G)$ and $\Def_{L}(\Gamma_{o}, G)$  are obtained in the foundational 
papers \cite{Weil, Weil2, Wang}, see also \cite{BeGe}. 
For a recent contribution in the context of solvable Lie groups $G$, see \cite{BK}; 
the examples which are constructed  in \cite[Section 2.3]{BK} show that there exist deformation spaces of the form
$\Def(M_{o}, G)$, which have infinitely many connected components.  

\paragraph{Rigidity of lattices and 
action of the automorphism group of $G$} \index{rigidity!of lattice} 
\index{lattice!rigidity of}
Note that the group $\Aut(G)$ of automorphisms of 
$G$ has natural actions on the space of development maps 
$\Dev(M_{o}, G)$ and on $\Hom_{L}(\Gamma_{o},G)$.
Indeed, let $\phi \in \Aut(G)$, and $D: G \ra G$ a  development map for an $(G,G_{L})$-structure on $M_{o}$ with
holonomy $\rho \in \Hom_{L}(\Gamma_{o},G)$. 
Then the composition $$ \phi \circ D: G \ra G $$ is a
development map with holonomy $\phi \circ \rho$. 
These actions descend to actions on $ \Def(M_{o}, G)$,
$\Def_{L}(\Gamma_{o}, G)$ respectively, such that
\eqref{eq:hol_lat} becomes an equivariant map.

\begin{example}[Rigid lattices] \index{lattice!rigid}
A lattice $\Gamma_{o}$ is called \emph{rigid} in $G$
if $\Aut(G)$ acts transitively on $\Def_{L}(\Gamma_{o}, G)$. 
For example, lattices in nilpotent Lie groups $G$, or 
lattices in simple Lie groups $G$ not locally isomorphic to
$\SL(2,\bbR)$ are rigid, see \cite{OnVi}. In these two cases
we then have identities
$$ \Def(M_{o}, G) \stackrel{\approx}{\lra}  \Def_{L}(\Gamma_{o}, G) \approx \Aut(G) / \Inn(G) \; ,
$$
for any lattice $\Gamma_{o} \leq G$. Here, $ \Inn(G)$ denotes the
group of inner automorphisms of $G$.
\end{example}

More generally, we call a lattice $\Gamma_{o}$ \emph{smoothly rigid}, if the holonomy map \eqref{eq:hol_lat} is a homeomorphism, \index{lattice!smoothly rigid}
that is, if $\Def(M_{o}, G) = \Def_{L}(\Gamma_{o}, G)$. For example, lattices in solvable Lie groups are smoothly rigid, by a theorem of Mostow; but there do exist solvable Lie groups which  admit non-rigid lattices. See \cite{OnVi}, or \cite{BK} and the references therein for specific examples. \\

The deformation spaces of the form $\Def(M_{o}, G)$ play an important role in the analysis of general deformation spaces,
since many geometric structures arise from \'etale representations.
An illustrative example is given by the stratification of 
\index{deformation space!stratification of}
the space of deformations of flat affine structure
on the two-torus which is studied in detail in 
\S  \ref{sect:def_homogeneous}. 

\paragraph{The induced map of an \'etale representation}
\index{etale@\'etale representation} \index{representation!\'etale}  
Let $G'$ be a simply connected Lie group and $\Gamma_{o} \leq G'$ a cocompact lattice. We put $M_{o} = G'/\Gamma_{o}$. 
Let us assume for simplicity that $\Gamma_{o}$ is smoothly 
rigid as well. 
Now let $(X,G)$ be a homogeneous space and $\rho:G' \ra G$ be an \'etale representation (see Definition \ref{def:etale}). Then the orbit map 
which is associated to an
open orbit of $G'$ defines a subgeometry 
$$ o: (G',G'_{L}) \ra (X,G)
\, $$
which in turn gives rise to a map \eqref{eq:subgeometry1} of deformation
spaces 
$$   \Def(M_{o}, G') \, \lra  \, \Def(M_{o},X,G) \; , $$
that is, we obtain a map 
$$  \Def_{L}(\Gamma_{o}, G') =  \Hom_{L}(\Gamma_{o},G')/G' \ra \Def(M_{o},X,G) \; . $$
This map factors over the action  
of the normalizer $\N_{G}(\rho)$ of $\rho(G')$ in $G$,
that is, we have an induced map 
$$ \Hom_{L}(\Gamma_{o},G')/ \, \N_{G}(\rho) \ra \Def(M_{o},X,G) \; . $$
We remark that, if $\Gamma_{o}$ is rigid in $G'$ then 
$$ \Hom_{L}(\Gamma_{o},G')/ \, \N_{G}(\rho) = \Aut(G') /N \; ,$$
where $N \leq \Aut(G')$ denotes the image of 
$\N_{G}(\rho)$ in $\Aut(G')$. 

%

\subsubsection{Dynamics of the $G$-action on $\Hom(\Gamma,G)$}
In Examples \ref{ex:Teichm} and \ref{ex:completeas} above, the map $hol$ is a homeomorphism,  and the corresponding deformation spaces are Hausdorff.  These properties hold in particular  if the holonomy image in  $\Hom(\Gamma,G)/G$ is obtained as a quotient by a \emph{proper} group action. In fact, if $G$ acts properly (and freely) on the image of $\mathsf{hol}$, then, \index{slice theorem}
 by the slice theorem (cf.\ \cite{Palais}), the projection map 
$\Hom(\Gamma,G) \ra \Hom(\Gamma,G)/G$ admits a section near every
holonomy homomorphism. It then follows from Theorem \ref{thm:Deformations} that 
$hol: \Def(M) \ra \Hom(\Gamma,G)/G$ is a local homeomorphism. 

\begin{example}[Subvariety of stable points]
If $G$ is a reductive \index{stable point of reductive group action}
linear algebraic group, then, by a general fact on representations 
of such groups, there exists a Zariski-open subset of  
stable points in $\Hom(\Gamma,G)$, where $G$ acts properly. 
Recall that, for any representation of $G$ on a vector space, or 
any action of $G$ on an affine variety $V$, a point $x \in V$ is called \emph{stable}
if the orbit $Gx$ is closed and $\dim G x = \dim G$.
The set of stable points may be empty though. For the
action of $G$ on $\Hom(\Gamma, G)$ it is non-empty if there are 
points $\rho \in \Hom(\Gamma,G)$ such that $\rho(\Gamma)$ 
is sufficiently dense in $G$. In the specific context  where $\Gamma$ is abelian 
(or solvable),  $\Hom(\Gamma,G)$ has no stable points
(as follows from \cite[Theorem 1.1]{Millson}).
See \cite{Millson, Goldman} 
for further discussion of these facts and 
for some applications. 
\end{example}

One cannot expect $\Def(M)$ to be a Hausdorff space, in general. 
In fact, the image of $hol$ in $\Hom(\pi_1(M) , G)/G$ may 
contain non-closed points. In this situation 
also $\Def(M)$ has non-closed points. The following example 
is due to Bill Goldman: 
\index{deformation space!non-closed point of}

\begin{example}[Non-closed points in $\Def(T^2, \bbA^2)$]
\label{ex:defissing}
Let
$$ A_{\epsilon} = \begin{matrix}{cc} \lambda & \epsilon \\
0 & \lambda \end{matrix} \; , \; \text{where $\lambda>1$.} $$
Then 
$M_{\epsilon} =  \langle A_{\epsilon} \rangle \lmod \, \bbAo$
is  a flat affine two-torus,
which has an
infinite cyclic holonomy group generated by 
$A_{\epsilon}$ (see also Example \ref{ex:expholonomy}).
Let $\rho_{\epsilon}$ denote a corresponding 
holonomy homomorphism for $M_{\epsilon}$. 
Since the $A_{\epsilon}$,  $\epsilon \neq 0$,  
are all conjugate elements of $\GL(2,\bbR)$, 
the closure of the $\GL(2,\bbR)$-orbit of $\rho_{1} \in 
\Hom(\bbZ^2,\GL(2,\bbR))$  
contains the holonomy homomorphism $\rho_{0}$.
Therefore, the orbit $[\rho_{1}]$ is not closed in 
 $\Hom(\bbZ^2,\Aff(2))/\Aff(2)$. By Corollary \ref{cor:bbaotop},
 $M_{1}$ defines a non-closed point in the deformation space. 
 \end{example}

 Observe that $\rho_{0}$ is the holonomy of the Hopf torus ${\mathcal H}_{\lambda}$. \index{Hopf torus}
 By Theorem \ref{thm:Deformations} there exists   
 a corresponding family of development maps 
 $D_{\epsilon}$  
 with holonomy $\rho_{\epsilon}$  
 which converges to the development
 map of the Hopf torus $M_{0}= {\mathcal H}_{\lambda}$.
 We observe that these development maps belong to 
 affine structures which are isotopically equivalent to the tori $M_{\epsilon}$. Hence, the closure of $M_{1}$ in
 the deformation space contains the Hopf torus
 ${\mathcal H}_{\lambda}$.

(To see explicitly how the development maps for the 
tori $M_{\epsilon}$ converge to
the Hopf torus in the deformation space,  we may use  
the constructions in \S \ref{sect:tori_bbao} in this chapter.
In fact, we construct  
 $M_{\epsilon}$  as 
 a quotient space $M_{\epsilon} = \cT_{A_{\epsilon}, \id, 2}$ of $\tbbAo$, as in  Example 
 \ref{ex:gentorusnh}. Then we
deform the development $D= D_{0}$ of $M_{0}$ 
as in the proof of Theorem \ref{thm:Deformations} to
obtain a sequence of development maps 
 $D_{\epsilon}: \, \tbbAo  \, \rightarrow \, \bbA^2$
 for $\cT_{A_{\epsilon}, \id, 2}$ which converges to $D_{0}$.) 
  
 \subsubsection{Dynamics of the $\Diff_{0}(M)$-action on $(X,G)$-structures}
 In favorable cases, the topology on $\Def(M)$ may be determined by constructing
 slices for the action of $\Diff_{0}(M)$ on  $\Str(M,X,G)$. The study of the action of 
 $\Diff_{0}(M)$ on $\Str(M,X,G)$ may then be used to deduce information on the
topology (diffeomorphism groups carry the 
$C^\infty$-topology) of $\Diff_{0}(M)$, or, vice versa, on the 
topology of $\Str(M,X,G)$. The theory of slices for action of diffeomorphism 
groups on spaces of Riemannian metrics was developed by Palais and
Ebin \cite{Ebin}. \index{slices for group of diffeomorphisms} 
\index{diffeomorphism group!of manifold} \index{manifold!diffeomorphism group of}
Recall that a continuous action of $\Diff(M)$ on a space $\mathcal S$ is called
proper if the map $\Diff(M) \times \mathcal{S} \ra \mathcal{S} \times \mathcal{S}$, 
 $(g,s) \mapsto ( g \cdot s, s)$ is proper. If the action is proper, the quotient
 space is Hausdorff (cf.\  \cite[III, \S 4.2]{BourbakiTop}).
 \index{proper group action}
 \index{diffeomorphism group!proper action of}
 
\begin{example}[$\Diff(M_{g})$ acts properly] 
The group of diffeomorphisms of a closed surface $\Diff(M_{g})$   
acts properly on the space of conformal structures \index{conformal structure}
$\Str(M_{g}, \Hyp_{2}, \PSL(2,\bbR))$, if $g \ge 1$. 
\index{diffeomorphism group!of surface} \index{surface!diffeomorphism group of}
In particular, the identity component 
$\Diff_{0}(M_{g})$ acts properly and freely.
Moreover, the projection map $$\Str(M_{g}, \Hyp_{2}, \PSL(2,\bbR)) \ra \Teich_{g}$$  is a trivial 
$\Diff_{0}(M_{g})$-principal bundle. Since
$\Str(M_{g}, \Hyp_{2}, \PSL(2,\bbR))$ and 
$\Teich_{g}$ are contractible, this implies at once 
that \emph{the group $\Diff_{0}(M)$
is contractible}. These results were shown in \cite[\S5 D]{EE}. 
\index{diffeomorphism group!contractible}
\end{example}

Similar results hold also for the space $\Teich_{1}$ of conformal structures 
(flat Riemannian metrics) on the two-torus. In fact, $\Diff(T^2)$ acts 
properly on the space $\Str(T^2, \bbR^2, E(2))$,
and the moduli space of such structures is a Hausdorff space. 
\index{diffeomorphism group!of two-torus} \index{torus!diffeomorphism group of}
However, the action of 
$\Diff_{0}(T^2)$ is not free,
since every flat Riemannian
structure on $T^2$ has 
$S^1\times S^1$ acting as a group of isometries.
In this situation, we may replace
$\Diff(T^2)$ with the subgroup $\Diff(T^2,x_{0})$. 
Indeed, $S^1 \times S^1$ is a deformation
retract of $\Diff_{0}(T^2)$ and the group
$\Diff_{0}(T^2,x_{0})$ is contractible  (cf.\ \cite{EE}).\\

In general, the action of $\Diff_{1}(M)$ on a space of structures  $\Str(M,X,G)$ 
need not be free neither proper, as we show in the following examples.

\paragraph{Action of $\Diff(T^2)$ on the space of flat affine structures}
\index{diffeomorphism group!action on flat affine structures}
In the case of flat affine structures on the two-torus, the action 
of $\Diff_{0}(T^2)$ 
on the set of all affine structures
$\Str(T^2,\bbA^2)$ is not proper, for otherwise $\Def(T^2, \bbA^2)$ would be 
a Hausdorff space.  But, in fact, as we show in  
Example \ref{ex:defissing}, $\Def(T^2, \bbA^2)$ has singularities. \\

An interesting in-between  
case arises when restricting to the subspace $\Str_{c}(T^2,\bbA^2)$ of \emph{complete} flat  affine structures. This case bears some resemblance to the case of conformal structures, although here the action of $\Diff(T^2)$ on the set of structures $\Str_{c}(T^2,\bbA^2)$ is not proper. However,  the action of the subgroup $\Diff_{0}(T^2)$ on  $\Str_{c}(T^2,\bbA^2)$ \emph{is}  proper. 

\begin{example}[Action of $\Diff(T^2)$ on $\Str_{c}(T^2,\bbA^2)$]
Observe first that every complete flat affine structure on $T^2$
is homogeneous and the identity component of its automorphism 
group acts simply transitively. This follows from the classification
given in Theorem \ref{thm:classification}.
Therefore, like in the case of Euclidean structures, 
$\Diff_{0}(T^2, x_{0})$ acts freely on $\Str_{c}(T^2, \bbA^2)$ 
and $$   \Def_{c}(T^2, \bbA^2)  =  
\Str_{c}(T^2, \bbA^2) / \Diff_{0}(T^2, x_{0}) \; . $$ 

Since $\mathsf{hol}:  \Diff_{0}(M,x_{0}) \backslash \, \Dev_{c}(T^2, \bbA^2) \ra \Hom_{c}(\bbZ^2,\Aff(2))$ locally admits continuous equivariant sections (see the discussion before Theorem \ref{thm:cdefspace}), it follows that the map
$$ \Str_{c}(T^2, \bbA^2) \ra \Def_{c}(T^2, \bbA^2)$$ is a locally trivial principal bundle
for $\Diff_{0}(T^2, x_{0})$. (It is also a universal bundle, since $\Str_{c}(T^2, \bbA^2)$ is contractible, as we see in Proposition \ref{prop:ss_iscon} below.) This already implies
that $ \Diff_{0}(T^2, x_{0})$ acts properly on 
$\Str_{c}(T^2, \bbA^2)$. On the other hand,  $\Diff(T^2, x_{0})$ does \emph{not} act
properly, since the action of the (extended) mapping class group 
of the two-torus
$$ \Diff(T^2, x_{0}) /\Diff_{0}(T^2, x_{0}) \cong \GL(2,\bbZ)$$ 
on the deformation space $\Def_{c}(T^2, \bbA^2) = \bbR^2$ 
is not properly discontinuous. (See also Example \ref{ex:completeas}).
\end{example}

In the previous example a slightly stronger result holds.  
Indeed, by the proof of \cite[Corollary 4.9]{BauesG} the projection map 
$$ \Hom_{c}(\bbZ^2, \Aff(2)) \lra \Def_{c}(T^2, \bbA^2)$$
admits a global (continuous) section.
Since the space $\Def_{c}(T^2, \bbA^2) = \bbR^2$ is contractible, we may use the covering homotopy theorem to conclude that there exists a continuous section 
$s:  \Def_{c}(T^2, \bbA^2) \ra \Str(T^2, \bbA^2)$.
This shows that the 
above principal bundle is trivial. (For an explicit construction of such a section, refer to \S \ref{sect:devsection} of this chapter.) 

\index{deformation space!of complete flat affine structures}
A typical application is:
\begin{proposition} \label{prop:ss_iscon}
The space $\Str_{c}(T^2, \bbA^2)$ of complete flat affine structures on the two-torus is contractible. 
\end{proposition}
\begin{proof} The group $ \Diff_{0}(T^2, x_{0})$ acts freely on 
the set of complete flat affine structures.
By the above, invariant sections exists for this action of $ \Diff_{0}(T^2, x_{0})$. It follows that  there is a homeomorphism 
$$ \Str_{c}(T^2, \bbA^2) \approx  \Diff_{0}(T^2, x_{0}) \times \Def_{c}(T^2, \bbA^2) \; .$$
 In particular, since 
$\Diff_{0}(T^2,x_{0})$  and 
$ \Def_{c}(T^2, \bbA^2)$ are contractible, the space $\Str_{c}(T^2, \bbA^2)$ is contractible.
\end{proof}

See \S \ref{sect:affine_conns} for a description of $\Str(T^2, \bbA^2)$ and
$\Str_{c}(T^2, \bbA^2)$ as subsets of  the affine space of torsion free 
flat affine connections of $T^2$. 
%
%
%

\subsection{Spaces of marked structures} \label{sect:markedstructures}
\index{$(X,G)$-structures!marked}
Let $M_{0}$ be a fixed smooth manifold. 
A diffeomorphism $f:M_{0} \ra M$, where $M$ is an $(X,G)$-manifold is called a  \emph{marking} of $M$. Two marked  $(X,G)$-manifolds $(f,M)$ and $(f',M')$ are called equivalent if there exists 
an $(X,G)$-equivalence $g: M' \ra M$ such that
$g \circ f'$ is \emph{homotopic} to $f$. Let $\StrM(M_{0}, X,G)$
denote the set of classes of marked $(X,G)$-manifolds. 
By composing with $f$,  every
local $(X,G)$-chart for the marked manifold $(f,M)$ extends to
the development map of an $(X,G)$-structure on 
$M_{o}$. This correspondence descends to a bijection 
of $\StrM(M_{0}, X,G)$ with the deformation space 
$\Def(M_{0},X,G)$. 
We can thus topologize the space of classes 
of markings with the topology induced from the 
$C^\infty$-topology on development maps. 

\begin{example}[Teichm\"uller metric on $\Teich_{g}$]
\index{Teichm\"uller metric} \index{Teichm\"uller space}
Classically, Teichm\"uller space $\Teich_{g}$ is represented 
as a space of marked conformal structures on Riemann surfaces.
Let $S_{0}$ be a closed Riemann surface of genus $g$.
A marking of a Riemann surface $R$ is an orientation preserving 
quasi-conformal homeomorphism $f: S_{0} \ra R$. 
Two marked surfaces $(f,R)$ and $(f',R')$ are equivalent 
if there exists a biholomorphic map $h: R' \ra R$ such that
$h \circ f'$ is \emph{homotopic} to $f$. Teichm\"uller space 
 $\Teich_{g}$ is the set of classes of marked surfaces.  
 The infimum of dilatations $K(\ell)$, where $\ell: R \ra R'$ is
 a quasiconformal map homotopic to $f' \circ f$, defines
 the \emph{Teichm\"uller distance} of $(f,R)$ and
 $(f',R')$ in $\Teich_{g}$:
 $$ d_{\Teich}([f,R], [f',R']) = \inf \log K(\ell) \; . $$
With the metric topology induced by $ d_{\Teich}$, 
$\Teich_{g}$  is homeomorphic to 
the Fricke space 
$$ \mathfrak{F}(S_{0}) = \Hom_{c}(\Gamma_{g}, \PSL(2,\bbR))/  \PSL(2,\bbR) , $$
as defined in Example \ref{ex:Teichm}. See \cite{Abikoff} and
\cite{PT} in Vol.\ I of this handbook for
reference on this material. Therefore, the topology defined 
by Teichm\"uller's metric coincides with the topology on $\Teich_{g}$, which is defined by
the convergence of development maps for hyperbolic structures.
\end{example}

We may also consider various refined versions of 
classes of marked $(X,G) $-manifolds and corresponding
deformation spaces. 

\paragraph{$(X,G)$-manifolds with basepoint}
Fix a basepoint $x_{0} \in X$ and write $X= G/H$, 
where $H = G_{x_{0}}$.
The space $\StrM_{\mathsf{p}}(M_{0}, X,G)$ of 
\emph{basepointed} marked structures is defined 
as follows. Let  $\mathsf{p}: (\tilde M_{0}, \tilde m_{0}) \ra (M_{0}, m_{0})$ be a fixed universal cover. A
marking of $(M,m)$ is  a based diffeomorphism 
$f:(M_{0},m_{0}) \ra (M,m)$. Two marked basepointed
 $(X,G)$-manifolds are equivalent if  there exists an $(X,G)$-\-equivalence $g: (M',m') \ra (M,m)$ such that
$g \circ f'$ is  homotopic to $f$ by a basepoint preserving homotopy. 
 Let $ \Dev_{\mathsf{p}}(M_{0})$ be the
set of basepoint preserving development maps. For every 
based local 
$(X,G)$ chart $\varphi$, defined  near $m_{0}$, there exists a unique 
development map $D$ for $M_{0}$, which extends $\varphi \circ f \circ \mathsf{p}$ from a neighborhood  of $\tilde m_{0}$. This correspondence induces a homeomorphism 
$$ \StrM_{\mathsf{p}}(M_{0}, X,G) \, \stackrel{\approx}{\lra}  \, 
\Diff_{1}(M_{0}, m_{0}) \backslash \Dev_{\mathsf{p}}(M_{0})/H  \; . $$

\begin{example}[Homogeneous $(X,G)$-structures] 
\label{ex:hom_structures} \index{$(X,G)$-structures!homogeneous}
Note that the natural (forgetful) map 
$\StrM_{\mathsf{p}}(M, X,G) \ra  \StrM(M, X,G)$ is surjective, but it is usually not
injective. In fact, let $D$ be a development map.
Then for the classes $G \circ D \circ \Diff_{1}(M,m)$  and 
$G \circ D \circ \Diff_{1}(M)$ to
coincide it is necessary  that 
the group of $(X,G)$-equivalences acts transitively on $M$.
That is, $D$ is the development map of a homogeneous 
 $(X,G)$-structure. We let $\Dev_{h}(M,X,G)$ denote the set of development
maps of homogeneous $(X,G)$-structures.  
\end{example} 

Note further that, in general, 
$\Diff_{1}(M_{0},m_{0})$ is a proper subgroup of 
all basepoint preserving diffeomorphisms 
which are \emph{freely} homotopic to the identity. The difference is obtained by the natural action of $\pi_{1}(M_{0}, m_{0})$
on based homotopy classes of maps. 
However, if $\pi_{1}(M_{0}, m_{0})$ is abelian, 
the inclusion is an isomorphism.

\begin{example} \label{ex:def_homaff}  
\index{flat affine torus!homogeneous} 
\index{$(X,G)$-structures!flat affine}
Let $\Def_{h}(T^2,\bbA^2)$ be the deformation
space of homogeneous flat affine structures on the two-torus. Then
$$\Def_{h}(T^2,\bbA^2) = \Diff_{1}(T^2, x_{0}) \backslash 
\Dev_{h}(T^2)/\Aff(2) . $$ In particular, since every complete 
affine two-torus 
is homogeneous,  $\Def_{h}(T^2,\bbA^2)$ contains 
the subspace of complete flat affine structures
$$\Def_{c}(T^2,\bbA^2) = \Diff_{1}(T^2, x_{0}) \backslash 
\Dev_{c}(T^2)/\Aff(2)  \; .$$
\end{example}

\subsubsection{Framed $(X,G)$-manifolds} \label{sec:framedXG}
The holonomy of a marked 
$(X,G)$-\-mani\-fold is a $G$-conjugacy class 
of homomorphisms. 
To get rid of the dependence on the conjugacy class, 
one introduces framed structures. The holonomy theorem
implies that the deformation space of framed structures is
a locally compact Hausdorff space. We shall discuss only the particular simple case of $(\bbA^n,\Aff(n))$-manifolds. 

\begin{example} \label{ex:fat}
Let $m \in M$. 
A frame $\cF_{m}$ for a flat affine structure 
on $M$ is a choice of basis of the tangent space $T_{m} M$. 
The pair $(M,\cF_{m})$ is called a \emph{framed flat affine manifold}. 
Fix a frame $\cF_{m_{0}}$ for $M_{0}$, as well, and call a frame preserving diffeomorphism $(M_{0}, \cF_{m_{0}}) \ra ( M, \cF_{m})$ a marking of $( M, \cF_{m})$.
Two marked  framed flat affine manifolds $(f,M, \cF_{m})$ and 
$(f',M', \cF'_{m'})$ are called equivalent if there exists 
a frame preserving affine diffeomorphism  
$g: (M',\cF'_{m'}) \ra (M,\cF_{m})$ such that $g \circ f'$ is based homotopic to $f$.
The set of classes is denoted $\StrMf(M_{0}, \bbA^2)$.
\end{example} 

Let us fix a base frame 
$\cE_{x_{0}}$ on affine space $\bbA^n$. Given a marked framed flat affine manifold $(f,M, \cF_{m })$, there exists a unique frame preserving affine chart for $M$, which is defined near $m$. 
This chart lifts to a unique development 
map $D: (\tilde M_{0}, \tilde \cF_{\tilde m_{0}}) \ra (\bbA^n, \cE_{x_{0}})$. 
The correspondence descends to a bijection 
$$ \StrMf(M_{0}, \bbA^n) = 
\Diff_{1,f}(M_{0}, \cF_{m_{0}}) \backslash \Dev_{f}(M_{0}) \; ,$$ where $\Diff_{1,f}(M_{0}, \cF_{m_{0}})$ denotes the group of frame preserving
diffeomorphisms which are based homotopic to the identity. 
By the deformation theorem, there is a map 
$$ \mathsf{hol}: \StrMf(M_{0},\bbA^n) \ra \Hom(\pi_{1}(M,m_{0}), \Aff(n)) \;  $$ 
which is continuous and which is a local 
homeomorphism onto its image. 
\emph{This shows that $\StrMf(M_{0},\bbA^n)$  is a locally compact Hausdorff space}.\\

As is apparent from 
Example \ref{ex:framepreserving}, the natural map 
$\StrMf(M_{0},\bbA^n) \ra \StrM(M_{0},\bbA^n)$ is 
surjective, and it factors over $\StrM_{\mathsf{p}}(M_{0},\bbA^n)$.
The following tower of maps thus sheds some light on
the topology of the deformation space of flat affine structures: 
\begin{align} \label{eq:deftower}
 \xymatrix@1{
\StrMf(M_{0},\bbA^n) \ar[d]
\\
\StrM_{\mathsf{p}}(M_{0},\bbA^n) \ar[d] \\ 
\Def(M_{0},\bbA^n)  =   \StrM(M_{0},\bbA^n)  \; .}  
\end{align}
Note that the group $\GL^+(n,\bbR) = \Diff_{1}(M_{0},m_{0})/ 
\Diff_{1,f}(M_{0}, \cF_{m_{0}})$ acts on $\StrMf(M_{0},\bbA^n)$,
such that the first projection map in the tower is the quotient map for this action. In particular, \emph{$\StrM_{\mathsf{p}}(M_{0},\bbA^n)$
arises as the quotient space of a locally compact Hausdorff
space by a reductive group action}. 
(Note also that the holonomy map is equivariant with respect to 
the conjugation action of $\GL^+(n,\bbR)$ 
on holonomy homomorphisms.)

\begin{example}[Homogeneous framed flat affine structures on the torus] \label{ex:def_homfaff}
As we have seen already in Example \ref{ex:def_homaff} above, the lower map 
in the tower \eqref{eq:deftower} is a bijection on homogeneous flat affine structures, that is, $\Def_{h}(T^2,\bbA^2) = \StrM_{\mathsf{p},h}(T^2,\bbA^2)$.
The deformation space of homogeneous structures 
is thus obtained as a quotient by an action of 
$\GL^+(2,\bbR)$:
 $$ 
\Def_{h}(T^2,\bbA^2) =  {\StrMf}\ _{\!  \! \!,h}(T^2,\bbA^2) /
\GL^+(2,\bbR)   \; . $$    
We shall further
study  this quotient space in \S \ref{sect:homog_structs}.
Observe that the action of 
$\GL^+(2,\bbR)$ is not free, 
since, in fact, the Hopf tori have non-trivial stabilizers.  
On the other hand,  as follows from the discussion in
Example \ref{ex:completeas}, $\GL^+(2,\bbR)$ acts freely on the subspace of complete 
affine structures, and the map 
${\StrMf}\  _{\! \! \!,c}(T^2,\bbA^2) \ra  \Def_{c}(T^2,\bbA^2)$
is a trivial $\GL^+(2,\bbR)$-principal bundle. 
\end{example}

\section{Construction of flat affine 
surfaces} \index{surface!flat affine}

%
A flat affine manifold is called \emph{homogeneous} if its group of
affine automorphisms acts transitively.  
\index{flat affine manifold!homogeneous}  \index{flat affine torus!homogeneous}
Homogeneous  flat affine manifolds may be constructed from 
 \'etale affine representations 
of two-dimensional Lie groups in 
a straightforward way. Compact examples
can be derived from \'etale affine representations 
of the two-dimensional group manifold $\bbR^2$
by taking quotients with a discrete uniform subgroup.
Every flat affine surface constructed in this way is then
a \emph{homogeneous} flat affine torus.
In fact, an easy argument (see \S \ref{sect:classification}) 
shows that \emph{all} homogeneous flat affine surfaces are obtained in this way.
Therefore, all homogeneous flat affine tori are 
affinely diffeomorphic to quotients of abelian 
Lie groups with left-invariant flat affine structure. 
This also relates homogeneous
affine structures on tori to two-dimensional associative algebras,
a point of view which will be discussed in \S \ref{sect:thedefspace}.
In \S \ref{sect:etale_affine}, we describe 
the classification of  abelian  \'etale affine representations on $\bbA^2$.
By the above remarks, this amounts to a rough classification of 
homogeneous flat affine tori.

A genuinely more geometric  approach is to construct flat affine 
surfaces by gluing patches of affine space along their boundaries. 
The affine version of Poincar\'e's
fundamental polygon theorem allows
to construct flat affine tori by gluing affine 
quadrilaterals along their sides. 
The flat affine two-tori thus obtained depend on
the shape of the quadrilateral and also on the 
particular affine transformations which are used in the gluing process.
%
This, in turn, gives natural  
coordinates for an open subset in the deformation space
of flat affine structures on the two-torus.  
As it turns out, the flat affine tori which are obtained by 
gluing an affine quadrilateral along its sides are all homogeneous,
and they form a dense subset in the deformation space of homogeneous 
flat affine structures on the torus.  This material is explained  in \S \ref{sect:agluing}.  \index{torus!flat affine}
 \index{flat affine torus}

To construct all flat 
affine two-tori, it is required to glue more general objects. 
In the following sections \S \ref{sect:tori_bbao} and \S \ref{sect:acylinders}, 
we discuss in detail a construction method for flat affine 
tori with development image 
$\bbAo$, which builds on the idea of cutting flat affine surfaces into simple building blocks. 
Here flat affine tori 
are constructed by gluing several copies of 
half annuli in $\bbAo$, or cutting the surface into 
affine cylinders. Equivalently, these tori are obtained 
by gluing certain strips which are situated in the universal covering space of $\bbAo$, and which project to 
annuli in $\bbAo$.  In this way also \emph{non-homogeneous} examples 
of flat affine tori arise.  

As follows from the main classification theorem, which will be proved  in 
 \S \ref{sect:classification}, 
the above construction methods exhaust all flat affine two-tori. 

\subsection{Quotients of flat affine Lie groups}
\label{sect:etale_affine}

If a Lie group $G$ has an \'etale action 
(cf.\ Definition \ref{def:etale}) on affine space 
we call it an  \'etale affine \index{etale@\'etale representation!affine}
Lie group.  An \'etale affine Lie group carries a natural
left invariant flat affine structure, and,  thus, for every discrete
subgroup $\Gamma \leq G$, the quotient space $\Gamma \lmod G$ inherits the structure of a flat affine manifold. If $G$ is 
abelian the resulting flat affine structure is homogeneous. \\

The following result will be established in the course of the proof 
of the  classification theorem (see  \S\ref{sect:univcovs}): 
\index{flat affine torus!homogeneous} 

\begin{proposition} Every homogeneous flat affine two-torus is 
affinely diffeomorphic to a quotient of an abelian \'etale affine Lie group.
\end{proposition}

\noindent Up to affine conjugacy there are six types $\mathsf{T}, 
\mathsf{D}, \mathsf{C_{1}}, \mathsf{C_{2}}, \mathsf{B}, \mathsf{A}$ of
 \'etale \emph{abelian} subgroups in the affine group $\Aff(2)$. Both
 the plane and the halfplane admit two distinct simply transitive 
 abelian affine actions $\mathsf{T}, 
\mathsf{D}$, and  $\mathsf{C_{1}}, \mathsf{C_{2}}$ respectively.  

\begin{example}[Affine automorphisms of development images] 
\label{ex:domains}
\hspace{1cm} 
\begin{enumerate} 
\item {\bf  (\emph{The plane $\bbA^2$})} The groups 
$$ \mathsf{T} =  \left\{ 
\begin{matrix}{ccc}  1 & 0 & u   \\  0 & 1 & v \\ 
0 &  0 & 1  
\end{matrix} 
\right\} \;  
  \text{ and } \;  \mathsf{D} =
 \left\{ 
\begin{matrix}{ccc}  1 & v & u+ \frac {1}{2} v^2  \\  0 & 1 & v \\ 
0 &  0 & 1  
\end{matrix}  
\right\} \; 
$$
are abelian groups of affine transformations which are simply transitive on the plane. 


\item {\bf  (\emph{The half space ${\cal H}$})} 
Let  ${\cal H}$ be the
half space $y >0$. Then
$$  \Aff({\cal H}) = 
\left\{ \begin{matrix}{ccc} \alpha & z & v  \\  0 & \beta & 0 \\  
0 &  0 & 1  \end{matrix} \Mid  \alpha \neq 0 ,\,\beta > 0\right\} \; $$
is its affine automorphism group.
The subgroups 
$$ \mathsf{C}_{1} = \left\{ \begin{matrix}{cc}  \exp(t) & z  \\  0 &  \exp(t)   \\ 
\end{matrix} \right\}  \subset \GL(2,\bbR) $$
and 
$$ \mathsf{C}_{2} = \left\{ \begin{matrix}{ccc}  1 & 0 & v \\ 0 &  \exp(t) & 0 \\ 
0 &  0 & 1  \end{matrix} \right\} \subset \Aff(2)$$ 
are simply transitive abelian groups of affine transformations on ${\cal H}$. 
The half spaces $(x,y)$, $y>0$ and $y <0$ are open orbits for the groups
$\mathsf{C}_{i}$.



\item {\bf  (\emph{The sector ${\cal Q}$})} Let ${\cal Q}$ denote the upper right
open quadrant. Then 
$$   \mathsf{B} = \Aff(\mathcal Q)^0 =  \left\{ 
\begin{matrix}{cc} 
   a  & 0 \\  0 &  b  
\end{matrix} \Mid a>0, b> 0  \right\}  \; \subseteq \GL(2,\bbR)
$$ 
is an abelian, simply transitive \emph{linear}  group of transformations of  ${\cal Q}$.


\item  {\bf  (\emph{The punctured plane $\bbAo$}) }
$$   \mathsf{A}  = \left\{   \exp(t) \begin{matrix}{cc}    \cos \theta & \sin \theta \\
- \sin \theta  &  \cos \theta  \\ 
\end{matrix} \right\} \; \subseteq \GL(2,\bbR) = \Aff(\bbAo) $$
is an abelian \emph{linear} group, which is simply transitive on 
$\bbAo$. \end{enumerate}
\end{example}

Let $\mathsf{G} \leq \Aff(2)$ be one of the above 
 groups and $\Gamma \leq \mathsf{G}$ a lattice. Then $\Gamma$ acts properly discontinuously and with compact quotient on every open orbit $U \subset \bbA^2$ of  $\mathsf{G}$ and the quotient
 space $$  M = \; \;
 _{\mbox{$\Gamma$}} \,  \lmod \, U $$  is a flat affine two-torus. 
The group $\Gamma$ is a discrete abelian subgroup 
of $\Aff(2)$ and, by choosing an appropriate  fundamental domain,  its action defines a tessellation of the open domain $U$. In fact,  a convex affine quadrilateral may be chosen as
fundamental domain. 
See Figure \ref{figure:basictess} and Figure \ref{figure:basictess2} for some examples.  
Since $\mathsf{G}$ is abelian and centralizes $\Gamma$, the affine action of $\mathsf{G}$ on $U$ descends to  $M$. Thus, $\mathsf{G}$ acts on $M$ by affine transformations, and $M$ is a homogeneous flat affine
manifold. 

\begin{figure}[htbp]  
\begin{center}
\includegraphics[scale=0.5, clip=true, trim=  0cm 0.35cm 0cm 0.1cm ]{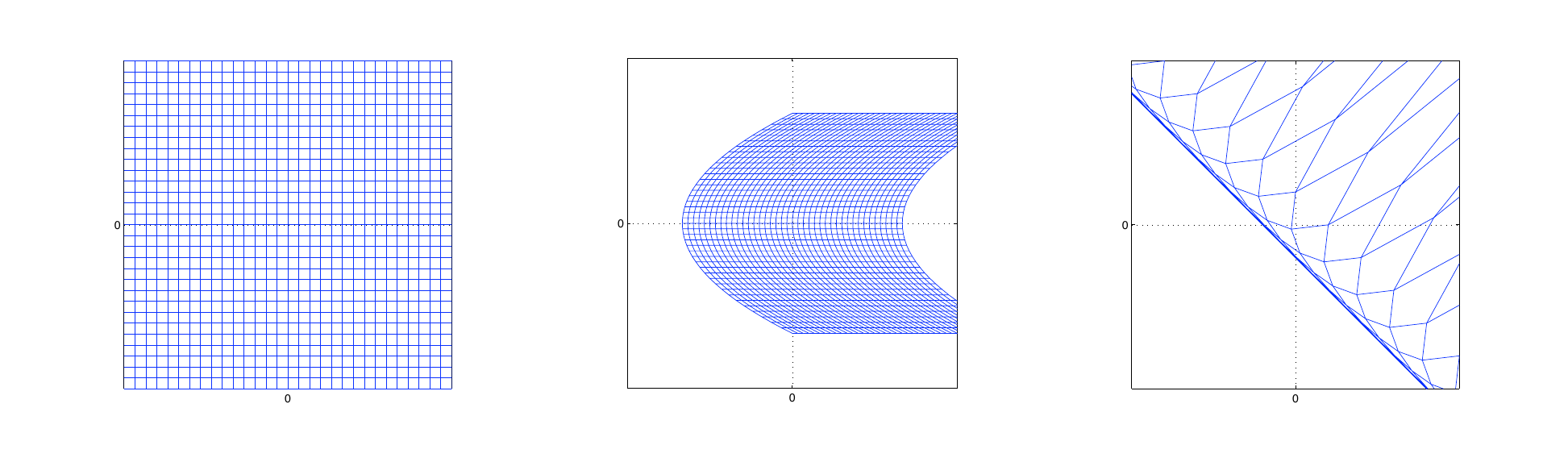}
\caption{Tesselations of homogeneous affine domains of type $\mathsf{T},\mathsf{D}, \mathsf{C}_{1}$.}
\label{figure:basictess}
\end{center}
\end{figure}

For further reference we note the following:
\begin{lemma}[Normalisers of \'etale affine groups] \label{lem:normalizers}
\begin{enumerate}
\item 
The  \'etale affine groups $\mathsf{A}$, $\mathsf{B}$ have index two and eight in their normalizers in $\Aff(2)$. 
The quotients are generated  by the reflections 
$$
\begin{matrix}{cc} 
0 & 1\\
1 & 0 \\
\end{matrix} \in \GL(2,\bbR) 
\text{, respectively, }
\begin{matrix}{cc}
-1& 0 \\
0 & 1 \\
\end{matrix}, 
\begin{matrix}{cc}
1& 0 \\
0 & -1 \\
\end{matrix},
\begin{matrix}{cc} 
0 & 1\\
1 & 0 \\
\end{matrix}  .  
$$ 
\item The normalizers in $\Aff(2)$ of the \'etale affine groups $\mathsf{C}_{1}$, $\mathsf{C}_{2}$ are 
$$  \left\{ \begin{matrix}{cc}  \alpha & z  \\  0 &  \beta   \\ 
\end{matrix} \right\} \; \subset \GL(2,\bbR), \; 
\left\{ \begin{matrix}{ccc}  \alpha & 0 & v \\  0 &  \beta  & 0 \\ 
0 &  0 & 1  \end{matrix} \right\} \; \subset  \Aff(2)$$
respectively.  
\item The normalizer in $\Aff(2)$ of the \'etale affine group $\mathsf{D}$  is the semi-direct 
product generated by $\mathsf{D}$ and the linear group 
\[ {\cal N}_\mathsf{D} = \left\{ \begin{matrix}{ccc}d^2 & b  & 0  \\
                                               0  & d &   0\\
                                               0  & 0 &  1 \\
                          \end{matrix} \, \Mid  \;   d \in \bbR^{*}, \, b \in \bbR \, \right\}  \; .  \]

\end{enumerate}
\end{lemma}

\begin{figure}[htbp]  
\begin{center}
\includegraphics[scale=0.5, clip=true, trim=  0cm 0.35cm 0cm 0.1cm ]{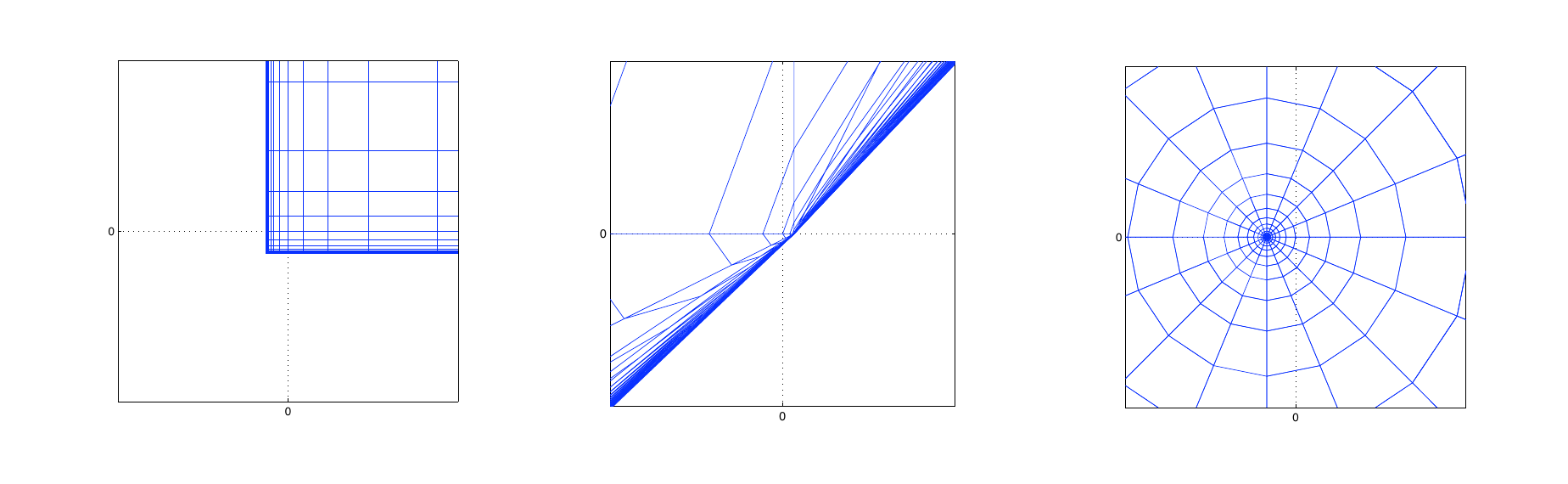}
\caption{Tesselations of homogeneous affine domains of type $\mathsf{B},\mathsf{C}_{1},\mathsf{A}$.}
\label{figure:basictess2}
\end{center}
\end{figure}

 In the case of the group $\mathsf{A}$, which is not simply connected, we can consider, more generally, the universal covering group  $\tilde {\mathsf{A}}$ of 
 $\mathsf{A}$. The covering homomorphism 
 $\tilde {\mathsf{A}} \ra \mathsf{A}$ turns $\tilde {\mathsf{A}}$ 
 into an \'etale affine Lie group.
 Let $\Gamma$ be a lattice in $\tilde {\mathsf{A}}$. Then
$$ M =  \; \; _{{\mbox{$\Gamma$}}} \, \lmod \, \tilde {\mathsf{A}} $$   
inherits a flat  affine structure, for which the orbit map $\tilde {\mathsf{A}} \ra \bbAo$ at any point $x \in \bbAo$ 
is a development map. The holonomy group $h(\Gamma) \leq \mathsf{A}$ is 
the image of $\Gamma$ under the covering $\tilde {\mathsf{A}} \ra \mathsf{A}$. Since $\Gamma$ is central in $\tilde {\mathsf{A}}$, the group
 $\tilde {\mathsf{A}}$ acts on $M$ by affine transformations. 
 In particular, as before, $M$ is a \emph{homogeneous} 
 flat affine two-torus. 

\begin{figure}[htbp]  
\begin{center}
\includegraphics[scale=0.4, clip=true, trim=  0cm 0.9cm 0cm 0.7cm ]{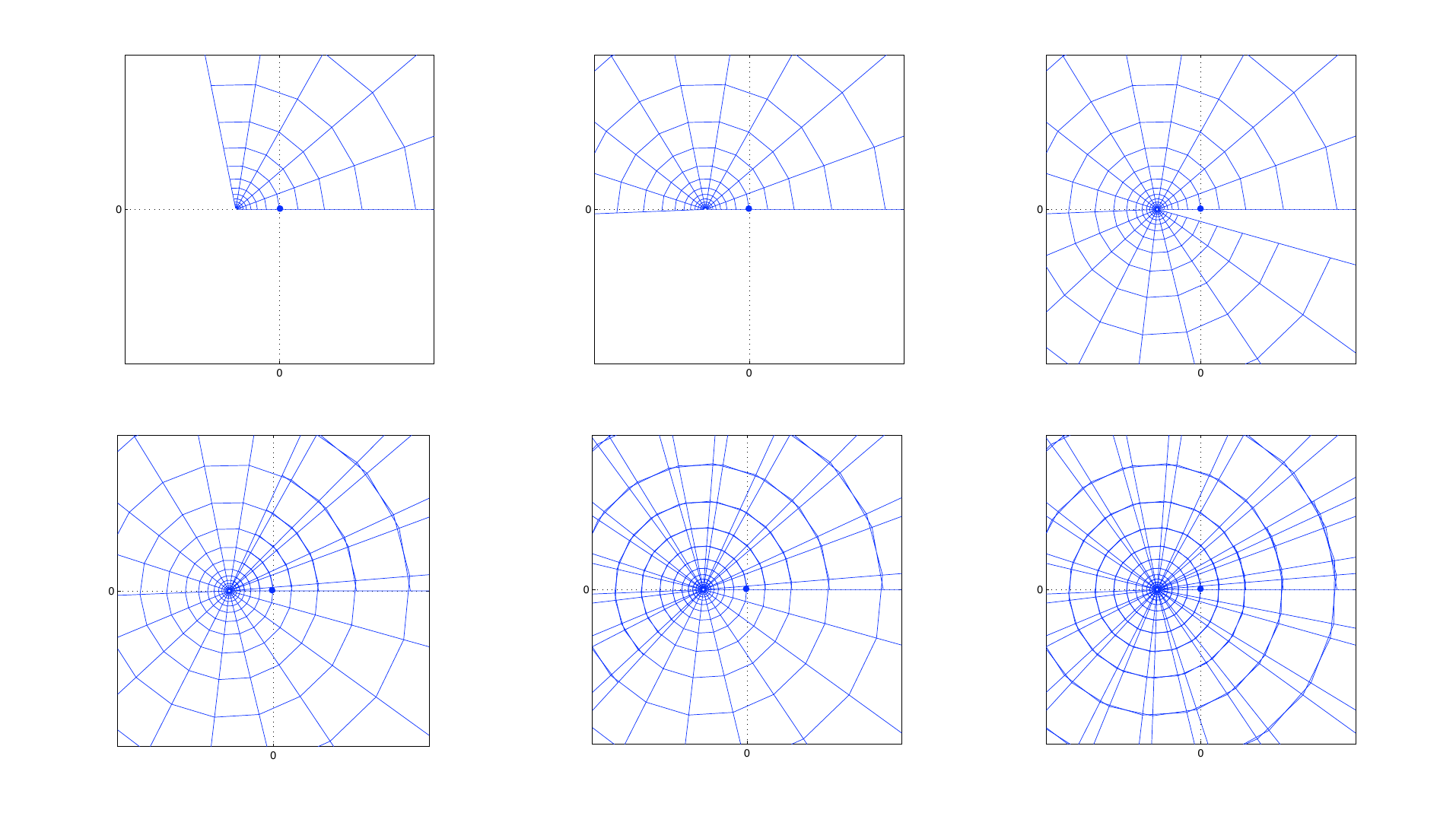}

\caption{Development process with non-discrete holonomy group.}
\index{holonomy!non-discrete}
\label{figure:non-discrete}
\end{center}
\end{figure}

Note that, in this case, it may also happen that the holonomy $h(\Gamma)$ is \emph{not} discrete in $\Aff(2)$, see Figure \ref{figure:non-discrete}. 

If the holonomy is cyclic,  as is the case for Hopf tori 
(Example \ref{ex:hopftori}), $M$ cannot be constructed by gluing an affine quadrilateral which is contained in the development image. However, $M$ may always be obtained by gluing a 
strip which is situated in 
$\tilde {\mathsf{A}}$, see Example \ref{ex:cylinders2}.

\subsection{Affine gluing of polygons}  \label{sect:agluing}

Let $\cP \subset \bbA^2$ be polygon with ${\cal S}$ its set of sides.
Let $\{ g_S \in \Aff(2) \mid  \,  S \in {\cal S} \}$ be a set of affine transformations pairing the sides of $\cP$ and let 
$M$ be the corresponding identification space of $\cP$.
We say that the \emph{affine gluing criterion}\footnote{See \cite{BauesG} for more details on this definition} holds  if, for each vertex $x \in \cP$ with cycle of edges $S_1,  \ldots, {S_m}$, 
the \emph{cycle relation} $g_{S_1} \cdots g_{S_m} = 1$ holds, 
and furthermore the corners at $x$ of the polygons 
$g_{S_1} \cdots g_{S_i} \cP$, $i = 1, \ldots ,m$, add up 
subsequently to a disc, while intersecting only in their consecutive
boundaries. This disc then provides an affine coordinate neighborhood  in the identification space of $\cP$ defined
by the pairing of sides.  If the gluing criterion is satisfied the identification space $M$ inherits the structure of a flat affine manifold from $\cP$. \\

The following result is the analogue of Poincar\'{e}'s fundamental polygon theorem (cf.\ \cite{Maskit} for the classical 
\index{Poincar\'{e} fundamental polygon theorem}
version) for gluing flat affine surfaces: 

\begin{proposition}[\mbox{see \cite[Proposition 2.1]{BauesG}}] \label{prop:agluing}
If the affine gluing criterion holds then the group
$\Gamma  \subset \Aff(2)$ generated by the side-pairing transformations 
$\{ g_S | \, S \in {\cal S} \}$ acts properly discontinuously and with fundamental domain $\cP$
on a flat affine surface $\bar{X}$ which develops 
$\Gamma$-equivariantly onto an open set $U$ in $\bbA^{2}$.  
The inclusion of $\cP$ into $\bar{X}$ identifies $M$ and the orbit space 
$ _{{\mbox{$\Gamma$}}} \lmod \bar{X}$.
\end{proposition} 

It follows that $M$ inherits a natural flat affine structure from $\cP$. In fact, the surface $\bar{X}$ is the holonomy covering space
of $M$ and the group $\Gamma$ is the holonomy group of $M$.
Note also that the construction of $\bar{X}$ is sketched in the proof of Theorem \ref{thm:Deformations}. 
The situation is pictured in the following commutative diagram
of maps: 
\begin{eqnarray*}
  \bar{X} & \stackrel{\bar D}{\longrightarrow}  &  U \subset \bbA^{2}  \\ 
  \downarrow &        & \downarrow  \\
M=  \; _{{\mbox{$\Gamma$}}} \lmod \bar{X} & \longrightarrow   & {_{{\mbox{$\Gamma$}}} \lmod U} \; \; . \\  
\end{eqnarray*} 

\begin{example}
Figure \ref{figure:glue1} 
shows how to glue a trapezium $\cT$ with angle $\alpha < \pi$. 
The sides $S_1$ and $S_3$  are glued with a homothety. 
$S_2$ and $S_4$ are glued with a rotation of angle $\alpha$. 
The developing image is $U = \bbAo$. 
It is tessellated by the translates of $\cT$ if and only if  $m \alpha = 2 \pi$ for an 
integer $m$. If the angle $\alpha$ is rational, $p \alpha = q 2\pi$, the development $\bar{X} \rightarrow \bbAo$ is
a finite cyclic covering of degree $q$. Otherwise the development is an infinite cyclic covering and $\bar{X}$ is simply connected. 
\end{example}

\begin{figure}[ht] 
\unitlength1mm
\begin{center}
\begin{picture}(20, 25)    
\small 

\thinlines 
\put(10,23){\circle*{1}}
\put(10,23){\line(-1, -2){9}} 
\put(10,23){\line( 1, -2){9}}
\put(9,18){$\alpha$}
\thicklines

\put(6,15){\line(1, 0){8}}
\put(6,15){\vector(1, 0){7}}
\put(6,15){\vector(1, 0){5}}
\put(8,11){$S_3$}

\put(6,15){\line(-1, -2){4}}
\put(2,7){\vector(1, 2){3}}
\put(-2,10){$S_4$}

\put(14,15){\line(1, -2){4}}
\put(18,7){\vector(-1,2){3}}
\put(18,10){$S_2$}

\put(2, 7){\line(1, 0){16}}
\put(2, 7){\vector(1, 0){11}}
\put(2, 7){\vector(1, 0){9}}
\put(8,3){$S_1$}

\end{picture} 
\caption{Gluing a trapezium in $\bbA^{2}$.}
\label{figure:glue1} 
\end{center}
\end{figure}
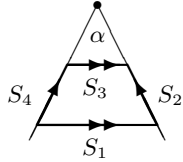 

\subsubsection{Gluing affine quadrilaterals} 
Theorem \ref{thm:Benzecri} implies that the gluing criterion imposes strong restrictions on the possible combinatorial types of polygons and pairings.
However, flat affine two-tori are easily obtained by gluing an affine 
quadrilateral $\cP$ in the way indicated in Figure \ref{gluingq}. 
The equivalence
class of the flat affine manifold thus obtained depends
on the affine equivalence class of $\cP$ and the particular
side pairing transformations chosen, see \cite[\S 3]{BauesG}. 
\begin{figure}[ht]  \unitlength1mm
\begin{center}
\begin{picture}(20, 7)(0,10)    
\small 
\put(2,7){\line(0, 1){8}}
\put(2,7){\vector(0, 1){5}}
\put(2,15){\line(1, 0){16}}
\put(2,15){\vector(1, 0){11}}
\put(2,15){\vector(1, 0){9}}
\put(18,7){\line(0, 1){8}}
\put(18,7){\vector(0,1){5}}
\put(2, 7){\line(1, 0){16}}
\put(2, 7){\vector(1, 0){11}}
\put(2, 7){\vector(1, 0){9}}
\end{picture} \end{center}
\caption{Gluing a torus.}
\label{gluingq}
\end{figure}
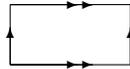 

The gluing conditions for such a pairing are easily verified:
\begin{lemma}[{\cite[Lemma 3.2]{BauesG}}] \label{gluelemma}
Let 
$ A,B  \in \Aff(2)$, $A (0,0) =(1,0) $, $A(0,1) = p$, $B(0,0) = (0,1)$,
$B(1,0)= p$.
The side pairing transformations  $\{ A, A^{-1}\! , B, B^{-1} \}$ for the polygon with vertices $\cP = ( (0,0), (1,0), p, (0,1) )$ satisfy the gluing conditions if and only if\/  $\det l(A)>0$ and $\det l(B)>0$ (where 
$l$ denotes the linear part of an affine transformation) and 
$[A,B] = \Id$. 
\end{lemma}  

We remark further: 
\begin{proposition} Every flat affine two-torus obtained by gluing a quadrilateral $\cP$ on its sides is homogeneous. Conversely, 
if $M$ is a homogeneous flat affine two-torus with non-cyclic affine holonomy group, then $M$ may be obtained by gluing a quadrilateral. 
\end{proposition}
\begin{proof} For the proof that the gluing torus of $\cP$ is
homogeneous, we have to appeal to some of the facts
which are explained in \S \ref{sect:classification}. In particular,
we use Proposition \ref{prop:devimages} and 
Proposition \ref{prop:liftofN}.
The gluing conditions imply that the minimal connected 
abelian subgroup $N$ which contains the holonomy of $M$ is at least two-dimensional. Therefore, $N$ is one of the two-dimensional abelian subgroups listed in Example \ref{ex:domains}.
By Proposition \ref{prop:liftofN}, $N$ acts on the
development image of $M$. Let $U$ be an open orbit
for $N$, which is one of the domains of Example \ref{ex:domains}. 
Then, by convexity of the homogeneous domain $U$, the polygon $\cP$ is contained in $U$. The construction of $M$ and its development process show that the development image of $M$ is covered by the holonomy translates of $\cP$. Since the holonomy is contained in $N$, it follows that the development image of $M$ is contained in the open orbit $U$.  On the other hand, by Proposition \ref{prop:liftofN}, the development image contains $U$. This implies that the development image of $M$ equals the open orbit $U$.  
Therefore, $N$ acts transitively on the development image, and thus also on $M$. In particular, it follows that $M$ is homogeneous. 

We omit the proof of the converse statement. A special case
is treated in \cite[\S 4.6]{BauesG}.
\end{proof} 

\subsubsection{The gluing variety}
By Lemma \ref{gluelemma}, the set of side pairings  $$
{\cal V} = \{ (p, A, B ) \} \, \subset  {\bbR}^2 \times {\bbR}^6 \times {\bbR}^6 \; , $$ 
which satisfy the gluing conditions for the  (convex) quadrilateral $$\cP = ( \, (0,0), (1,0),p , (0,1) \, ) $$
form a semi-algebraic subset
of $ {\bbR}^2 \times {\bbR}^6 \times {\bbR}^6$.
It is easily computed that ${\cal V}$ is four-dimensional
and that the set of solutions with respect to a fixed $\cP$ is 
of dimension two. We let $\mathsf{p}: {\cal V} \ra \bbR^2$ denote the projection to the first factor. Note that the projection of  ${\cal V}$ to the matrix factors  ${\bbR}^6 \times {\bbR}^6$
defines an 
\emph{embedding} 
$  {\cal V} \, \hookrightarrow \,  \Hom(\bbZ^2, \Aff(2))$
as a subset of the holonomy image.
We call the set 
${\cal V}$ the {\em gluing variety of quadrilaterals}.\\

\paragraph{Embedding into the deformation space}
Observe that the gluing of a quadrilateral $\cP$ 
naturally constructs a \emph{framed affine two-torus}.
If we choose a diffeomorphism of the unit-square
with $\cP$, the development process of the gluing extends this
diffeomorphism to the development map of a \emph{marked} framed affine two-torus (see Example \ref{ex:fat}). This defines a continuous map
$$ {\cal V} \supset \mathsf{p}^{-1} (\cP) \ra \StrMf(T^2,\bbA^2) \; ,
$$
where $\mathsf{p}: {\cal V} \ra \bbR^2$ is the projection
to the first factor.  
We may furthermore choose a natural identification
of the unit square with $\cP$ (for example, by decomposing 
any quadrilateral into two triangles and using affine identifications  
of the triangles). Then, using \eqref{eq:deftower}, we obtain
 a continuous (open) embedding
$$ \nu:  {\cal V} \ra  \StrMf(T^2,\bbA^2) \;  $$
to the space of classes of framed flat affine tori.
Note that $ \nu$ 
is a section of  the holonomy map $\mathsf{hol}: 
\StrMf(T^2,\bbA^2) \ra  \Hom(\bbZ^2, \Aff(2))$.
Since $\cal V$ also defines 
a slice for the $\GL(2,\bbR)$-orbits on  $
\Hom(\bbZ^2, \Aff(2))$, the map $\nu$ descends 
to an embedding $ {\cal V} \ra \StrM_{\mathsf{p}}(T^2,\bbA^2)$, 
whose image consists of homogeneous structures.
Therefore, by the discussion in Example \ref{ex:def_homfaff},  
\emph{the gluing variety $ {\cal V}$
embeds as a locally closed subset of the 
deformation space $\Def(T^2,\bbA^2)$}.

\subsection{Tori with development image $\bbAo$} \label{sect:tori_bbao}

Here we discuss how to glue flat affine tori from annuli which are contained in the once punctured plane or the universal covering 
flat affine manifold of the once punctured plane.

\subsubsection{Hopf tori and quotients of $\bbAo$} \label{sect:Hopftori}
The simplest examples of flat affine tori with development image 
$\bbAo$ are obtained by gluing closed
annuli  
along their boundary curves. 

\begin{example}[Hopf tori]  \label{ex:hopftori}
Let $A_{\lambda}$, 
$\lambda >0$,  be a dilation with scaling factor $\lambda$, 
and $\Gamma =  \,  \langle  \, A_{\lambda} \,  \rangle$ 
the subgroup of $\GL(2,\bbR)$
generated by $A_{\lambda}$. Then $\Gamma$ acts properly discontinuously on 
$\bbAo$ and the quotient space 
$$ \cH_{\lambda} =  \;  \genfrac{}{}{0pt}{0}{}{\Gamma}\backslash \; \left(\bbAo\right)$$ is a compact flat affine two-torus 
$\cH_{\lambda}$, which is called a \emph{Hopf torus}.  
 \end{example} 

The Hopf torus $\cH_{\lambda}$ is obtained by gluing a closed  
annulus  $\cA_{\lambda} \subset \bbAo$ of width $\lambda$
along its boundary circles, see Figure \ref{figure:annuli}.

\begin{figure}[htbp]  
\begin{center}

\includegraphics[scale=0.5, clip=true, trim=  0cm 0.35cm 0cm 0.1cm ]{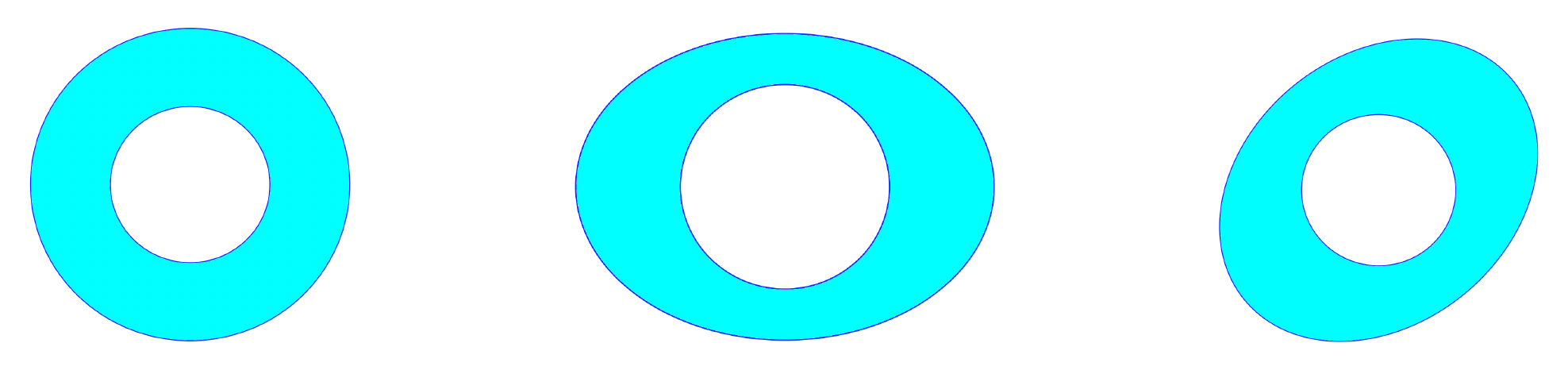}
\caption{Gluing of annuli: Hopf torus, non-homogeneous flat affine tori.}
\label{figure:annuli}
\end{center}
\end{figure}

The geometric construction of the Hopf tori $\cH_{\lambda}$ may be refined 
as follows:  

\begin{example}[Finite coverings of  Hopf tori]  \label{ex:finhopftori}
Let $$ X_{k} \,  \lra  \; \bbAo$$ be a $k$-fold covering flat affine manifold. 
Then we may lift the action of $A_{\lambda}$ on $\bbAo$ to
a properly discontinuous action of an affine transformation 
$A_{\lambda,k}$ of $X_{k}$. The quotient spaces 
$$ \cH_{\lambda,k} = \;   \langle  A_{\lambda,k} \rangle \, \lmod  X_{k} $$ 
are 
flat affine manifolds, which are $k$-fold 
covering spaces of $\cH_{\lambda}$. 
Geometrically, $X_{k}$ is a  
topological annulus with
a flat affine structure which is obtained 
by cutting $\bbAo$ at a radial line and then 
gluing $k$ copies of $\bbAo$ along this 
geodesic ray. A \emph{geodesic} in a 
flat affine manifold is a curve which 
corresponds to a straight line in all affine coordinate 
charts. Thus, correspondingly, the manifolds $\cH_{\lambda,k}$ 
are obtained by gluing $k$ copies of $\cH_{\lambda}$ 
at a closed geodesic. 
\end{example} 

Note that the family of Hopf tori $\cH_{\lambda,k}$ 
gives a simple example of a family of distinct 
flat affine manifolds which have identical holonomy homomorphism.

\begin{example}[Finite quotients of  Hopf tori]  \label{ex:finquothopftori}
Let $\R_{\alpha}$ be a rotation with angle $\alpha= \frac{p}{q} \pi $
a rational multiple of $\pi$. Then the finite group of rotations 
of order $2q$ generated by $\R_{\alpha}$ acts without fixed points
on $ \bbAo$ and on the Hopf tori $\cH_{\lambda,k}$. 
Therefore, the  quotient spaces $$ \cH_{\lambda, \alpha,k} =  \;  \langle  \R_{\alpha} \rangle
 \,  \lmod \cH_{\lambda,k}$$ are flat affine two-tori.
\end{example} 

Since $A_{\lambda}$ is in the center of $\GL(2,\bbR)$,
$\cH_{\lambda}$ is a homogeneous flat affine manifold 
with affine automorphism group 
$$ \Aff(\cH_{\lambda})= \GL(2,\bbR)/\Gamma \; .
$$ Hence, its finite coverings
$ \cH_{\lambda,k}$ are homogeneous, as well. Similarly, 
$ \cH_{\lambda, \alpha,k}$ are homogeneous flat affine two-tori, 
with $\Aff(\cH_{\lambda, \alpha})^0$
isomorphic to $\GL(1,\bbC)/\Gamma$, except for $\alpha = \pi$. In the latter case $\Aff(\cH_{\lambda, \pi}) = \PGL(2,\bbR)/ \Gamma$.

\paragraph{Expanding holonomy}
Non-homogeneous quotients of $\bbAo$ may be constructed
by using expanding elements of $\GL(2,\bbR)$.
A matrix $A \in \GL(2,\bbR)$ is called an \emph{expansion} if 
it has real eigenvalues $\lambda_{1}, \lambda_{2} > 1$.  
($A^{-1}$ is then called a \emph{contraction}.)
Every expansion acts properly
discontinuously on $\bbAo$, see Figure \ref{figure:expanding}.
This motivates the following: 
\begin{definition}[Expanding elements]
A matrix $A \in \GL(2,\bbR)$  is called  \emph{expanding} if it
acts properly on  $\bbAo$ and
every compact subset of $\bbAo$ is moved to infinity by its iterates 
$A^k$, $k \to \infty$.
\end{definition}
Note that, if $A$ is expanding,  it is either an expansion, or a product of 
an expansion with $\R_{\pi}$, or it is conjugate to a product of a 
dilation and a rotation. 

\begin{figure}[htbp] 
  \centering
    \includegraphics[clip=true, height= 2.2cm, trim= 32cm 11.5cm 12cm 11.5cm]{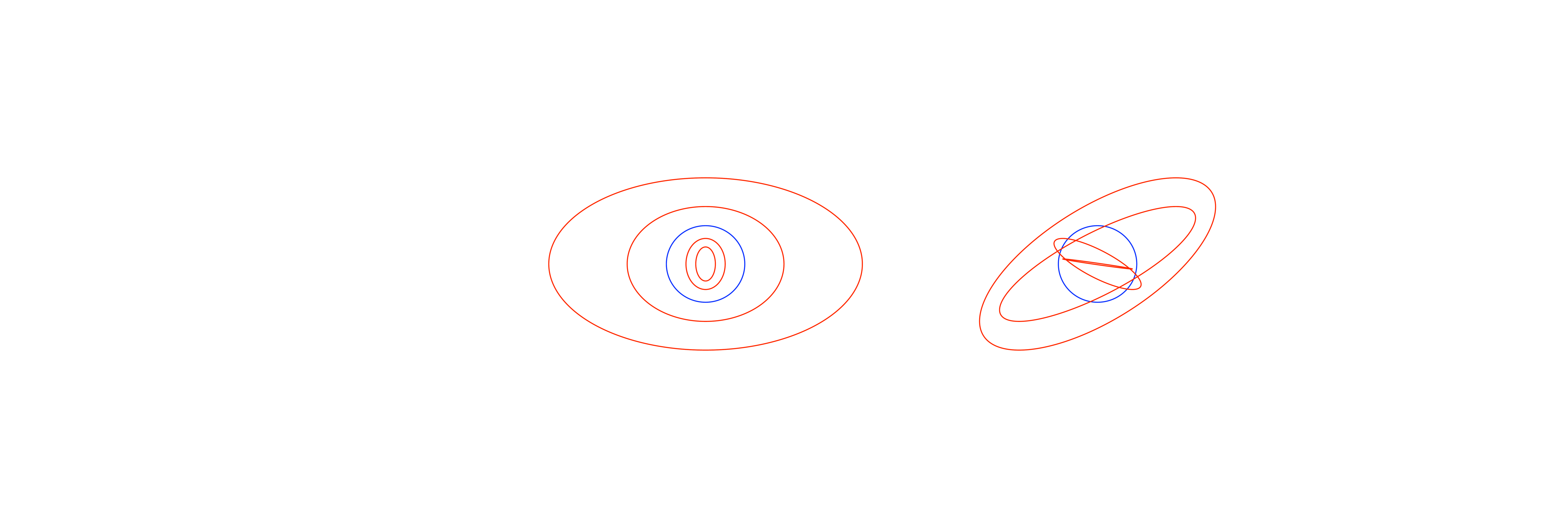}
  \caption{Dynamics of expanding elements in $\GL(2,\bbR)$.}
  \label{figure:expanding}
\end{figure}

\begin{figure}[htbp] 
  \centering
    \includegraphics[clip=true, height= 2.2cm, trim= 32cm 12cm 10cm 12cm]{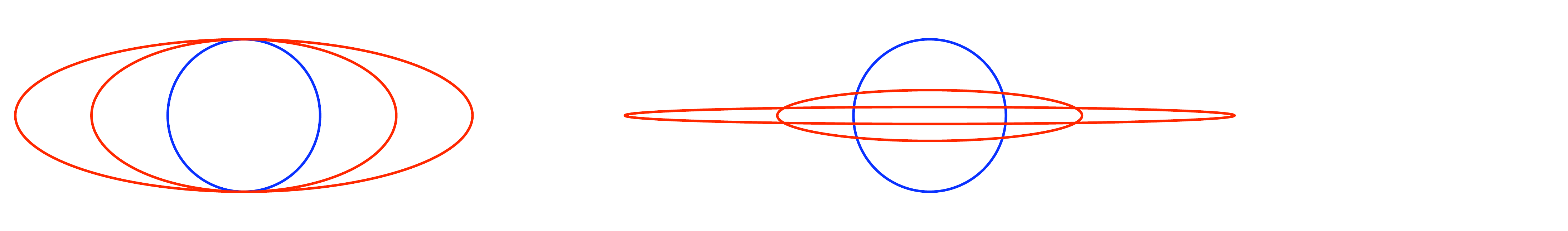}
  \caption{Dynamics of non-expanding elements in $\GL_{2}(\bbR)$.}
  \label{figure:nonexpanding}
\end{figure}


\begin{example}[Tori with expanding holonomy] 
\label{ex:expholonomy}
If $A$ is an expanding element then the quotient space 
$$ \cH_{A} = \langle A \rangle \lmod \, \bbAo$$
is  a flat affine two-torus with development image $\bbAo$.
If $A$ is an expansion, the torus $\cH_{A}$ is 
obtained by gluing an annulus $\cA_{A} \subset \bbAo$, as indicated 
in Figure \ref{figure:annuli}. Note, if $A$ is an expansion which is not a dilation then $\cH_{A}$ is a flat affine torus, which
is  \emph{not} homogenous.
\end{example}
%
%

\subsubsection{Quotients of $\widetilde \bbAo$} \label{sect:gentori}
We consider the universal covering
flat affine manifold 
$$\mathsf{q}: \widetilde \bbAo\,  \longrightarrow \, \bbAo$$  of 
the open domain $\bbAo$ in $\bbA^2$. Let  $\Aff( \widetilde \bbAo)$ be its group of affine
diffeomorphisms. 

\paragraph{Universal covering of $ \GL(2,\bbR)$}
The development $\mathsf{q}$ induces a surjective homomorphism  
$$ \mathsf{p}: \Aff( \widetilde \bbAo) \lra \Aff(\bbAo) = \GL(2,\bbR) \; , $$ 
which exhibits $\Aff( \widetilde \bbAo)$ 
as the universal covering group of $\GL(2,\bbR)$. 

Let $\R_{\pi} \in \GL(2,\bbR)$ denote rotation by $\pi$. 
The center of the group $$ \Aff( \widetilde \bbAo) = 
 \tGL(2,\bbR)$$ is therefore generated by an element 
$\tau$  which satisfies $\mathsf{p}(\tau) = \R_{\pi}$, 
and the kernel of  $\mathsf{p}$ is generated by $\tau^2$.
(cf.\ \S \ref{sect:GL2R}.)

\paragraph{Polar coordinates.} 
We let  $(r, \theta)$ denote polar coordinates for $\widetilde \bbAo$.  Then  $$ \tau: (r, \theta) \mapsto (r, \theta + \pi) \, . $$
More generally, the universal covering $\widetilde \SO(2,\bbR)$ of the rotation group is a subgroup of  $\tGL(2,\bbR)$ which acts by translations in the $\theta$-direction. 

\paragraph{Elements with non-zero rotation angle.} 
Let $B \in \GL^+\!(2,\bbR)$ have positive eigenvalues.
After conjugation, we may assume that $B$ preserves the horizontal coordinate
axis in $\bbA^2$. We let $\tilde B= \tilde B_{0}$ denote
the lift of $B$ to $\widetilde \bbAo$ which  preserves the line $\theta = 0$.
It follows that
$\tilde B$ preserves all horizontal strips 
$${\bar \Omega}_{\ell} = \;  \{ \, (r, \theta) \mid \ell \, \pi \leq \theta \leq (\ell+1) \,  \pi \, \}$$ 
and their boundary components.
We observe that  (the group generated by) any other lift
$$ \tilde B_{k} = \tau^k \tilde B \, , \, k \neq 0 \; , $$  
acts properly on 
$\widetilde \bbAo$. 

\begin{definition}[Non-zero angle of rotation] \label{def:posrot}
Let $\tilde B \in \widetilde \GL^+\!(2,\bbR)$.
We say that $\tilde B$ has a non-zero rotation angle if 
$\tilde B$ acts properly on $ \widetilde \bbAo$, and for every compact subset the $\theta$ coordinates are unbounded
under the iterates  $\tilde B^k$, $k \to \infty$.
\end{definition}

The property to have non-zero rotation is an affine invariant
that is, it is invariant by conjugation 
in $\tGL^+\!(2,\bbR)$. 
In particular, 
$\tilde B \neq 1$ has non-zero angle of rotation if and only if $\tilde B$ is conjugate
to an element of $\widetilde \SO(2)$ or $B$ has positive eigenvalues and 
$\tilde B = \tilde B_{k} = \tau^k \tilde B_{0}$, $k \neq 0$, 
as above. 

\paragraph{Proper actions on $\widetilde \bbAo$.}
Let $\bar {\mathcal H}_{0}^k = \bigcup_{\ell = 0 \ldots k} 
{\bar \Omega}_{\ell}$ be the successive 
union of  $k$ strips  ${\bar \Omega}_{\ell}$.
Note that the development image of ${\bar \Omega}_{\ell}$ is
a closed halfspace with the origin $0$ removed, and
the development image of $\bar {\mathcal H}_{0}^k$ 
is $\bbAo$, $k \geq 1$  (compare also Figure 
\ref{figure:deviscov}).

\begin{example}[Affine cylinders without boundary]
Consider the quotient flat affine manifolds 
$$  X_{B,k} =   \;  {\langle \tilde B_{k} \rangle}
 \,  \lmod \, \widetilde  \bbAo \,  , \;  \,   k \geq 1 \; .$$
These are open affine cylinders which are obtained by gluing 
$\bar {\mathcal H}_{0}^k$  along its two incomplete 
boundary geodesics.   
The development image of $X_{B,k}$ is $\bbAo$, and 
its holonomy group is generated by  $\R_{\pi}^k B$.
\end{example}

Let $A$ be an expansion which commutes with $B$ and
$\tilde A= \tilde A_{0}$ the lift of $A$ which preserves the
line $\theta=0$. Then $\tilde A$ acts properly on 
$X_{\tilde B}$.  
\begin{example}[Quotients of $\tbbAo$] \label{ex:gentorusnh}
We  obtain the quotient affine torus
$$     \cT_{\tilde A, \tilde B} = \cT_{A,B,k}  =   \;  \langle \tilde A \rangle
 \,  \lmod   X_{B,k} \; , k \neq 0 \,   . $$
The holonomy homomorphism of 
$\cT_{A,B,k}$  is determined by $A$, $B$ and
the parity of $k$. 
\end{example}

\begin{example}[Affine cylinders without boundary, general case]  \label{ex:acylinders1}
Let $\tilde B$ be an element of $\widetilde \GL^+\!(2,\bbR)$ which has non-zero rotation. 
Then the quotient flat affine
manifolds 
$$  X_{\tilde B} =   \;  {\langle \tilde B  \rangle}
 \,  \lmod \, \widetilde  \bbAo \,    $$
are open cylinders which are obtained by gluing a
strip  $$\bar {\mathcal H}_{\alpha}  = \{ (r, \theta) \mid 0 \leq \theta \leq \alpha \}$$
in $\widetilde \bbAo$ along its two boundary geodesics.

\end{example}


Now let  $B \in \GL(1,\bbC)$ and $\tilde B$ a lift with 
non-zero rotation, and $A \in  \GL(1,\bbC)$ an expanding
element. Then $\tilde A$ acts properly on $X_{\tilde B}$ 
if and only if $ \tilde A$ and $ \tilde B$ generate a lattice in
$\widetilde  \GL(1,\bbC)$. 
\begin{example} \label{ex:gentorush}
The quotient flat affine torus
$$     \cT_{\tilde A, \tilde B} =   \;  \langle \tilde A \rangle
 \,  \lmod   X_{\tilde B,k} \; , $$
is a homogeneous flat affine two-torus 
with holonomy in $\GL(1,\bbC)$.
\end{example}

\subsection{Affine cylinders with geodesic boundary} 
\label{sect:acylinders}
We show that by gluing flat affine cylinders whose boundary 
curves are incomplete geodesics 
we may construct  flat affine two-tori with development image $\bbAo$. This yields another construction of the manifolds $\cT_{A,B,k}$
which have been introduced in Example \ref{ex:gentorusnh}.
As a matter of fact, as a key step in the course of the proof of Theorem \ref{thm:classification},  we shall show that
all \emph{non-homogeneous}  flat affine tori 
may be obtained in this way.\\


\begin{example}[Affine cylinders with geodesic boundary]  \label{ex:acylinders2}
Let $\bar {\mathcal H}_{0}$ be  the closed upper
halfplane with the origin $0$ removed. If $A$ is an expansion then 
$$ \cC_{A} =    \;  \langle A \rangle
 \,  \lmod \bar {\mathcal H}_{0} $$
is  topologically an annulus with two boundary components, 
and it is also a flat affine manifold with (incomplete) geodesic boundary, see Figure  \ref{figure:affinecylinders} and Figure \ref{figure:affinecylinder3d}. We call  $\cC_{A}$ an 
affine cylinder.\footnote{Benoist \cite{Benoist_Ta} calls $\cC_{A}$ an annulus.}
\end{example} 

\begin{figure}[htbp] 
\begin{center}
\includegraphics[scale=0.6, clip=true, trim=  0cm 0.7cm 0cm 0.7cm ]{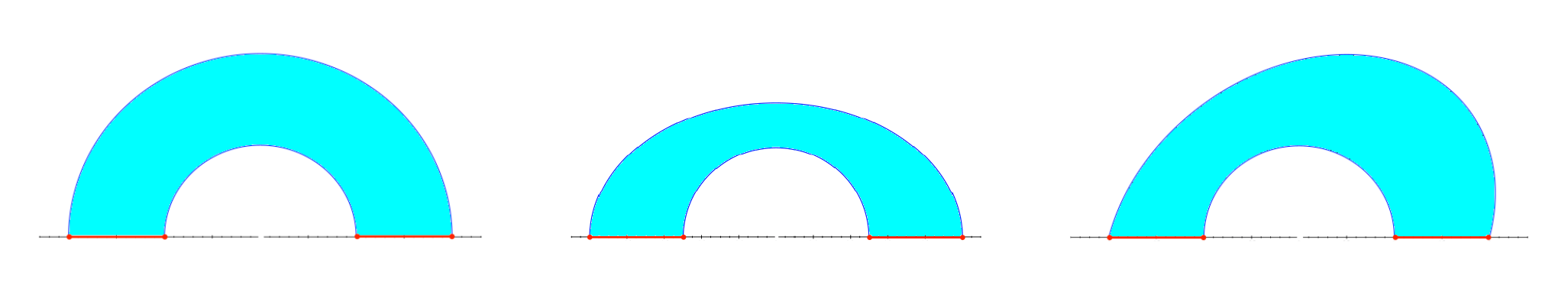}
\caption{{Gluing of flat affine cylinders $ \cC_{A} $.}}
 \label{figure:affinecylinders}
\end{center}
\end{figure}

More generally, let $\bar {\mathcal H}_{0}^k = \bigcup_{\ell = 0 \ldots k} 
{\bar \Omega}_{\ell}$, and $\tilde A$ the lift of $A$, which preserves the line
$\theta =0$.  Then the manifold with boundary 
$ \cC^k_{A} =    \;  \langle \tilde A \rangle
 \,  \lmod \bar {\mathcal H}_{0}^k$ is called an \emph{affine cylinder}. 

\begin{figure}[htbp] 
\begin{center}
\includegraphics[scale=0.2, clip=true, trim=  0cm 1.6cm 0cm 0.8cm]{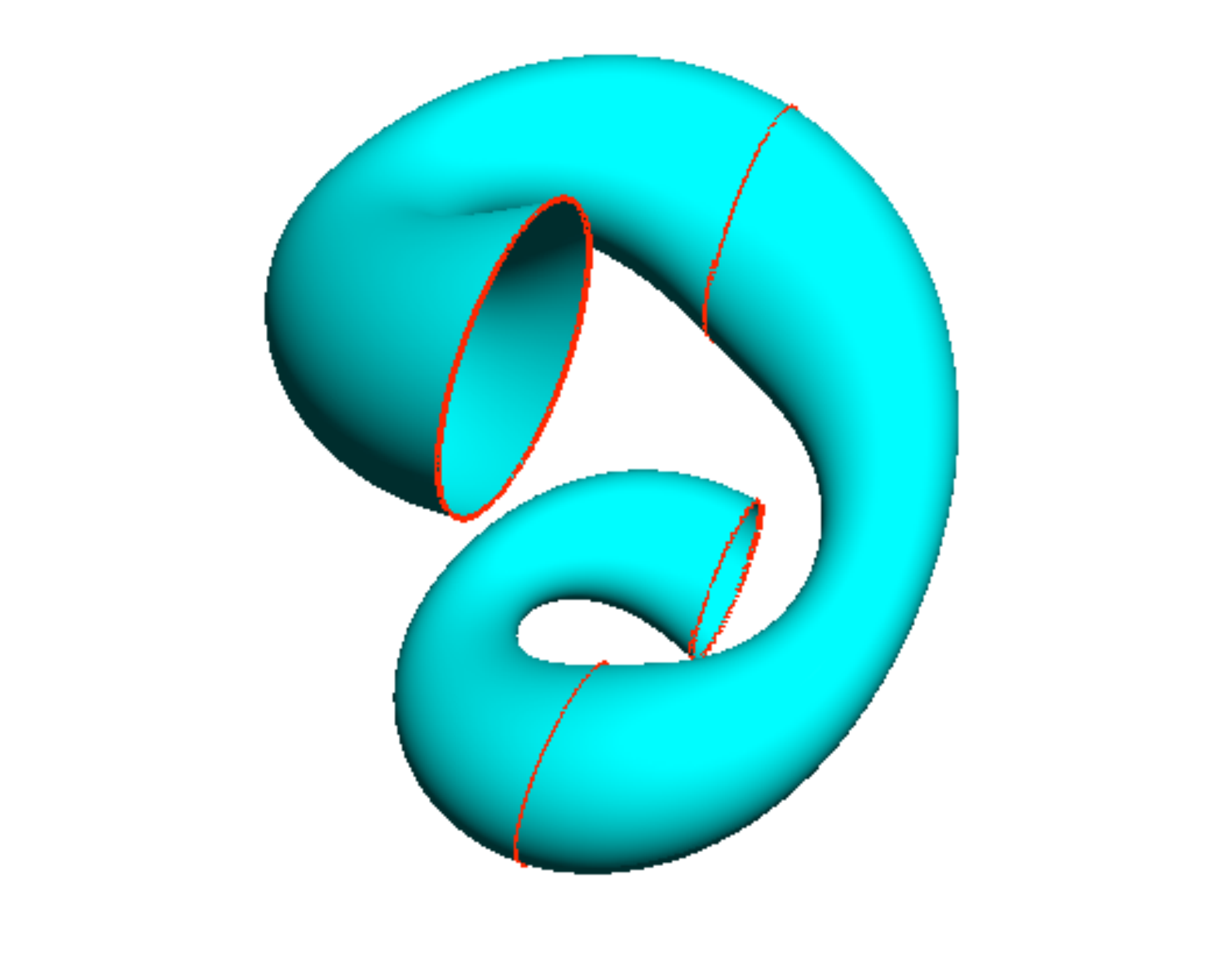}
\caption{{Affine cylinder $ \cC_{A}^3$ with geodesic boundary.}}
 \label{figure:affinecylinder3d}
\end{center}
\end{figure}

\subsubsection{Gluing of flat affine cylinders} \label{sect:TABk}
Let $A$ be an expansion and $B \in  \GL^+\!(2,\bbR)$, 
commuting with $A$, such that $B$ has positive eigenvalues.
Then, as shown by Example \ref{ex:gentorusnh},  
every  lift $\tilde B_{k}$, $k \geq 1$
acts properly on $   \langle \tilde A \rangle
 \,  \lmod  \widetilde \bbAo$  and yields the quotient flat affine torus
$$    \cT_{A,B,k}  =   \;  \langle \tilde A \rangle
 \,  \lmod   X_{B,k} \; . $$
Thus, geometrically, the flat affine torus  $\cT_{A,B,k}$ is constructed 
by gluing the flat affine cylinder $ \cC^k_{A}$ along its two boundary
geodesics using the transformation $\tilde B_{k}$,  see Figure \ref{figure:Gluedtorus}.

\begin{figure}[htbp] 
\begin{center}
\includegraphics[scale=0.2, clip=true, trim= 0 4cm 0 4cm]{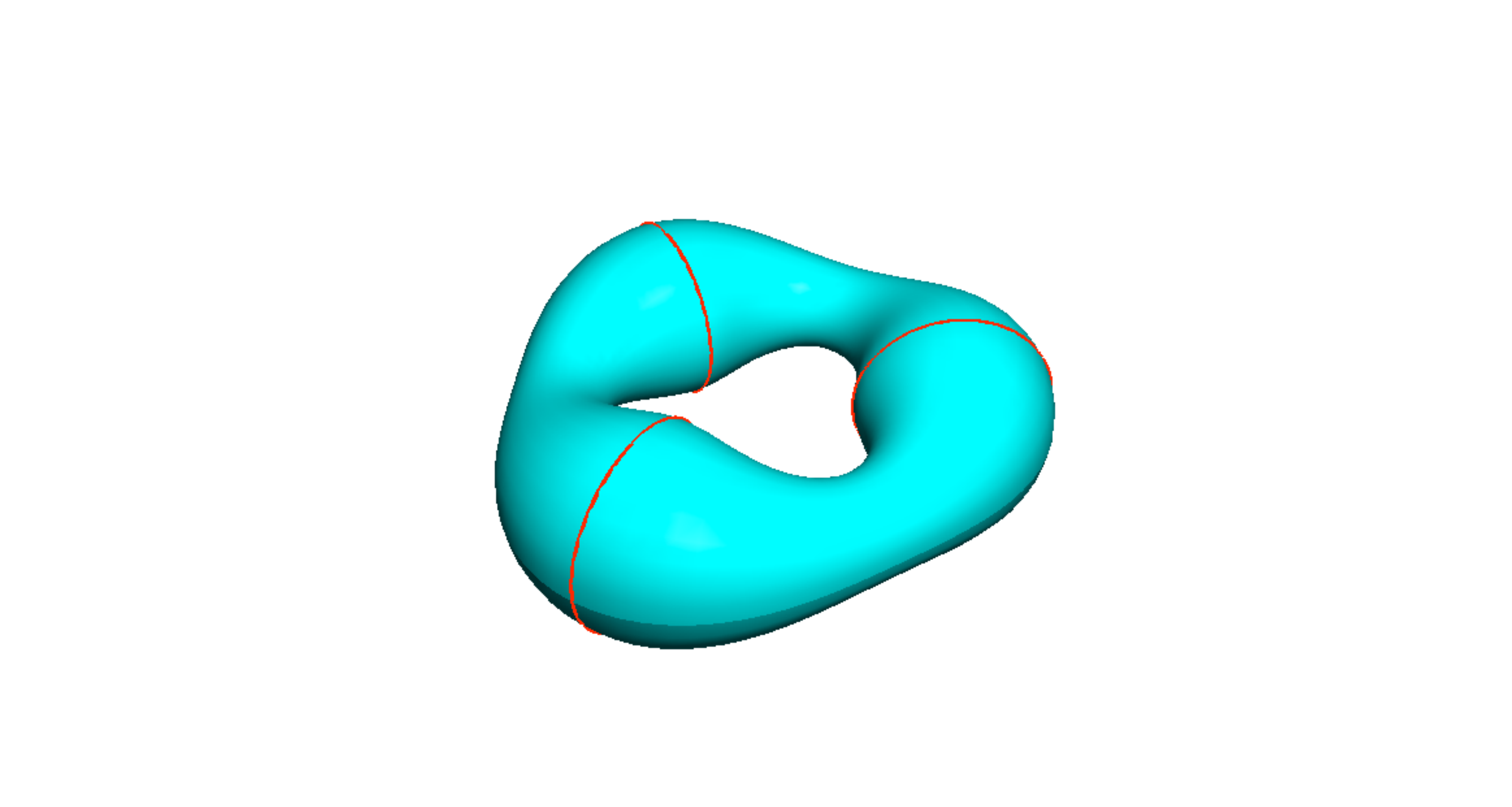}
\caption{{Gluing a flat affine torus $\cT_{A,B,3}$.}}
 \label{figure:Gluedtorus}
\end{center}
\end{figure}

\begin{remark}
Note that $M= \cT_{A,B,k}$ is a homogeneous flat affine two-torus if 
and only if its holonomy group 
$h(\Gamma) = \langle A , B \rangle$ is
contained in the group of dilations, in which case $\Aff(M)$ 
is a finite covering group of $\PSL(2,\bbR)$, and $M$ is
a Hopf torus. 
\end{remark}  

\begin{remark} Similarly, if $A$ is an expansion and $h(\Gamma) = \langle A , B \rangle$ is a discrete group of rank two then $h(\Gamma) $ acts properly
discontinuously and with compact quotient either on $\cH$ or on
$\mathcal Q$. The corresponding
torus is obtained by gluing $\cT_{A}$ or $\cT_{A, \alpha}$, $\alpha < \pi$. 
\end{remark}  

Finally, we remark that a homogeneous flat affine torus with \index{flat affine torus!homogeneous} 
development image $\bbAo$ may be constructed by gluing
cylinders if and only if it admits a closed (non-complete) geodesic: 

\begin{example}[Homogeneous tori with a  closed geodesic]  \label{ex:cylinders2}
Let $A_{\lambda}$ be a dilation, and $B \in \GL(1,\bbC)$. Then every 
lift of $B$ to $\widetilde \bbAo$  is of the form $$\tilde B_{k}: \;  (r, \theta) \mapsto 
( \lambda r, \theta + \alpha) \; ,  \;  \alpha = \alpha_{0} + 2k \pi$$
with $\alpha_{0} \in [0, 2 \pi)$.  If  $\alpha \neq 0$, we define $
  X_{B,k} =   \;  {\langle \tilde B_{k} \rangle}
 \,  \lmod \, \widetilde  \bbAo \,  , \;  \,   $  and 
$$     \cT^h_{\lambda,B, k}  =   \;  \langle \tilde A_{\lambda} \rangle
 \,  \lmod   X_{B,k} \; . $$
Let $\bar {\mathcal H}_{\alpha}  = \{ (r, \theta) \mid 0 \leq \theta \leq \alpha \}$
be a strip in $\widetilde \bbAo$, and 
$$ \cC_{\lambda, \alpha} =    \;  \langle \tilde A_{\lambda} \rangle
 \,  \lmod \bar {\mathcal H}_{\alpha} \; . $$ 
Then,  $\cT^h_{\lambda,B, k}$ is a homogeneous flat affine two-torus which
is obtained by gluing the cylinder $\cC_{\lambda, \alpha}$ with $\tilde B_{k}$.  
 \end{example} 
 
 \begin{figure}[htbp] 
\begin{center}
\includegraphics[scale=0.6, clip=true, trim=  0cm 1.15cm 0cm 0.3cm ]{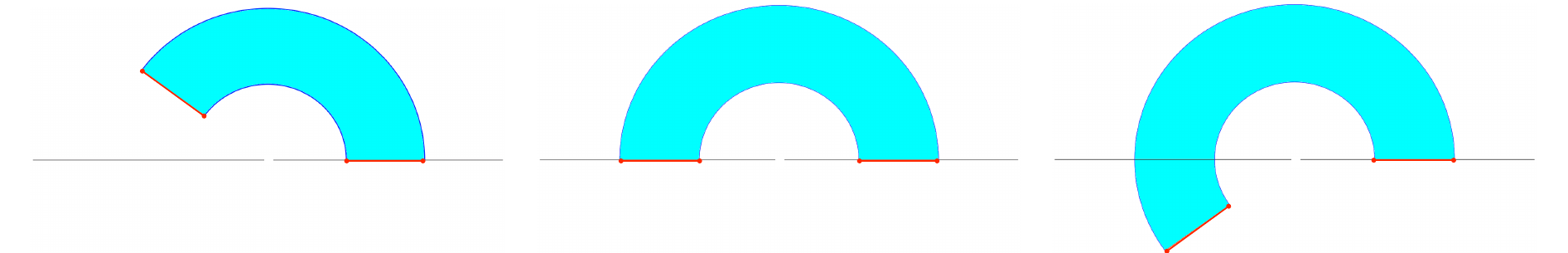}
\caption{{Affine cylinders $\cC_{\lambda, \alpha}$}, $\alpha <2\pi$.}
 \label{figure:genaffinecylinders}
\end{center}
\end{figure}


\section{The classification of flat affine structures on the two-torus}
\label{sect:classification}

The classification of flat affine structures on the two-torus was carried
out by Kuiper \cite{Kuiper} and completed by Nagano-Yagi in 
\cite{NaganoYagi}, and independently also by Furness and Arrowsmith 
in \cite{FurArr}. Later on much of the work in \cite{NaganoYagi} was clarified 
and beautifully generalised in Benoist's paper \cite{Benoist_Np}.
In this section we describe the classification result 
in detail and explain its proof, following loosely along 
the lines of \cite{NaganoYagi},  and 
employing also the main ideas from \cite{Benoist_Np,Benoist_Ta}
to establish in Proposition \ref{prop:Deviscovering} the crucial fact that the development map of a flat affine two-torus is always a covering map onto its image.  \index{development maps!for flat affine two-torus}

\begin{theorem} \label{thm:classification}
Let $M$ be a flat affine two-torus. Then $M$ is affinely 
diffeomorphic to either
\begin{enumerate}
\item a quotient of a simply connected two-dimensional affine homogeneous domain
by a properly discontinuous group of affine transformations.
\item  or a quotient space of the universal covering $\widetilde \bbAo$ of the once punctured plane.
\end{enumerate}
In particular, the universal covering flat  affine manifold of $M$ is affinely 
diffeomorphic to the affine plane $\bbA^2$,  the half-plane ${\cal H}$, or
the sector (quarter plane) ${\cal Q}$, in the first case, and to $\widetilde \bbAo$, in
the second case. 
\end{theorem}

%
The first step in the proof of Theorem \ref{thm:classification}
consists of the determination of the 
open domains in $\bbA^2$ which appear as 
the development images of flat affine structures on the two-torus.
This is done in  \S  
\ref{sect:devimages}. If $M$ is homogeneous
then the development map is a covering map. 
The main step is then the determination of the structure of flat affine 
two-tori with development image $\bbAo$, which are not homogeneous. 
This is carried out in
 \S  \ref{sect:univcovs}.  We prove that such tori
may be obtained by gluing affine cylinders with geodesic boundary.  We deduce that also in this case
the  development map of $M$ is a covering map onto its image.\\

The following further consequences are implied by the  theorem or its proof.
 
\paragraph{Classification of divisible affine domains}
An affine domain is called \emph{divisible} if it admits a discontinuous affine action with compact quotient. Since they admit a simply transitive abelian group, all development images of flat affine structures on the two-torus are divisible by abelian discrete groups (isomorphic to $\bbZ$ or $\bbZ^2$).  Conversely, by Benz\'ecri's theorem, every divisible plane affine domain is the development image of a flat affine structure on the two-torus. 
By Theorem \ref{thm:classification}, the universal covering of a flat affine two-torus is 
a homogeneous flat affine manifold, which covers a convex divisible 
homogeneous domain in $\bbA^2$.

\paragraph{The affine automorphism group of $M$}
\begin{enumerate} 
\item If  $M$ has  development image $\bbA^2$,  the half-plane 
${\cal H}$, or the sector $ {\cal Q}$, then $M$ is a homogeneous
flat affine manifold. The connected component $\Aff(M)^0$ of the group
$\Aff(M)$ of affine diffeomorphisms of $M$ is a two-dimensional compact abelian 
Lie group, acting transitively and freely on $M$.  
\item If the development image is the once-punctured plane $\bbAo$ and
$M$ is homogeneous, 
then the group $\Aff(M)^0$ is either a quotient of $\widetilde{\GL}^+\!\!(2,\bbR)$,
as is the case for Hopf tori (Examples \ref{ex:hopftori}, \ref{ex:finhopftori}), or 
$\Aff(M)^0$ is a quotient of $\GL(1,\bbC)$, as in Example \ref{ex:finquothopftori}. In either case, the action of $\GL(1,\bbC)$
on $\bbAo$ descends to a transitive and free action of a two-dimensional compact abelian Lie group on $M$.
\item 
Otherwise, $\Aff(M)^0$ is a two-dimensional abelian connected Lie group
which has a one-dimensional compact factor. In this case,  
$M$ is not homogeneous, as in Example \ref{ex:expholonomy}.
\item The affine automorphism group of $M$ acts prehomogeneously on $M$,
that is, it has only finitely many orbits on $M$.  
\item The one-dimensional orbits of $\Aff(M)^0$ are non-complete
geodesics in $M$ along which
$M$ may be cut into flat affine cylinders. 
\end{enumerate} 

\paragraph{Homogeneous and complete flat affine tori}

\begin{enumerate}

\item  Every \emph{homogeneous} flat affine two-torus $M$ is 
affinely diffeomorphic to a quotient of an abelian \'etale affine Lie group of type $\mathsf{T}$, $\mathsf{D}$, $\mathsf{C}_1$, 
$\mathsf{C}_2$, $\mathsf{B}$ or $\mathsf{A}$ as listed in Example \ref{ex:domains}.
\item  Every \emph{complete} flat affine two-torus $M$ is affinely diffeomorphic to a quotient of an abelian simply transitive affine Lie group of type $\mathsf{T}$ or $ \mathsf{D}$. In particular, $M$
is also a homogeneous  flat affine two-torus.
\end{enumerate}

\subsection{Development images} \label{sect:devimages}

The classification of development images is as follows:
\index{development image!of flat affine structure}
\begin{proposition}  \label{prop:devimages}
Let $M$ be a flat affine two-torus. 
Then the development image of $M$ is either 
the affine plane $\bbA^2$,  the half-plane ${\cal H}$, 
the sector (quarter plane) ${\cal Q}$ or the once-punctured plane  $\bbAo$,
respectively.  
\end{proposition}


\paragraph{Proof of Proposition \ref{prop:devimages}}
Let $h(\Gamma) \leq \Aff(2)$ be the holonomy 
group of $M$.  Let $N$ be the 
identity component of a maximal
abelian subgroup of $\Aff(2)$ which contains $h(\Gamma)$. 
Note that  $N$ contains the
identity  component of the Zariski-closure of $h(\Gamma)$.
Therefore, $h(\Gamma) \cap N$ 
is of finite index in $h(\Gamma)$. 
We let $\tilde N$ denote the universal covering group of $N$. The first observation is 
that \emph{$N$ acts on the development image}, and it has only finitely many orbits on $\bbA^2$.

\begin{proposition} \label{prop:liftofN}
The action of $N$ on $\bbA^2$ lifts via the development map 
to an action of $\tilde N$ on the universal covering flat affine manifold $\tilde M$
of $M$. Moreover, it follows  that
\begin{enumerate} 
\item $N$ acts on the development image $\Omega$ of $M$.
\item $N$ has only finitely many  orbits on $\Omega$.
\end{enumerate}
\end{proposition}
\begin{proof} Let $Y$ be an affine vector field on $\bbR^2$ which
is tangent to the action of $N$, and let $\bar Y$ denote its lift to $\tilde M$ via $D$. 
Since $h(\Gamma)$ commutes with $N$, the vector field
$\bar Y$ is $\Gamma$-invariant and projects to a vector field on 
$M$. Since $M$ is compact, the flow of $\bar Y$ is complete. 
Therefore, the action of $N$ integrates to an action of the universal covering
group $\tilde N$ on $\tilde M$, such that $D(\tilde n x) = n D(x)$, where $\tilde n \in \tilde N$ and $n = h(\tilde n)$ is its image in $N$. This implies (1). 

To prove (2), we note that, up to affine conjugacy, every maximal 
abelian and connected subgroup $N$ of $\Aff(2)$ is either one of the abelian groups 
$\mathsf{T}$, $\mathsf{D}$,  $\mathsf{C_{1}}, \mathsf{C_{2}}$,  $\mathsf{B}$, 
$\mathsf{A}$ (as listed in Example 
\ref{ex:domains}),  or the group 
$$ \mathsf{N} = \left\{ 
\begin{matrix}{ccc} 1 & u & v   \\  0 & 1 & 0 \\ 
0 &  0 & 1  
\end{matrix} \Mid  u,v \in \bbR
\right\}   \; . $$ 
All of the groups appearing in Example 
\ref{ex:domains} are simply transitive on an affine domain,
and they have finitely many orbits on $\bbA^2$. This shows (2).
  
To complete the proof, 
we contend that $h(\Gamma)$ is \emph{not} contained  in the group $ \mathsf{N}$. 
We remark that the orbits of $\mathsf{N}$ are the horizontal lines on $\bbR^2$. 
If $h(\Gamma)$ is contained in $\mathsf{N}$, then by (1), $\Omega$ is a 
union of orbits. Thus horizontal lines define a one-dimensional foliation of 
$\Omega$, which is preserved by $h(\Gamma)$. This, in turn, defines a 
foliation on the manifold $M$, which has an open subset of the real line 
as its space of leaves. This is not possible, since $M$ is compact: 
The space of leaves is a quotient of $M$, and therefore it is a compact and 
closed subset of the line as well. 
\end{proof}

It follows that the development image $\Omega$ is a finite union of orbits 
of one of the connected abelian groups listed in Example 
\ref{ex:domains}. Since $\Omega$ is a connected open subset
of $\bbA^{2}$, as well, it follows that $\Omega$ must be one of the 
domains listed in  Proposition \ref{prop:devimages}. 
Conversely, as follows from \S \ref{sect:etale_affine}, each of these
domains appears as the development image of a homogeneous
flat affine structure on the torus. This completes the proof of 
Proposition \ref{prop:devimages}.

\subsection{The classification of manifolds modeled on $(\bbAo, \GL(2, \bbR))$}
\label{sect:bbaoclass}


As follows from the proof of Theorem \ref{thm:classification}, 
every compact manifold $M$ modeled on $(\bbAo, \GL(2, \bbR))$ 
is either complete and the development image of $M$ is $\bbAo$,
or $M$ is (isomorphic to) a quotient of the open quadrant $\mathcal Q$, or a quotient of an open 
half space $\mathcal H$. In the first case
$$ M =   \;  _{\mbox{$\Gamma$}} \,  \lmod \,  \tbbAo \; , $$ 
where $\Gamma \leq \widetilde \GL^+ \! (2,\bbR)$ is a 
discontinuous subgroup, and in the second case 
the affine holonomy group $\Gamma$ is a discrete subgroup of  
$\mathsf{B} \leq \GL(2, \bbR)$, $\mathsf{C}_{1} \leq \GL(2, \bbR)$
respectively. \\

We arrive at the following classification theorem for 
 $(\bbAo, \GL(2, \bbR))$-manifolds which are complete (see also Corollary \ref{cor:inhomog}):

\begin{theorem} Let $M$ be a compact complete $(\bbAo,\GL(2, \bbR))$-manifold. If  $M$ is not homogeneous then it is 
isomorphic to a torus $\cT_{\tilde A, \tilde B}$, as constructed in \S \ref{sect:TABk} and Example \ref{ex:gentorusnh}.
Moreover,  
$M$ is homogeneous
if and only if it can be modeled on  $(\bbAo, \GL(1, \bbC))$.
\end{theorem}

In particular, if $M$  
is not homogeneous
then it is obtained by gluing flat affine cylinders  $ \cC^k_{A}$, where $A$ is an expansion and  $\tilde B \in \widetilde \GL^+ \! (2,\bbR)$ has non-zero angle of rotation and commutes with $A$. 
Furthermore if $A$ is a dilation, $B$ cannot be conjugate to an element of $\GL(1,\bbC)$. 

\begin{example}[Holonomy in $\GL^+(2,\bbR)$ is not injective] \label{ex:hol_notinj}
This phenomenon already occurs for homogeneous flat affine manifolds which are quotients of the universal covering 
flat affine Lie group 
$\tilde {\mathsf{A}}$ of the \'etale flat affine group $\mathsf{A}
= \GL(1,\bbC) \leq \GL(2,\bbR)$. Here different lattices $\Gamma_{1}$ and
$\Gamma_{2}$ of $\mathsf{A}$ determine non-isomorphic 
flat affine manifolds. However, different lattices may project to the same holonomy group in $\mathsf{A}$. 

More striking examples arise as a consequence of 
the construction of the tori $\cT_{\tilde A, \tilde B}$, as constructed in \S \ref{sect:TABk}. In fact, for every non-complete homogeneous flat affine manifold modeled on 
$\mathsf{C}_{1}$ or $\mathsf{B}$ one can construct 
a \emph{non-homogeneous} $(\bbAo,\GL^+(2,\bbR))$-manifold 
which has the same holonomy group in $\GL^+(2,\bbR)$.  
These examples show in particular that the affine holonomy group does \emph{not} determine the development image. 
\end{example}

We can consider also the corresponding
$(\tbbAo,\widetilde \GL^+\! (2,\bbR))$-manifolds.
Here we have:
\begin{theorem} 
\label{thm:bbao_rigid}
All compact $(\tbbAo,\widetilde \GL^+\! (2,\bbR))$-manifolds
are determined up to isomorphism by their holonomy group 
in $\widetilde \GL^+ \! (2,\bbR)$. 
\end{theorem}
\begin{proof} For the  complete manifolds the rigidity is shown 
in Example \ref{ex:toprig}. In particular, complete
and non-complete manifolds do not share the same holonomy
in $\widetilde \GL^+ \! (2,\bbR)$ (although they often do in $ \GL^+(2,\bbR)$).  In fact, the holonomy group of every complete manifold has an element with non-zero angle of rotation (cf.\ Definition \ref{def:posrot}), which is not possible if the
development image is one of the domains $\mathcal H$,
$\mathcal Q$. Similarly the non-complete examples are 
lattice quotients of a simply connected abelian Lie group contained
in $\GL^+(2,\bbR)$ which acts simply transitively on some open
domain in $\bbA^2$. Therefore, these manifolds are determined
by their affine holonomy group. 
\end{proof}

Also the  following is an immediate consequence of the above classification result: 

\begin{corollary} Let $\Gamma \leq  \widetilde \GL^+ \! (2,\bbR)$ 
be a non-finite discrete subgroup which is acting properly on $\widetilde \bbAo$. Then one of the following hold:
\begin{enumerate}
\item $\Gamma$ is isomorphic to $\bbZ$.  Moreover, it is generated by an expansion, or it is generated by an element of non-zero rotation. 
\item  $\Gamma$ is isomorphic to $\bbZ^2$ and it is
conjugate to one of the subgroups constructed in Examples 
\ref{ex:gentorusnh} or \ref{ex:gentorush}.
\end{enumerate}
\end{corollary}

\subsection{The global model spaces} \label{sect:univcovs}
The classification theorem, Theorem \ref{thm:classification}, 
implies that there do exists four simply connected flat 
affine manifolds 
which appear as the universal covering space of a flat affine two-torus. 
These simply connected model spaces for
two-dimensional compact flat affine manifolds 
are the plane $\bbA^2$, the half-plane ${\cal H}$,
the sector ${\cal Q}$, and  $\widetilde{\bbAo}$, 
the universal covering space of the once punctured plane. \\
%

This follows from the classification of development images 
once the following fact is established. \index{development maps!for flat affine two-torus}

\begin{proposition}  \label{prop:Deviscovering} 
The development map of a flat affine two-torus $M$ is a covering map. 
\end{proposition}

The main step in the proof of Proposition \ref{prop:Deviscovering} relies on the 
decomposition of $M$ into fundamental pieces, which are 
called bricks. This concept is due to Benoist \cite{Benoist_Np}. 
The bricks in this case are flat affine cylinders with geodesic boundary.
This brick decomposition for flat affine two-tori resembles the pants 
decomposition for closed hyperbolic surfaces (see \cite{Ratcliffe}). 
\emph{It will also
serve us in the parametrisation of the deformation space in
\S \ref{sect:cart}. }
%
%

\subsubsection{The brick decomposition for flat affine two-tori}
Let $\tilde M$ be the universal covering flat affine manifold of $M$, 
$\Gamma \leq \Aff(\tilde{M})$ the 
group of covering transformations,  and
$D: \tilde M \ra \bbA^2$ the  development map. Let $N \leq \Aff(2)$ be
the identity component of a maximal abelian subgroup containing $h(\Gamma)$, 
as in \S \ref{sect:devimages}. 
By Proposition \ref{prop:liftofN}, the action of $N$ on the development image $D(\tilde M)$ lifts to an action of its universal covering 
$\tilde N$ on $\tilde M$, such that 
$D$ is an equivariant map $\tilde M \ra D(\tilde M)$. 

\paragraph{Homogeneous flat affine tori}
In case that  the development image $D(\tilde M)$ is  $\bbA^2$, or the half-plane ${ \cal H}$, or the sector ${\cal Q}$, $N$ is simply connected and
acts simply transitively on $D(\tilde M)$.
It follows that $\tilde N$ acts simply transitively on $\tilde M$, and
$D$ is an equivariant local diffeomorphism. Hence, 
$D: \tilde M \ra D(\tilde M)$  is an affine diffeomorphism. 
If the development image is $\bbAo$, and $N$ is conjugate to $\GL(1,\bbC)$ 
then $\tilde N$ acts simply transitively on $\tilde M$. It follows 
that $D$ is a covering map. We thus have an affine covering 
$$ \widetilde \bbAo \; \, \stackrel{D}{\longrightarrow} \; \, \bbAo \; \;  . $$
This proves that the development map is a covering map 
for all homogeneous flat affine tori $M$.

\paragraph{Inhomogeneous flat affine tori}
We assume now 
that $M$ is \emph{not} homogeneous. Therefore, 
the development image of $M$ is $\bbAo$ and
$N$ is different from $\GL(1,\bbC)$.  Then $N$ equals either
the group of diagonal matrices with positive entries $\mathsf{B}$
or the group $\mathsf{C}_{1}$ (compare Example \ref{ex:domains}).  
The open orbits of $N$  on $\bbAo$ are the open quadrant in the case $N= \mathsf{B}$, or the open half space in the case $N = \mathsf{C}_{1}$. 
In particular,  in this case, $N$ does not act transitively on $\bbAo$, and therefore $M$ is not a homogeneous flat affine torus. However, 
the orbits of $N$ on $M$ decompose $M$ into finitely many pieces, 
the \emph{bricks},  from which $M$ is constructed.  

\begin{proposition}[Brick Lemma for the flat affine two-torus] \label{prop:bricklemma}
Let $\Omega= \tilde N \tilde x_{0}$ be an open orbit of $\tilde N$ on $\tilde M$,
and\/ $\bar \Omega$ the closure of\/ $\Omega$. Let\/ $\Gamma_{0} = 
\{ \gamma \in \Gamma \mid  \gamma \bar \Omega = \bar \Omega\}$. Then 
\begin{enumerate}
\item $D: \, \bar \Omega \, \ra  \,\bbAo$ is a diffeomorphism 
 onto its image.
\item $\bar \Omega/ \, \Gamma_{0}$ is 
a flat affine cylinder with geodesic boundary. 
\end{enumerate}
\end{proposition}
\begin{proof}  
Observe that $N = \tilde N$ is simply connected.
Put $x_{0}= D \tilde x_{0}$. It follows that $D: \tilde N \tilde x_{0} \ra N x_{0}$ is a diffeomorphism. The complement of $N \tilde x_{0}$ in its 
closure  $\overline{N \tilde x_{0}}$ consists of 
one-dimensional orbits for $N$, which are diffeomorphic to a ray in $\bbAo$. 
Since $D$ is a local diffeomorphism on $\tilde M$,  
$\overline{N \tilde x_{0}}$ has precisely two such orbits in its closure,
which map to their corresponding orbits in $\bbAo$, see Figure \ref{figure:strips}.
It follows that $D$ is injective on the closure  $\overline{N \tilde x_{0}}$ and, 
in fact, $ D:  \overline{N \tilde x_{0}}  \ra \bbAo$ is a diffeomorphism onto its image. 
This proves (1).

To prove (2), remark first that $\tilde N$ has at most finitely many orbits on
the compact manifold $M$. This implies that there are only finitely 
many orbits of $\Gamma \tilde N$  on $\tilde M$. Since $\Gamma$ acts
properly on $\tilde M$, it follows that 
every compact subset $\kappa$ of $\tilde M$ intersects only finitely many 
orbits of $\tilde N$. In particular, $\kappa$ intersects
only finitely many  components of $\Gamma \bar \Omega$. Therefore,
$\Gamma \bar \Omega$ is closed in $\tilde M$. Hence, $\bar \Omega$
projects to a compact subset in $M$.

We may assume (by replacing $M$ with  a finite covering manifold if necessary) 
that $h(\Gamma)$ is contained in $N$. 
Note then, if $\gamma \in \Gamma$ such that $\gamma \bar \Omega 
\cap \bar \Omega \neq \emptyset$ then $\gamma \in \Gamma_{0}$. In fact, since
$h(\gamma) \in N$, there exists $\tilde n \in \tilde N$ such that 
$h(\gamma \tilde n) = 1$. Since $D$ is a diffeomorphism on $\bar \Omega$
and $\gamma \tilde n \bar \Omega \cap \bar \Omega \neq \emptyset$, we conclude that   $\gamma \tilde n $ preserves both boundary components
of $\bar \Omega$. Thus $\gamma \tilde n \Omega = \Omega$,
and therefore $\gamma = n^{-1} \in \tilde N$. In particular, $\Gamma_{0} = \Gamma \cap 
\tilde N$. Moreover, it follows that $\bar \Omega/ \Gamma_{0}$  is the image of 
$\bar \Omega$ in $M = \tilde M/\Gamma$.

We thus proved that $\bar \Omega/ \Gamma_{0}$ is compact. 
Since $\bar \Omega$ has two boundary components which are
geodesic rays, $\bar \Omega/ \Gamma_{0}$ must be a flat affine cylinder
with geodesic circles as boundary components.  This proves (2).
\end{proof}

\begin{figure}[htbp] 
\begin{center}
\includegraphics[scale=0.6, clip=true, trim=  0cm 1.15cm 0cm 0.3cm ]{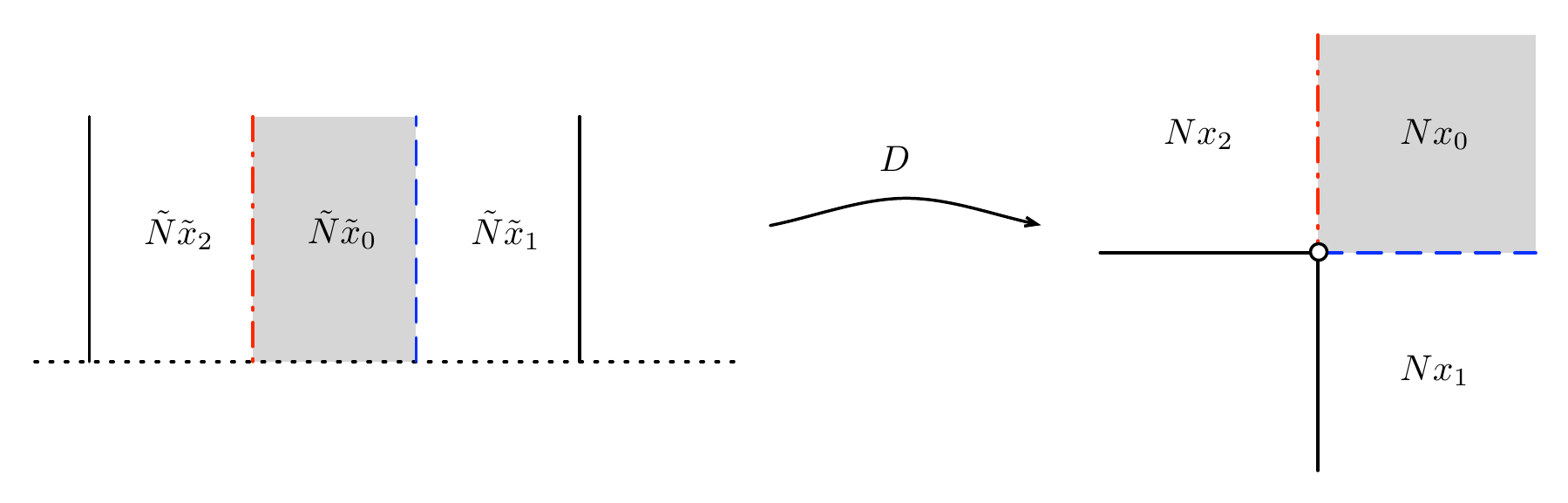}
\caption{{The orbits of $\tilde N$ in $\tilde M$.}}
 \label{figure:strips}
\end{center}
\end{figure}

The development image of $\bar \Omega$ is a closed half space 
or a sector. This implies:

\begin{proposition} If the development image is $\bbAo$ then 
the holonomy $h(\Gamma_{0})$ is generated by an expansion.
\end{proposition}
\begin{proof}
The proof of the previous proposition 
implies that $h(\Gamma_{0})$ is contained in $N$.
Since 
$\Gamma_{0}$ acts properly on $\overline{N \tilde x_{0}}$, and since $D$ is a
diffeomorphism onto its image $\overline{N x_{0}}$, 
it follows that $h(\Gamma_{0})$ acts properly and with compact quotient on the orbit closure $\overline{N x_{0}}$. 
This implies that $h(\Gamma_{0})$ has 
positive eigenvalues on one-dimensional orbits, and a fortiori, by properness
on $\overline{N x_{0}}$, it must be a group of expansions of $\GL(2,\bbR)$ (see Figures \ref{figure:nonexpanding}
and \ref{figure:expanding}). 
\end{proof}

\paragraph{Final step in the proof.} 
By the brick lemma (Proposition \ref{prop:bricklemma}), $M$ decomposes as 
a finite union of copies of a flat affine cylinder $\bar \Omega/ \Gamma_{0}$, which
are glued along their boundary geodesics. Therefore, there exists  a finite union 
$\bar {\mathcal{H}}$ of neighbouring copies of  $\bar \Omega$ in $\tilde M$, such that the torus $M$ is obtained by identifying the two boundary geodesics of $\bar {\mathcal H}/\Gamma_{0}$ by an affine 
transformation  $\tilde B$ in $\Aff(\tilde M)$.  
Since $h(\tilde B) \in \GL^+\!(2,\bbR) $ commutes with $N$, it follows that 
$B= h(\tilde B)$ is contained in $N$ or $B \in   \R_{\pi} N$. It follows that
the development image of $\bar {\mathcal H}$ must be $\bbAo$ or $\bar {\mathcal H}_{0}$, the closed half space with the origin removed,
respectively, see Figure \ref{figure:deviscov}.
Hence, $\bar {\mathcal H}$ is affinely 
diffeomorphic to one of the strips $\bar {\mathcal{H}}_{0}^k$,
which are defined in \S \ref{sect:acylinders}. Let $A \in  \GL^+\!(2,\bbR)$
be the expansion which generates $\Gamma_{0}$. Then $\bar {\mathcal H}/\Gamma_{0}$ is a flat affine cylinder $\mathcal C^k_{A}$, as defined in Example \ref{ex:acylinders2}. Therefore,
$M$ is affinely diffeomorphic to a flat affine torus  
${\mathcal T}_{A,B,k}$ constructed in Example 
\ref{ex:gentorusnh}. In particular, $M$ is affinely diffeomorphic to 
a quotient of $\widetilde \bbAo$, by a properly discontinuous subgroup 
$$ \Gamma = \langle \tilde A, \tilde B_{k} \rangle$$ of affine
transformations in $\Aff(\widetilde \bbAo)$.

\begin{figure}[htbp] 
\begin{center}
\includegraphics[scale=0.6, clip=true, trim=  0cm 0.8cm 0cm 0.4cm ]{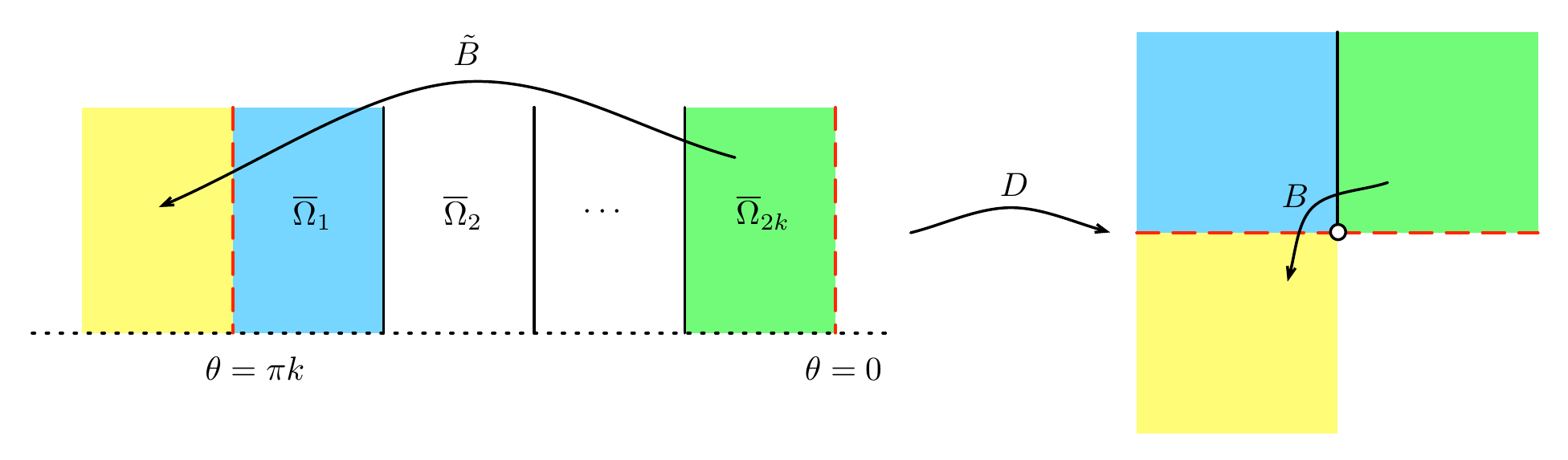}
\caption{{\bf $\bar {\mathcal H} = \bar \Omega_{1} \cup \cdots \cup \bar \Omega_{2k}$, $B = \R_{\pi}$}.}
 \label{figure:deviscov}
\end{center}
\end{figure}

\begin{corollary} \label{cor:inhomog}  \index{flat affine torus!non-homogeneous}
Every non-homogenous flat affine two-torus $M$ is affinely
diffeomorphic to a two-torus $\cT_{A,B,k}$ 
(see Example \ref{ex:gentorusnh}), and $M$ 
is obtained by gluing
a flat affine cylinder $\mathcal{C}^k_{A}$, where $A \in  \GL^+\!(2,\bbR)$
is an expansion, and $B \in  \GL^+\!(2,\bbR)$ has 
positive eigenvalues and commutes with $A$.
\end{corollary}

\section{The topology of the deformation space} \label{sect:thedefspace}

In this section we describe the global and local structure
of the deformation space $\Def(T^2,\bbA^2)$ of all flat affine
structures on the two-torus. The deformation
space decomposes into two \emph{overlapping}  
subsets: the open subspace $\Def(T^2,\tbbAo)$ of structures modeled on the once punctured plane $\tbbAo$ and
the closed subspace $\Def_{h}(T^2,\bbA^2)$ of homogeneous flat  affine structures.  We describe the structure and topology of these
two subspaces separately in Sections 6.3 to  6.4.  
In Section 6.5 we deduce our main result that the
holonomy map for the deformation space $\Def(T^2,\bbA^2)$
is a local homeomorphism. 

\subsection{Flat affine connections} \label{sect:affine_conns}
In this subsection we introduce flat affine connections. These
provide another point of view on flat affine structures,
which turns out to be particularly useful in the study of homogeneous flat affine manifolds. \\

An {\em affine connection\/} on the tangent bundle of $M$ is determined by a covariant differentiation
operation on vector fields which is
a $\bbR$-bilinear map
\begin{align*}
{\nabla}: \Vect(M) \times \Vect(M) & \longrightarrow \Vect(M) \\
(X,Y) & \longmapsto \nabla_X(Y)
\end{align*}
and for $f\in C^\infty(M)$,
satisfies
\begin{equation*}
\nabla_{fX} Y = f \nabla_{X} Y, \qquad \text{and } \qquad
\nabla_{X}(f Y) = f \nabla_{X} Y + (Xf) Y
\end{equation*}
(where $Xf\in C^\infty(M)$ denotes the directional derivative of 
$f$ with respect to $X$).
The connection is \emph{torsion free} if and only if, for all $X,Y \in \Vect(M)$, 
\begin{equation} \label{eq:torsion}
\nabla_X Y - \nabla_Y X = [X,Y]  \; , 
\end{equation}
and it is \emph{flat} if and only if the curvature tensor $R^\nabla$
vanishes. That is, if
\begin{equation}  \label{eq:curv}
R^\nabla(X,Y) = \nabla_X \nabla_Y - \nabla_Y \nabla_X - \nabla_{[X,Y]} = 0  \; .
\end{equation}

\subsubsection{Correspondence with flat affine structures} 
Specifying an affine structure on $M$ is equivalent 
to giving a torsion free flat affine connection on the tangent bundle
of $M$. Indeed, let $M$ be a flat affine manifold.  Then the affine structure 
defines a unique torsion free flat affine connection on $M$ by pulling back the canonical affine connection on $\bbA^n$
(that is, the usual derivative on $\bbR^n$) via a development map. Conversely, given any torsion free flat affine connection $\nabla$ on $M$, for each $p \in M$, the exponential map for $\nabla$ at $p$ 
is a connection preserving diffeomorphism from an open subset of the tangent vector space $T_{p} M$ (with the canonical flat affine connection) to a neighborhood  of $p$, compare \cite[VI. Theorem 7.2]{KN1}.  This gives rise to an atlas of locally affine coordinates and therefore determines a unique flat affine structure on $M$. 
Thus, there is a natural one to one correspondence  
\begin{equation} \label{eq:correspondac} 
   \Str(M,\bbA^n)  \, \longleftrightarrow \; \{ \text{torsion free flat affine connections on $M$} \}  
\end{equation}  
of the set of flat affine structures  $\Str(M,\bbA^n)$ with a set of affine connections.
An affine connection is called \emph{complete} if all of its geodesics can be extended to infinity. Under the correspondence \eqref{eq:correspondac} complete affine structures are in bijection with complete affine connections.
\\

Observe that the difference of two affine connections is a tensor field on $M$ and therefore the set of all affine connections forms an affine space. 

\begin{example}[Flat connections form a closed 
subset of an affine space] \label{ex:conn_affine_space}
Let $E$ denote the tangent bundle of the flat affine manifold 
$M$. Let $\nabla_{0}$ be the natural flat connection induced on $M$ by its flat affine structure. 
We choose $\nabla_{0}$ as a basepoint in the space of all affine connections on $M$. Every torsion free affine connection on $M$ is of
the form $\nabla = \nabla_{0} + S$, where $S \in \Gamma(S^2 E^*   \tensor{}  E)$ is a vector valued symmetric form on $M$.
The set of all torsion free affine connections $\nabla$ on $M$ is 
thus an affine space modeled on the vector space $\Gamma(S^2 E^*   \tensor{}  E)$.  Every torsion free \emph{flat}
affine connection on $M$ is of
the form $\nabla = \nabla_{0} + S$, where $S$ is contained 
in the closed subset $\cC$ of $\Gamma(S^2 E^*   \tensor{}  E)$
defined by the equation \eqref{eq:curv}, which encodes the vanishing of curvature. 
%
\end{example}


The space of sections 
$\Gamma(S^2 E^*   \tensor{}  E)$ carries the $C^\infty$-topology of maps. This defines
a topology on the space of torsion free affine connections. 

\begin{proposition} \label{prop:corres_is_homeo}
The natural correspondence \eqref{eq:correspondac} 
of flat affine structures with flat  torsion free affine connections is a homeomorphism. In particular,  the space of  flat affine structures $\Str(M,\bbA^n)$ is homeomorphic to the closed subset $\mathcal C$ in the tensor space $\Gamma(S^2 E^*   \tensor{}  E)$ as described above.
\end{proposition}
\begin{proof} Let us fix a flat affine structure on $M$ and let  $\nabla_{0}$ be its compatible torsion free flat affine connection. Let $\nabla = \nabla_{0}+S$ be another torsion free flat  affine connection on $M$. 
In a local flat affine coordinate chart for $M$, 
$S \in \cC$ is represented by a set of functions $\Gamma_{ij}^k$ which are called Christoffel symbols for $\nabla$, see \cite[III. Proposition 7.10]{KN1}. We observe
that the functions $\Gamma_{ij}^k$ also coincide with the coordinate representation of the tensor $S$. 
Therefore, a sequence $\nabla_{n}$ of affine connections is convergent if and only if the corresponding  Christoffel symbols converge in all local flat affine coordinate systems. 

Let $D_{n}$ be a sequence of development maps and 
consider the corresponding sequence of flat affine connections
$\nabla_{n}$. 
Since the Christoffel symbols for  $\nabla_{n}$ are polynomials in the first and second derivatives of $D_{n}$ (see \cite{KN1}), convergence of $D_{n}$ implies convergence of $\nabla_{n}$.
Therefore, the correspondence \eqref{eq:correspondac} is continuous. 

Conversely, for any torsion free flat affine connection $\nabla$ on $M$,  normal coordinate systems on $M$ define compatible  
coordinate charts for the flat affine structure defined by 
$\nabla$, see \cite[VI. Theorem 7.2]{KN1}.
Normal coordinate systems for $\nabla$ are determined by an ordinary differential equation whose solutions depend smoothly on the Christoffel symbols for $\nabla$. 
Hence, the flat affine coordinate charts for $\nabla$ 
depend smoothly on $\nabla$. This shows that the correspondence \eqref{eq:correspondac} is a homeomorphism. 
\end{proof}

%
\subsubsection{Translation invariant flat affine connections} \label{sect:transinv_cons}

An affine connection $\nabla$ on the two-torus $$ T^2 = S^1 \times S^1 $$ is called \emph{translation invariant} if the group $S^1 \times S^1$ acts by affine transformations.  Let $\nabla_{0}$ be the natural Riemannian flat affine connection on $T^2$.
It is characterized by the property that the translation vector fields of the $S^1 \times S^1$-actions are parallel. Then any other connection $\nabla=\nabla_{0} +S $ is translation invariant iff the covariant derivatives of the translation vector fields of the $S^1 \times S^1$-action are parallel with respect to $\nabla$. 
This condition is satisfied, if and only if the Christoffel symbols $S$ are constant functions in the flat coordinates for $\nabla_{0}$.
Therefore, the set of all translation invariant torsion free flat affine connection is in bijection with the subset $ \cC(T^2,\bbR)$ of $\cC$ which consists of all \emph{constant} (that is, of all $\nabla_{0}$-parallel) tensors contained in $\cC$. 

\begin{remark} Equation \eqref{eq:curv} shows that $\cC(T^2,\bbR)$ is a \emph{quadratic cone} in the vector space of symmetric bilinear maps  $S^2 \, \bbR^2 \tensor{} \bbR^2$.
Every element $S \in \cC(\bbR)$ 
represents a symmetric bilinear product  
 $$ \cdot_{\nabla}: \, \bbR^2 \times \bbR^2 \ra \bbR^2 \, , \; \,  u \cdot_{\nabla}  v := \, S(u,v)  \: , $$ 
which, for all $u,v,w$, satisfies  the associativity relation 
$$  (u \cdot_{\nabla} v) \cdot_{\nabla} w =  u \cdot_{\nabla} (v\cdot_{\nabla} w) \; . $$
This product defines a left-invariant flat affine connection on the abelian Lie group $\bbR^2$ by extending the covariant derivative from left-invariant vector fields to all vector fields. 
Indeed, there is a general  correspondence of \emph{associative}, and more generally \emph{left-symmetric} algebra products with left-invariant torsion free flat affine connections on Lie groups, see for example \cite[\S 5.1]{BauesPSR}. Under this correspondence
complete connections are represented by products which have 
the property that all maps $v \mapsto u \cdot_{\nabla} v$ have trace zero (compare \cite[Corollary 5.7]{BauesPSR}). 
\end{remark}
We summarize this discussion by the following: 

\begin{corollary} 
\begin{enumerate}
\item The set of all translation invariant flat affine connections on $T^2$ is homeomorphic to a four-dimensional 
homogeneous quadratic cone $\cC(T^2, \bbR)$ in the six-dimensional vector space $S^2 \bbR^2 \tensor{} \bbR^2$. 
\item 
The subset of \emph{complete} translation invariant flat affine structures on $T^2$ is homeomorphic to a two-dimensional homogeneous quadratic cone in the vector space $\bbR^4$.
\end{enumerate}
\end{corollary} 

In particular, one can deduce from (2) that the set of complete  translation invariant flat affine connections on the two-torus is homeomorphic to $\bbR^2$. In  view of Lemma \ref{lem:transhomotopic}, this 
gives yet another proof of the fact (cf.\ Example \ref{ex:completeas}) that the deformation space of complete affine structures on the two-torus is homeomorphic to $\bbR^2$.

\subsection{Translation invariant flat affine structures} \label{sect:homog_structs}

The usual representation of the two-torus $ T^2 = \bbR^2 / \bbZ^2$  as a quotient of the vector group $\bbR^2$ by its integral lattice tacitly induces various extra structures.
The translation action of the vector space $\bbR^2$
gives a simply transitive action of the abelian Lie group $S^1 \times S^1$ and the vector space structure on $\bbR^2$ descends to a compact abelian Lie group  structure
on $T^2$. Similarly, the ordinary flat affine structure on $\bbR^2$ induces the natural Riemannian flat affine structure on $T^2$
which is invariant by the translation
group $S^1 \times S^1$. 
%

\begin{definition} A flat affine structure on $T^2 = S^1 \times S^1$ is called {\em translation invariant} if the group $S^1 \times S^1$ acts by affine transformations.
\end{definition} 
A  translation invariant flat affine structure is thus compatible with the Lie group structure on $T^2$. In partciular, every flat affine torus with translation invariant flat affine structure is also a homogeneous flat affine torus. 
Note that the set of all  
translation invariant flat affine structures $\cT(T^2,\bbA^2)$
corresponds to the set of translation invariant flat affine connections $\cC(T^2, \bbR)$ under the map \eqref{eq:correspondac}.

\subsubsection{Relation with homogeneous flat affine tori} 
\label{sect:homandtransinv_structs}
Let $M$ be a homogeneous flat affine two-torus, and $\Aff(M)_{0}$ the identity component of its affine automorphism 
group.  By the classification Theorem \ref{thm:classification},  
the following two cases occur: \begin{enumerate}
\item 
Either the Lie group $\Aff(M)_{0}$ is isomorphic to $S^1 \times S^1$ and it develops to an 
action of an affine Lie group 
as listed in Example \ref{ex:domains}. 
\item Or $M$ is affinely diffeomorphic to a Hopf torus. In this case, $\Aff(M)_{0}$ contains a simply transitive group isomorphic to $S^1 \times S^1$, and this subgroup
develops to the action of an affine Lie group of 
type $\mathsf{A}$. 
\end{enumerate}
In particular,  for every homogeneous flat affine two-torus $M$,
the identity 
component $\Aff(M)_{0}$ of the affine 
automorphism group of $M$ contains a two-dimensional compact abelian Lie group,  which acts transitively and freely on $M$. 
This shows that every homogeneous flat affine two-torus is affinely diffeomorphic to a translation invariant flat affine torus.
\\

Recall that the diffeomorphism group $\Diff(T^2)$ acts 
on the set of all flat affine structures, and two flat affine structures on $T^2$ are called homotopic if they are equivalent by a diffeomorphism of $T^2$ which is homotopic to the identity.  \index{flat affine torus!homogeneous}
 \index{flat affine torus!translation invariant}
\begin{lemma} \label{lem:transhomotopic}
Every homogeneous flat affine two-torus is 
homotopic (isotopic) to a unique translation invariant flat 
affine two-torus.
\end{lemma}
\begin{proof} Let $(f,M)$ be a marked homogeneous flat affine two-torus, where $f: T^2 \ra M$ is a diffeomorphism. By the above remarks, we may choose a Lie subgroup $\cA$ of $\Aff(M)_{0}$ which acts simply transitively on $M$. The subgroup $\cA$ is unique  up to conjugacy in $\Aff(M)_{0}$. 
We also choose a basepoint $m_{0} \in M$. 
This fixes the structure of a compact abelian Lie group on $M$ 
which is isomorphic to $\cA$. Then there exists a \emph{unique} isomorphism of Lie groups $\phi: T^2 \ra M$ such that $\phi^{-1} \circ f$ is homotopic to the identity of $T^2$. In other words, 
$(\phi,M)$ and $(f,M)$ are equivalent markings (cf.\ \S \ref{sect:markedstructures}). By construction, the affine structure on $T^2$ induced by $\phi$ is translation invariant, and it is homotopic to the original homogeneous structure on $T^2$, which is induced by $(f,M)$. It is also independent of the choice of basepoint since $\cA$ acts transitively on $M$ (compare also Example \ref{ex:hom_structures}). Neither does it depend on the choice of the subgroup $\cA$ in $\Aff(M)_{0}$, since the conjugacy class of $\cA$ in $\Aff(M)_{0}$ is uniquely determined.  In particular, this argument implies that every two translation invariant structures which are homotopic do coincide. This shows
uniqueness.
\end{proof} 

The Lemma asserts that every orbit of the identity component $\Diff_{0}(T^2)$ of the group of all diffeomorphisms $\Diff(T^2)$ acting on homogeneous flat affine
structures intersects the subset of translation invariant 
structures in precisely a single point. The proof also shows that
on the subset of marked homogeneous tori which are in the complement of Hopf tori, we have a continuous projection 
onto translation invariant tori.  This proves that outside the 
Hopf tori the subset of translation invariant flat affine structures on $T^2$ defines a slice for the action of $\Diff_{0}(T^2)$ on the set of all homogeneous flat affine structures.

\subsubsection{Translation invariant development maps} \label{sect:devsection}
We construct an explicit \emph{continuous section} from the set of  translation invariant flat affine structures to development maps.
More specifically, we construct a continuous map 
\begin{equation} \label{eq:LSAsection} 
  \cT(T^2,\bbA^2)   =  \cC(T^2,{\bbR}) 
\; \stackrel{\mathcal E}{\longrightarrow} \;  \Dev(T^2,\bbA^2) \; ,\; \, 
S \mapsto D_{S} 
\end{equation} 
such that the development map $D_{S}$ defines  an affine structure on $T^2$ which has associated flat affine connection $\nabla = \nabla_{0}+ S$.
The construction is based on the relation of translation invariant flat affine structures
with the set 
of commutative associative algebra products on $\bbR^2$ as follows:

\begin{example}[Associated \'etale affine representation] \label{ex:developing_section}  \index{etale@\'etale representation!affine}
For $S \in \cC(T^2, \bbR)$, and $v \in \bbR^2$ we define 
an element $$ \bar \rho(v) = 
\begin{matrix}{cc} S(v,  \cdot ) &  v \\ 
0 & 0 
\end{matrix} \in \mathfrak{aff}(2)
$$ 
of the Lie algebra $\mathfrak{aff}(2)$ of the 
affine group $\Aff(2)$. In fact, 
the map $v \mapsto \bar \rho(v)$ is a Lie algebra 
homomorphism and the associated homomorphism of Lie groups 
$$ \rho = \rho_{S}: \bbR^2 \lra \Aff(2) \, , \; \, v \mapsto \, \rho(v) = \exp \bar \rho(v)$$
defines an affine representation of the 
Lie group $\bbR^2$ on $\bbA^2$ which is  \'etale in $0 \in \bbA^2$ 
(cf.\ Definition \ref{def:etale} and also the discussion in \cite[\S 2.1]{BC_1}).
\end{example}
Let $\nabla$ be the translation invariant flat affine connection
on $\bbR^2$ which is represented by $S \in \cC(T^2, \bbR)$.
The orbit map of the  \'etale representation  $\rho_{S}$  is  
$$ D_{S} = o_{S}: \bbR^2 \lra \bbA^2 \, , \; \, v \mapsto 
 \rho_{S} \cdot 0 \; $$ 
and it  is a development map for a translation invariant flat affine structure on $T^2 = \bbR^2 / \bbZ^2$ with associated affine connection $\nabla$.  (In fact, $D_{S}$ is also a \emph{frame preserving} development map. Compare Example \ref{ex:framepreserving} and \S \ref{sec:framedXG}.) 
Since, $D_{S}$ depends smoothly on $S$, $$ {\mathcal E}(S) = D_{S}$$  defines the required continuous section.\\ 

Note that the holonomy homomorphism $h = h^\nabla: \bbZ^2 \ra \Aff(2)$ for $D_{S}$ satisfies  $$ h^\nabla(\gamma) = \rho(\gamma) , \text{ for all $\gamma \in \bbZ^2$}  \; .$$  

We state without proof: 
\begin{lemma} The continuous map 
$  \cC(T^2, \bbR) \rightarrow \Hom(\bbZ^2,\Aff(2)$, $\nabla \mapsto h^\nabla$, 
is locally injective. 
\end{lemma} 


\subsection{The space of  structures modeled on $(\bbAo,\GL(2, \bbR))$}

%

Here we discuss in detail the subspace of the deformation space
of flat affine structures on the two-torus which consists 
of structures which have the once-punctured plane as
development image. Our main observation is that the holonomy map is a local homeomorphism on such structures.

\subsubsection{The holonomy map}
We consider first the subgeometry of structures 
which are modeled  on the universal covering of 
the once-punctured plane.
The topology of the deformation space of 
such structures 
is completely controlled by the holonomy map into
the space of conjugacy classes of homomorphisms
(the ``character variety''): 
 \index{holonomy map!is a local homeomorphism} 
 \index{holonomy map} 
 
\begin{theorem} \label{thm:tbbaotop}
The holonomy map  
$$ \Def(T^2, \tbbAo) \, \stackrel{hol}{\longrightarrow} \, 
\Hom(\bbZ^2, \tGL^+\!(2,\bbR))/ \tGL (2,\bbR)$$
embeds the deformation space homeomorphically as an open connected subset of  the character variety.
\end{theorem}
\begin{proof} Note, since $T^2$ is orientable, the holonomy takes values in $\tGL^+\!(2,\bbR)$. The map $hol$ is injective, by Theorem \ref{thm:bbao_rigid}. Since $hol$ is also continuous and
open, it is a homeomorphism onto an open subset. 
Connectedness of the deformation space 
will follow from the considerations 
in \S  \ref{sect:cart} below.
\end{proof}

Now we look at the deformation space of structures which
modeled on the once-punctured plane. For such structures the holonomy map is not injective, as we already remarked in Example \ref{ex:hol_notinj}. However, as we show now at least \emph{locally} the topology of the deformation space of $( \bbAo, \GL(2,\bbR))$ structures is fully controlled by the character variety:

\begin{corollary}   \label{cor:bbaotop}
The holonomy map 
$$  \Def(T^2, \bbAo) \, \stackrel{hol}{\lra} \, 
\Hom(\bbZ^2, \GL^+(2,\bbR))/ \GL(2,\bbR)$$
is a local homeomorphism onto its image, which is 
a connected open subset in the
character variety.  
\end{corollary}
\begin{proof} Since the subgeometry 
$$ (\tbbAo, \widetilde \GL(2,\bbR)) \lra (\bbAo, \GL(2,\bbR))$$ 
is a covering, the induced map on deformation spaces  (cf.\ \S \ref{sect:sub_geom}) 
$$  \Def(T^2, \tbbAo) \lra  \Def(T^2, \bbAo) $$
is a homeomorphism by Lemma \ref{lem:covering_geoms}. 
The commutative diagram \eqref{eq:subgeometry1} for the subgeometry  takes the form
\begin{align} \label{eq:subgeometry2}
 \xymatrix{
\; \; \Def(T^2, \tbbAo) \; \; \ar[d]_{\approx} \ar@{^{(}->}[r]^(0.35){hol}
& \; \;\Hom(\bbZ^2, \tGL^+\!(2,\bbR))/ \tGL\! (2,\bbR) \; \; \ar[d]  \\
\; \; \Def(T^2, \bbAo) \; \; \ar[r]^(0.35){hol}  &  \; \; 
 \Hom(\bbZ^2, \GL^+(2,\bbR))/  \GL(2,\bbR)  \,  .}  
\end{align}
Note that, by Theorem \ref{thm:tbbaotop}, the top horizontal map
is a topological embedding.  Furthermore, by Corollary \ref{cor:ind_cov2},  the right vertical map is a local homeomorphism. 
We deduce that the bottom
map $hol$ for $ \Def(T^2, \bbAo) $ 
is locally injective, and therefore it is a local homeomorphism onto an open subset.
\end{proof}

In the situation of Corollary \ref{cor:bbaotop}, 
all local topological properties of the deformation space 
are reflected in the character variety and also vice versa. 
For instance, singularities in the character variety give rise to singularities in the deformation space, as is the case in Example \ref{ex:defissing}. \emph{This shows that 
$ \Def(T^2, \bbAo) $ is not Hausdorff, and it is not even a $\mathrm{T}_{1}$-topological space.} \index{deformation space!non-closed point of}

\subsubsection{Cartography of the deformation space} \label{sect:cart}
Important strata in the   \index{deformation space!stratification of}
deformation space $$ \Def(T^2, \tbbAo)   $$  
arise from the orbit types of the action of 
$\widetilde \GL^+ \! (2,\bbR)$ on the image of
the holonomy map $$ \mathsf{hol}: \Dev(T^2, \tbbAo) \, \lra \, \Hom(\bbZ^2, \widetilde \GL^+ \! (2,\bbR)) \, \; .$$ 
We introduce several such strata and describe their topological relations with each other. We use this information to establish the connectedness of the deformation space.

\paragraph{Overview}
According to the classification theorem, tori which are modeled on $\tbbAo$ fall into three main classes distinguished by their development images. Namely the classes are formed by structures which have 
development image  equivalent to either
\begin{enumerate}
\item the once punctured plane $\tbbAo$ (``complete structures''), 
\item or a sector $\cQ$,
\item or the open half space $\cH$.
\end{enumerate}
The structures with development image $\tbbAo$ comprise 
non-homogeneous flat affine tori and the homogeneous
structures which arise from (lifts of) \'etale representations 
of type $\mathsf{A}$. 
The latter two strata arise from (the lifts of) \'etale representations 
of type $\mathsf{B}$ and $\mathsf{C}_{1}$ respectively
(see Example \ref{ex:domains} for notation). Therefore
all corresponding tori in these two strata are homogeneous. \\

 Another decomposition of the
deformation space is obtained by considering the subset 
$\mT$ of non-homogeneous structures and its complementary subspace $\Def_{h}(T^2, \tbbAo)$ consisting of homogeneous structures. The 
space of non-homogeneous structures can be decomposed 
into connected components parametrized by the
level of a non-homogeneous structure.
The subset of homogeneous 
structures is connected.  The subspace $\Def_{h}(T^2, \tbbAo)$ contains a two-dimensional
stratum $\mH$ of Hopf tori as a distinguished subset.  Non-homogenous structures are connected to homogeneous
ones only along the space $\mH$ of Hopf tori.

\paragraph{Hopf tori}  Recall that a torus which is modeled on 
$\tbbAo$ is called a \emph{Hopf torus} \index{Hopf torus}
if its holonomy is contained in the center of 
$\widetilde \GL^+\! (2,\bbR)$. The holonomy 
homomorphisms of marked Hopf tori form the closed subset of fixed points for the $\widetilde \GL(2,\bbR)$-conjugation action on the holonomy image of $\tbbAo$ - structures. Therefore, the Hopf tori form a closed subset  $$\mH \, \subset \; \, \Def(T^2, \tbbAo) \; . $$
\hspace{1cm} 

All Hopf tori are derived
from the \'etale affine representation of type $\mathsf{A}$
as follows.
Let 
$$ o: \bbR^2\,  \lra \, \bbAo\, , \; \, \,  (t,\theta) \,  \mapsto \, \exp(t) (\cos \theta, \sin \theta) $$
be the orbit map associated to the
representation $\mathsf{A}$. For $k_{1}, k_{2} \in \bbZ$ and 
$\lambda_{1},  \lambda_{2}>0$,  let 
 $\phi: \bbR^2 \ra \bbR^2$ be the linear map, which satisfies
$$ \phi(e_{1}) = (\log \lambda_{1}, k_{1}  \pi) \; , \; \phi(e_{2}) = (\log \lambda_{2}, k_{2}   \pi) \; . $$
Then development maps of the form 
$$  D = o \circ \phi : \bbR^2 \ra \bbAo $$  define 
a  two-parameter family of marked Hopf tori  
$$ \cH_{\lambda_{1},\lambda_{2},k_{1}, k_{2} }\, \; 
\, 
%
\scriptstyle  , \; \; 
(\log \lambda_{1}) k_{2} - (\log \lambda_{2}) k_{1} \, \neq \; 0  . $$ 
(See \S \ref{sect:lattices} for the general construction.)
The corresponding holonomy homomorphisms $h: \bbZ^2 \ra \Hom(\bbZ^2,\widetilde \GL(2,\bbR)) $  satisfy 
$$ h(e_{i}) = \, \mathrm{diag}(\lambda_{i}) \, \tau^{k_{i}} \in  \widetilde \GL^+ \! (2,\bbR) \, .$$
(Here $e_{1}, e_{2}$ denote generators of $\bbZ^2$, and $ \mathrm{diag}(\lambda) \in A$ the diagonal
matrix which has both diagonal entries
equal to  $\lambda$.) 
Observe that every marked Hopf torus is equivalent in the deformation space to precisely one of these tori. Forgetting about the marking, we note that every Hopf torus is affinely equivalent to a torus of the form $\cH_{\lambda_{1},\lambda_{2},k, 0}$, where we call 
$k = \mathrm{gcd}(k_{1}, k_{2}) \neq 0$ the \emph{level} of $\cH$.\\

For fixed $k_{1}, k_{2} \in \bbZ$, the set of all
 $\cH_{\lambda_{1},\lambda_{2},k_{1}, k_{2}}$ parametrizes
a closed (and also connected subset) of Hopf tori 
$\mH_{k_{1}, k_{2}}$, 
and the subset of all Hopf tori decomposes as 
$$  \mH \; = \bigcup_{(k_{1}, k_{2}) \,  \neq \, 0 } \mH_{k_{1}, k_{2}} \;  \subset  \; \,  \Def(T^2, \tbbAo) \, . 
$$

\paragraph{Non-homogeneous tori}  \index{flat affine torus!non-homogeneous}
We let $$ \mathfrak{T} \, \subset \; \, \Def(T^2, \tbbAo) $$ denote the subset of non-\-ho\-mog\-enous structures. Every marked manifold 
 $\cT$ which represents an element 
 of $\mT$ is equivalent as an $(\tbbAo,  \tGL^+\!(2,\bbR))$-\-manifold to a torus $$ \cT_{A,B,k} $$
 as constructed in Corollary \ref{cor:inhomog}.
Here $A \in \GL^+(2,\bbR)$ is  an expansion and $B\in \GL^+(2,\bbR)$ is upper triangular and commuting with $A$. 
 We call the number $k \in \bbZ - \{ 0 \}$ 
 the \emph{level} of $\cT$, respectively the level
 of its class in $\mT$. Since the level cannot be zero, 
 claim (2) of Lemma \ref{lem:tGL2con} implies that 
 \emph{the set  
 $\mathfrak{T}$ is indeed an open subset of the deformation space.} \\ 
 
 Let $\mT_k$ denote 
 the set of all elements in $\mT$ of level $k$.
We have the disjoint decomposition 
$$ \mT = \bigcup_{k \in \bbZ - \{ 0 \}} \mT_k \; \, .$$
Proposition \ref{prop:levclosed}
implies that all subsets $\mT_k$ 
and their complements  are closed subsets of $\mT$.

\paragraph{The closure of non-homogeneous tori}
We show now that the boundary of the set of non-homogeneous structures $\mT$
in the deformation space is formed by Hopf tori. 

\begin{proposition} \label{prop:nonhom_clos}
The closure of $\mT_{k}$ in 
$\Def(T^2, \tbbAo)$ is $\mT_{k} \cup \mH_{k}$.
\end{proposition}
\begin{proof} Note first that the Hopf tori $\cH_{\lambda_{1},\lambda_{2},k,0} \in \mH_{k}$, are in the closure of the elements $\cT_{A,B,k}$ of $\mT_{k}$. (Just deform $A$ and $B$ to dilations.) It remains 
to show that every homogeneous  torus $M_{o}$ in the closure of 
$\mT_{k}$ is a Hopf torus:  
Let $M_{\epsilon}$ equivalent to $\cT_{A,B,k}$ be a marked non-homogeneous torus of level $k \neq 0$ which is in the vicinity of $M_o$ in the deformation space. Let $h_{\epsilon}$ denote its holonomy homomorphism. We assume that $h_{\epsilon}$ converges to $h_{o}$ in the space of conjugacy classes of
homomorphisms. The holonomy group of $M_{\epsilon}$ is generated 
by $h_{\epsilon}(e_{i}) \in \widetilde\GL^+ \! (2,\bbR)$, $i = 1,2$.
The conjugacy class $\cC h_{\epsilon}(e_{i})$ is thus
in the vicinity of the class $\cC h_{o}(e_{i})$.  Proposition 
\ref{prop:levclosed} implies that the projections
$\mathsf{p} (h_o(e_{i})) \in \GL^{+}(2,\bbR)$
are conjugate to an element of $AN$. Since $M_{o}$ is 
homogeneous, $M_{o}$ is either a Hopf torus or $\lev h_o(e_{1}) = \lev h_o(e_{2}) = 0$ (in which case $M_{o}$ has development image $\mathcal{H}$ or $\mathcal{Q}$). In the latter case, if $M_{\epsilon}$ is close enough, we must
have $\lev h_\epsilon(e_{1}) = \lev h_\epsilon(e_{2}) =0$, 
again by Proposition 
\ref{prop:levclosed}. This contradicts the fact that the
level of the non-homogeneous torus $M_{\epsilon}$ is 
different from zero. Therefore,  $M_{o}$ is a Hopf torus.
\end{proof}

\paragraph{Homogeneous tori modeled on $\bbAo$}
The subset of homogeneous structures  
$$ \Def_{h}(T^2, \tbbAo) \subset \Def(T^2, \tbbAo) $$
decomposes into three strata $\mA$, $\mB$ and
 $\mC_{1}$, which are distinguished according to 
 the development images 
 $\tbbAo$, $\cQ$ and $\cH$ respectively. These structures arise from the \'etale affine representations of the abelian Lie group $\bbR^2$ 
 of type $\mathsf{A}$, $\mathsf{B}$ and $\mathsf{C}_{1}$
 respectively. Moreover, it follows (using the construction in \S \ref{sect:lattices}) that   
 the three strata are continuous images of homogeneous spaces
 via maps 
 $$ \GL(2,\bbR) / N  \lra \,  \Def(T^2, \tbbAo) \, , $$
 where $N$ describes the group of those automorphisms of 
 $\bbR^2$ which are induced by the conjugation action of the normalizers in $\GL(2,\bbR)$ for the groups 
 $\mathsf{A}$, $\mathsf{B}$ and $\mathsf{C}_{1}$. 
 (The normalizers  are listed in Lemma \ref{lem:normalizers}.) In particular, it follows that the strata $\mA$ and  $\mB$ are images of connected four-dimensional manifolds, while $\mC_{1}$  is a connected manifold of dimension three. 

Note that structures in $\mA$ may be continuously deformed to structures in $\mC_{1}$, as follows from (4) of Lemma \ref{lem:tGL2con}. Compare also Figure \ref{figure:deform2}. 
 Similarly structures in $\mB$ can be deformed to structures in
  $\mC_{1}$, see Figure \ref{figure:BC1}. This shows in particular that $ \Def_{h}(T^2, \tbbAo)$ is connected.
  
\paragraph{Connectedness}
The deformation space 
$$   \Def(T^2, \tbbAo) \, =  \;  \mT \, \cup \,  \Def_{h}(T^2, \tbbAo)$$
is connected. Indeed, by Proposition \ref{prop:nonhom_clos} 
every marked non-homogeneous flat affine torus in $ \mT$ can be
deformed to a Hopf torus contained in some $\mH_{k_{1},k_{2}}$. 
Since  $\mH_{k_{1},k_{2}}$ is a subset of the connected
space $ \Def(T^2, \tbbAo)_{h }$ it follows that $ \Def(T^2, \tbbAo)$
is connected.

\subsection{The subspace of homogeneous structures} 
\label{sect:def_homogeneous} \index{flat affine torus!homogeneous} 
We describe now the properties of the subset  $\Def_{h}(T^2,\bbA^2)$ of homogeneous flat affine structures as a subspace of the deformation space of all flat affine structures on $T^2$. 
Since the non-homogeneous structures form an open subset
the space $\Def_{h}(T^2,\bbA^2)$ is closed. The complement
of Hopf tori $$\Def_{h}(T^2,\bbA^2) - \mathfrak H$$ forms
a dense subset which is also open in the deformation space
of all structures  $\Def(T^2,\bbA^2)$.
We established in the previous subsections: 

\begin{proposition}
The set of all homogeneous structures $\Def_{h}(T^2,\bbA^2)$ is the continuous and injective image of the  quadratic cone $\cC(T^2,\bbR)$ under the map $\mathcal E$ in \eqref{eq:LSAsection}.
The map is a homeomorphism in the complement of Hopf tori. 
\end{proposition}

In particular:
\begin{corollary}
The deformation space $\Def_{h}(T^2,\bbA^2)$ of all homogeneous flat affine structures on the two-torus contains 
the complement of Hopf tori  
$$ \Def_{h}(T^2,\bbA^2) - \mathfrak H$$  as a dense open subset, which is a Hausdorff space and homeomorphic to a Zariski-open subset in a four-\-dimen\-sio\-nal quadratic cone in $\bbR^6$.
\end{corollary}

%
%
%
%
Note that the space of complete affine structures $\Def_{c}(T^2,\bbA^2)$ forms a two-dimensional closed subcone which is homeomorphic to $\bbR^2$, see \cite{BC_1}. 

\subsubsection{The action of the linear group on translation invariant structures  and conjugacy  of \'etale affine groups} 
\label{sect:orbit_closures}
The linear group $\GL({2},\bbR)$ naturally acts on the variety $\cC(T^2,\bbR)$ of commutative and associative algebra products on 
$\bbR^2$. The orbits of this action correspond to the isomorphism
classes of algebra products. Since the section map 
$$ {\mathcal E} :  \cC(T^2,\bbR) \, \lra \, \Def_{h}(T^2,\bbA^2)$$
is a continuous bijection this constructs a \emph{natural} induced 
action of $\GL(2,\bbR)$ on the deformation space of homogeneous structures $\Def_{h}(T^2,\bbA^2)$ which 
is continuous on the complement of Hopf tori.  This
group action may be used to reveal some of the 
topology of $\Def_{h}(T^2,\bbA^2)$
and the possible deformations of structures. \\

Recall the classification of  \'etale affine representations which is described in Section \ref{sect:etale_affine}.
Each orbit of $\GL({2},\bbR)$ in $\Def_{h}(T^2,\bbA^2)$ corresponds to exactly one of the 
affine conjugacy classes of abelian almost simply transitive groups of affine 
transformations on $\bbA^2$. We label the orbits accordingly with the symbols 
$\mA$, $\mB$, $\mC_{1}$, $\mC_{2}$, $\mD$ and $\mathsf{T}$. The decomposition 
of  $\Def_{h}(T^2)$ into the six orbit types of $\GL({2},\bbR)$ defines 
a natural stratification on $\Def_{h}(T^2)$ into manifolds which are
homogeneous spaces of $\GL(2,\bbR)$, and each orbit is 
a subcone of $\Def_{h}(T^2)$. Each such stratum may be also computed as the induced image of the subgeometry which is defined by the corresponding \'etale affine representation, see
the examples in Section \ref{sect:cart}, as well as Section \ref{sect:lattices} and Lemma \ref{lem:normalizers}.  \index{deformation space!stratification of}

The closure of each stratum 
consists of strata of lower dimensions and contains the unique closed stratum $\mathsf{T}$, which is a point. There are two open strata of dimension four labeled $\mA$ and $\mB$, which  correspond to homogeneous flat affine 
structures whose development images are the punctured plane, and
the sector respectively. In their closure are the three-dimensional 
orbits  $\mC_{1}$ and  $\mC_{2}$, whose corresponding flat affine structures
develop into the halfplane.  
The complete structures correspond to a
two dimensional orbit  $\mD$ and the translation structure $\mathsf{T}$.
We say that the orbit $\cO_{1}$ degenerates to the orbit $\cO_{2}$
if  $\cO_{2}$ is in the closure of $\cO_{1}$. Degeneration induces
a partial ordering on the strata of $\Def_{h}(T^2,\bbA^2)$, with the translation action
the unique minimal point.  By the theorem of Hilbert-Mumford if $\cO_{2}$
degenerates to $\cO_{1}$ then there exists a one-parameter group
$\lambda: \bbR \ra \GL(2,\bbR)$ such that $\lim_{t \ra 0} \lambda(t) o_{2} \in \cO_{2}$. Therefore, every degeneration may be constructed explicitly as a limit of a curve of flat affine structures in the stratum. Moreover, every point of $\Def_{h}(T^2,\bbA^2)$ directly degenerates to the translation structure,
compare, for example, Figure \ref{figure:qHopfTo trans}.
 \\

The graph shown in Figure \ref{fig:degenerations} describes all possible degenerations in the orbit stratification of $\Def_{h}(T^2,\bbA^2)$ with respect to the natural action of $\GL(2,\bbR)$.

\begin{figure}[htbp]

 $$ \xymatrix@1@R=-1ex@M=2ex{   \bullet \ar[r] &   \bullet \ar[r] &  \bullet \ar[r] & \bullet &  \bullet \ar[l]   &  \bullet \ar[l]  
\ar@(ul,ur) [llll]  \\
          	\mA &  \mC_{1}  & \mD  & \mathsf{T} &  \mC_{2} &   \mB   }$$
 \caption{Degenerations of $\GL({2},\bbR)$-orbit types in the deformation space.}
 \label{fig:degenerations}
\end{figure}
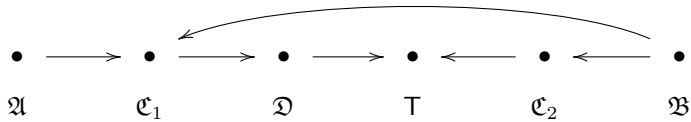

These degenerations are illustrated in Figures 1-5, Figure \ref{figure:qHopfTo trans} and Figure  \ref{figure:deform2}. 
\subsection{The deformation space of all flat affine structures on the two-torus} \label{sect:holislh}

Our main result is the following. \index{character variety} 
\index{deformation space!holonomy map} \index{holonomy map!for flat affine structures}
 \index{holonomy map!is a local homeomorphism} 
 \index{deformation space!of flat affine structures}
 
\begin{theorem} \label{thm:holonomy_main}
The holonomy map for flat affine 
structures on the two-torus 
$$ hol: \Def(T^2,\bbA^2) \ra \Hom(\bbZ^2, \Aff(2)) / \Aff(2) $$
is a local homeomorphism onto an open connected subset of  the character variety.
\end{theorem} 

\paragraph{The holonomy map for flat affine structures}
For the proof of Theorem \ref{thm:holonomy_main} we consider first the subgeometry of
$( \bbAo, \GL(2,\bbR))$-structures 
and its induced map on deformation spaces
 (cf.\ \S \ref{sect:sub_geom}):

\begin{proposition} \label{pro:Uo}
The induced map on deformation spaces $$  \Def(T^2, \bbAo) \,  \lra  \, \Def(T^2,\bbA^2)$$
is an embedding onto an open subset $\mU_{o}$ of  the space 
$\Def(T^2,\bbA^2)$. Moreover, the holonomy map for 
$\Def(T^2,\bbA^2)$ restricts to a local homeomorphism 
on this subset.
\end{proposition}
\begin{proof} The commutative diagram \eqref{eq:subgeometry1} for the subgeometry  takes the form
\begin{align*} 
 \xymatrix{
\; \; \Def(T^2, \bbAo) \; \; \ar[d]  \ar [r]^(0.35){hol}
& \; \;\Hom(\bbZ^2, \GL^+\!(2,\bbR))/ \GL\! (2,\bbR) \; \; \ar[d]  \\
\; \;  \Def(T^2,\bbA^2) \; \; \ar[r]^(0.35){hol}  &  \; \; 
\Hom(\bbZ^2, \Aff(2)) / \Aff(2)   \,  .}  
\end{align*}
The image of $ \Def(T^2, \bbAo)$ in $\Def(T^2,\bbA^2)$
consists of precisely those structures in $\Def(T^2,\bbA^2)$ 
whose linear part of the holonomy contains an expansion. Therefore, the image $\mU_{o}$ of the induced map is open in 
$\Def(T^2,\bbA^2)$. The left vertical map is
clearly injective. Note further that the operation of taking the
linear part of a homomorphism defines a continuous section of the right vertical map which is defined on the holonomy image $hol(\mU_{o})$. The latter set is open in $\Hom(\bbZ^2, \Aff(2)) / \Aff(2)$.  By Corollary \ref{cor:bbaotop}, the upper map $hol$ is a local homeomorphism.
Therefore, the right vertical map is a topological embedding, and
the lower map $hol$ is a local homeomorphism on $\mU_{o}$. \end{proof}

Below we construct an open neighborhood $\mU_{1}$ of the
translation structure $\mathsf{T} \in \Def(T^2,\bbA^2)$,  which
has the following properties: 
\begin{enumerate} 
\item $\mU_{1} \subset  \Def_{h}(T^2,\bbA^2)$ 
is contained in the subset of homogeneous flat affine structures, 
\item the restriction of the holonomy map $hol: \mU_{1} \ra \Hom(\bbZ^2, \Aff(2)) / \Aff(2)$ is injective,
\item  $\Def(T^2,\bbA^2) = \mU_{o} \cup \mU_{1}$.
\end{enumerate} 
Together with Proposition \ref{pro:Uo} this shows that
$$ hol:  \Def(T^2,\bbA^2) \ra \Hom(\bbZ^2, \Aff(2)) / \Aff(2)$$
is locally injective and therefore finishes the proof of
Theorem \ref{thm:holonomy_main}. \\

We observe the following refinement of Proposition
\ref{prop:nonhom_clos}: 

\begin{proposition} \label{prop:nonhom_clos2}
The closure of the subset $\mT$ of non-homogeneous structures in  the deformation space 
$\Def(T^2, \bbA^2)$ consists of Hopf tori. 
\end{proposition}
\begin{proof} Suppose there is a sequence  
of non-homogeneous marked tori $M_{i}$ which converge
in the deformation space  $\Def(T^2,\bbA^2)$ to a
flat affine torus $M$. Let $h_{i}: \bbZ^2 \ra \Aff(2)$ 
be their corresponding holonomy homomorphisms. 
By Corollary \ref{cor:inhomog}, we may assume that the 
linear parts of the $h_{i}$ are contained in the group of upper triangular 
matrices $AN \cup -AN$, where $AN$ is the index two subgroup
with positive diagonal entries.  Now if $M$ is homogeneous and
has development image different from the once-punctured 
plane, the linear parts of all $h_{i}$ are contained in $AN$ for sufficiently large $i$.
Let $D_{i}$ be a corresponding sequence of development maps for the $M_{i}$ which converges to a development map $D$ which represents $M$. 
Since $M$ is not modeled  on  the once-punctured plane the development map 
$D$ is injective. By Example \ref{ex:embeddings}, there is a neighborhood of $D$ 
in the space of development maps such that $D$ is injective on the fundamental
domain for the standard action of $\bbZ^2$ on $\bbR^2$. However, the development maps $D_{i}$ are not injective on this fundamental domain by construction of the non-homogeneous tori $M_{i}$, see Example \ref{ex:gentorusnh}. This contradicts the
fact that the $D_{i}$ converge to $D$.
The claim now follows from Proposition \ref{prop:nonhom_clos}. 
\end{proof}

Now the construction of $\mU_{1}$ goes as follows: 
Following the notation in Appendix A, define $U_{\epsilon}
= \{ g \in \tGL^+\!(2,\bbR) \mid |\theta(g)| < \epsilon \}$.
By the proof of
Proposition \ref{prop:ind_cov2}, we may choose 
an  open set  $$ \cU_{\epsilon} \subset \; \Hom(\bbZ^2, \tGL^+\!(2,\bbR))/ \tGL\! (2,\bbR) \, , $$ 
where 
$\cU_{\epsilon}$ is of the form  $\cC(U_{\epsilon} \times U_{\epsilon})$ such that the projection $$ \cU_{\epsilon} \, \lra \,  \Hom(\bbZ^2, \GL^+\!(2,\bbR))/ \GL\! (2,\bbR)$$ is injective. 
The holonomy preimage $hol^{-1}(\cU_{\epsilon})$ is a non-empty open subset of 
$ \Def(T^2, \tbbAo)$ which contains certain homogeneous structures of type $\mA$, and the strata $\mB$ and $\mC_{1}$. 
It corresponds to a non-empty
open subset $\cV_{\epsilon}$ of $\Def(T^2, \bbAo)$ such that the restriction 
$$ hol: \cV_{\epsilon} \, \lra \, \Hom(\bbZ^2, \GL^+\!(2,\bbR))/ \GL\! (2,\bbR)$$ is injective. Let $\cM = \Def(T^2, \bbAo) -  \cV_{\epsilon}$ be the complement (containing the non-homogeneous flat affine tori and also homogeneous structures of type $\mA$). Then we observe that $\cM$ is closed 
not only in $\Def(T^2, \bbAo)$, but also in $\Def(T^2, \bbA^2)$.
(Indeed, this follows since the closure of the space $\mT$ of non-homogeneous tori is contained in the space $\mH$ of Hopf tori.)
Now we put $\mU_{1} = \Def(T^2, \bbA^2) - \cM$.


%
%
%
%
%
%
%


\newpage 
\begin{appendix}

\section{Conjugacy classes in the universal covering group of $\GL(2,\bbR)$}
\label{sect:GL2R}
Let $\GL^+(2,\bbR)$ be the group of $2 \times 2$ matrices with positive determinant, and  let $$ \mathsf{p}: \widetilde \GL^+\!(2,\bbR) \ra 
\GL^+(2,\bbR)$$ be its universal covering group. 

\paragraph{Iwasawa decomposition}
Recall the Iwasawa decomposition
$$\GL^+(2,\bbR) =  {K}AN \;  ,  $$
where $K= \SO(2,\bbR)$ is the subgroup of rotations, $A$ is the group of diagonal matrices with positive 
entries, and $N$ the group of unipotent upper triangular matrices. 
Furthermore, we let $D$ be the central subgroup of $\GL^+(2,\bbR)$ contained in $A$ which consists of all elements of $A$ with 
identical diagonal entries.

Let $\tilde K \ra K$ be the universal covering of the
rotation group. 
There is an induced Iwasawa decomposition
$$  \widetilde \GL^+\!(2,\bbR) = \tilde{K} AN \; ,$$ 
where $A$ and $N$ are considered as subgroups of  $\widetilde \GL^+\!(2,\bbR)$.
%

Let $\mathcal Z$ be the subgroup of $\tilde{K}$, which is 
mapped by the covering projection  onto $\{+1, -1\} = \{ \E_{2} , \R_{\pi} \} \subset  \SO(2,\bbR)$. 
Note that the center of  $ \widetilde \GL^+\!(2,\bbR)$ consists 
of the subgroup $D$  extended by the group $\mathcal Z$.  
We choose a generator $\tau \in \tilde{K}$ for the infinite cyclic group 
$\mathcal Z$. Then $\mathsf{p}(\tau) = \R_{\pi} \, ( \, = - \E_{2}) $  and the element 
$\tau^2 \in \mathcal Z$ generates the kernel of the covering 
projection $\mathsf{p}$.

\paragraph{The rotation angle function}
We consider the diffeomorphism 
$\theta: \tilde K \ra \bbR$  which satisfies $\theta(1) = 0$ and the relation 
$$ \mathsf{p} = \begin{matrix}{rc} \cos \theta & \sin \theta \\
- \sin \theta & \cos \theta   \end{matrix}  \; . $$
Using the Iwasawa decomposition we construct an angular map 
$$ \theta:  \widetilde \GL^+\!(2,\bbR) \ra \bbR  \;  $$ 
by extending $\theta: \tilde K \ra \bbR$. This means that  
 $\theta(g) = \theta(k)$, where $g \in \widetilde \GL^+\!(2,\bbR)$
 has decomposition $g= kan$.  
 Geometrically, $\theta(g)$ thus is  the angle of rotation or polar angle for the image 
$\mathsf{p}(g) (e_{1})$ of the first standard basis vector $e_{1}$.
More specifically, when considering the action 
of  $\widetilde \GL^+\!(2,\bbR)$ on $\tbbAo$ (see \S \ref{sect:gentori}), the function $\theta$ can also be read off as the 
$\theta$-coordinate of $$ g \cdot (r,0) = \; (s, \theta(g)) \,  \in \tbbAo \; . $$
\hspace{1cm}\\

The following properties of the function $\theta$ are easy to
verify:
\begin{lemma} \label{lem:theta}
Let $g,h \in  \widetilde \GL^+\!(2,\bbR)$. Then 
\begin{enumerate}
\item $\theta(g) = 0$ if and only if $g \in AN$. 
\item $\theta(kg) = \theta(k) + \theta(g)$, for all $k \in \tilde K$;  
$\theta(\tau^m) = m \pi$.
\item $ | \theta(gh) - \theta(g) - \theta(h)| < \pi$.
\item $ | \theta(g) - \theta(g^{-1}) | <  \pi$.
\item $  | \theta(gkg^{-1}) -\theta(k) | < \pi$, for all $k \in \tilde K$.
\item $  | \theta (g h  g^{-1})|  < \pi$, for all $h \in AN$.
\end{enumerate}
\end{lemma}
\begin{proof} Recall (see \S \ref{sect:gentori}) that the action of $AN$ on $\tbbAo$ preserves all lines $(r, \ell \pi) \in \tbbAo$, where $\ell \in \bbZ$,
and the interior of all strips 
$$\bar \Omega_{\ell} = \{ (r, \theta) \mid \ell \pi \leq \theta \leq (\ell +1) \pi \} \, \subset \,  \tbbAo \, . $$
Therefore, for any $g \in 
\GL^+\!(2,\bbR)$ with $\theta(g) \in (\ell \pi,  (\ell +1) \pi)$ and $h \in AN$, we have $$ \ell \pi < \theta(h gh^{-1}) < (\ell +1) \pi \; . $$ In particular, one deduces (5) and (6). 
\end{proof}

\subsection{The induced covering on conjugacy classes}
Let $G$ be a Lie group, and 
$$\cC{G} = \{ C(g) = \Ad(G) g \mid g \in G \}$$
its set of conjugacy classes. The set $\cC{G}$ 
carries the quotient topology induced
from $G$. Observe that the center of $G$ acts on $\cC{G}$.
Indeed, for any $z \in Z(G)$, we have $z C(g) = C(zg)$.
Given a covering 
projection of Lie groups 
$\mathsf{p}: G' \ra G$,  there is a natural induced surjective map on
conjugacy classes $$ \cC{G'} \lra \cC{G} \,  , \; \,  C(g) \mapsto C(\mathsf{p}(g)) \; . $$  The kernel $\kappa$ of the covering is a central subgroup 
of $G'$ which acts on $\cC{G'}$. As a matter of fact, 
$ \cC{G} =  \cC{G'}/ \kappa$ is the quotient space of this
action.

\begin{proposition} \label{prop:ind_cov1}
The natural projection map on conjugacy classes 
$$ \cC{ \widetilde \GL^+\!(2,\bbR)} \, \longrightarrow \, \cC{ \GL^+\!(2,\bbR)} $$
is a covering map.
\end{proposition}
\begin{proof} We consider the action  of the kernel $\kappa = \langle \tau ^2 \rangle$ of the covering map $\mathsf{p}$ on  $\cC{ \widetilde \GL^+\!(2,\bbR)}$. 
For this, let $g \in  \widetilde \GL^+\!(2,\bbR)$ and consider its neighborhood  $$ U_{\epsilon} = U_{\epsilon}(g) = \{ h \in \widetilde \GL^+\!(2,\bbR) \mid | \theta(g) - \theta (h)| < \epsilon \}  \; . $$
We also put  $$\cC U_{\epsilon} = \{  C(h) \mid h \in U_{\epsilon} \}$$ 
for the corresponding neighborhood of $C(g)$ in the
space of conjugacy classes. 

Let us assume first that $g \in  \tilde K D$, $g \notin \mathcal Z D$.
Let $V_{\epsilon} \subset U_{\epsilon}(g)$ be a neighborhood
of $g$, such that all its elements are conjugate 
to an element of $ \tilde K \cdot D$. Let $h \in V_{\epsilon}$. 
By using (5) of Lemma \ref{lem:theta}, we deduce that, for all 
$\ell \in C(h)$, \begin{equation} \label{eq:Ueps}
| \theta(g) - \theta(\ell)| < \pi  + \epsilon \; . \end{equation}
%
We observe that the open subsets $\cC  V_{\epsilon}$ and  $\tau^{k}\cC  V_{\epsilon} = \cC \tau^{k} V_{\epsilon}$ of $\cC{ \widetilde \GL^+\!(2,\bbR)}$ intersect if and only if there exist elements
$h,\ell \in V_{\epsilon}$ such that $$ \tau^k  \ell \in C(h) \, . $$ 
If this is the case then, by the above estimate \eqref{eq:Ueps}, we have
 \begin{equation} \label{eq:Ueps2} | \theta(g) - \theta(\tau^k \ell) | = |  \theta(g) - k \pi - \theta(\ell)|< \pi  + \epsilon \; . \end{equation}
Furthermore, $  |\theta(g) -\theta(\ell)| < \epsilon$, since $\ell \in U_{\epsilon}$. If $\epsilon$ is small \eqref{eq:Ueps2} is possible if and only if $k \in \{0,1,-1\}$. For $\epsilon$ small enough, this implies that all neighborhoods of the form 
$\tau^{2k}\cC  V_{\epsilon} = \cC  \tau^{2k} V_{\epsilon}$ are mutually disjoint. Therefore,  $\cC V_{\epsilon}$ is a fundamental neighborhood of $\cC(g)$ for the action of $\kappa$ on 
$\cC\!  \GL^+\!(2,\bbR)$.

Assume next that $g \in AN$ is upper triangular. Since $g$ has real and positive eigenvalues, we may choose a small neighborhood $V_{\epsilon}$ as above such that all its elements
are conjugate to an element of $AN$ or of $\tilde K D$. 
In particular, for all $h \in V_{\epsilon}$ which are conjugate to
an element of  $AN$, we deduce from (6) of Lemma \ref{lem:theta} that the range of $\theta$ on the conjugacy class 
$\cC(h)$ is contained in the open interval $(- \pi, \pi)$. 
Consequently, 
$\theta (\cC( \tau^k h)) $ is contained in $( (k -1) \pi, (k +1) \pi)$.
It follows that all neighborhoods of the form 
$\tau^{2 k } \cC  V_{\epsilon}$ are mutually disjoint, and 
thus $\cC V_{\epsilon}$
is a fundamental neighborhood of $\cC(g)$ for the action of $\kappa$.

An analogous argument works for $g$ with negative eigenvalues, that is,
$ g \in \tau AN$. 
Therefore $\kappa$ acts discontinuously and freely on $\cC{ \widetilde \GL^+\!(2,\bbR)}$. This implies the proposition.
\end{proof}

\begin{corollary} \label{cor:ind_cov1}
The natural map on conjugacy classes 
$$  \cC{ \widetilde \GL\!(2,\bbR)} \, \longrightarrow \, \cC{ \GL\!(2,\bbR)} $$
is a local homeomorphism.
\end{corollary}
\begin{proof} Indeed, local injectivity is implied 
by the commutative diagram 
\begin{align*} 
 \xymatrix{
\; \;  \cC{ \widetilde \GL^+\!(2,\bbR)} \; \; \ar[d] \ar[r]
& \; \;  \cC{ \GL^+\!(2,\bbR)} \; \; \ar[d]  \\
\; \;  \cC{ \widetilde \GL\!(2,\bbR)} \; \; \ar[r] &  \; \; 
 \cC{ \GL\!(2,\bbR)}   \,  .}  
\end{align*}
\end{proof}

%
%

\paragraph{Closures of sets of conjugacy classes }
For any subset $M \subset G$ we define $\cC M$ to be the set of conjugacy classes of elements in $M$, and $\overline{\cC M}$ its closure in $\cC G$.
%
We shall require the following lemma: 

\begin{lemma} \label{lem:tGL2con}
With the above convention the following hold in the space of
 conjugacy classes $\cC  \widetilde \GL^+\!(2,\bbR)$:
\begin{enumerate}
\item Let $g_{i} \in  \widetilde \GL^+\!(2,\bbR)$ be a sequence of elements such that each $g_{i}$ is  conjugate to an element of $A N$ and such that
the sequence $|\theta(g_{i})|$ converges to $\pi$. Then the sequence 
$g_{i}$ leaves every compact subset of\/ $\widetilde \GL^+(2,\bbR)$. 
\item $\cC(AN) = \overline{\cC(AN)}$ is a closed subset of $\cC  \widetilde \GL^+\!(2,\bbR)$.
\item $\cC N = \overline{\cC N} 
%
=  \left\{ \cC\!\begin{matrix}{cc} 1 & 1 \\ 0 & 1 \end{matrix}, 
\cC\!\begin{matrix}{cc} 1 & -1 \\ 0 & 1 \end{matrix}, \E_{2} \right\}$  consists of three conjugacy classes.
\item $\overline{\cC \tilde K} = \cC \tilde K \, \cup \, \bigcup_{k} 
\! \cC \, \tau^k  N 
 $, $\; \overline{\cC D \tilde K} =\,  D \,\overline{\cC \tilde K}$.  
\end{enumerate}
\end{lemma}
\begin{proof} 
Let  $g \in \widetilde \GL^+\!(2,\bbR)$ such that 
$|\theta(g)|= \pi $, and $\bar g \in \GL^+\!(2,\bbR)$ the projection
of $g$. Clearly, by definition of $\theta$, $\bar g$  
has at least one negative eigenvalue. 
The sequence $g_{i}$ can not have a subsequence convergent
to $g$, since the corresponding $\bar g_{i}$ have positive eigenvalues. 
Thus  (1) follows.

To prove (2), we consider the subset $C(AN) \subset\widetilde \GL^+\!(2,\bbR)$ which is the preimage of $\cC(AN)$. For $g \in \widetilde \GL^+\!(2,\bbR)$, let $\dis(g)$ denote  the discriminant of the characteristic polynomial of $\mathsf{p}(g) \in \GL^+(2,\bbR)$. Then $g \in C(AN)$ if and only 
if the following hold: 
\begin{enumerate}
\item[i)] $|\theta(g)| < \pi$,
\item[ii)] $\dis(g) \geq 0$, 
\item[iii)] both eigenvalues of $\mathsf{p}(g)$ are positive.
\end{enumerate}
In view of (1), the condition i) is closed. Therefore,
$C(AN)$ is a closed subset of  $\widetilde \GL^+\!(2,\bbR)$, proving (2). 

Regarding (4), remark first that the closure of $\cC K$ in 
$\cC  \GL^+(2,\bbR)$ is contained in the union of $\cC K$
and $\cC N$, and $\cC \! -\E_{2}N$.  Now here is an example of  a  sequence 
$$k_{\varphi} = \begin{matrix}{cc} \cos \varphi + \sqrt{\sin \varphi} & - \sin \varphi - 1 \\ 
 \sin \varphi &  \cos \varphi - \sqrt{\sin \varphi} 
 \end{matrix}$$
 of matrices, where $k_{\varphi}$ is conjugate to the rotation $\R_{\varphi} \in \tilde K$, and 
 which, for $\varphi \rightarrow 0$ is converging to 
 $$ 
 \begin{matrix}{lr} 1 &  - 1 \\ 
 0 &  1
 \end{matrix} \, \in  N  \; .
 $$ Therefore, $\cC N$ is in the closure of $\cC \tilde K$. 
 Since $\cC \tilde K$ is invariant by left-multiplication with $\tau$,
 the same is true for its closure. 
 This shows that
 $\tau^k  \, \cC N \subset \overline{\cC \tilde K}$. 
\end{proof}

Every element $g \in \widetilde \GL^+\!(2,\bbR)$ with 
$\mathsf{p}(g) \in AN \leq  \GL^+\!(2,\bbR)$ is of the form 
$\tau^k g_{o}$, 
where $g_{o} \in AN  \leq  \widetilde \GL^+\!(2,\bbR)$. 
The integer $\lev g = k  \in \bbZ$ is
called the \emph{level} of $g$. The notion of level is 
defined for the conjugacy class $\cC(g)$ of $g$. 
The following states that the level separates 
these conjugacy classes. 
In particular, the subset of conjugacy classes 
$$ \tau^k \cC AN \, \subset \, \cC \widetilde \GL^+\!(2,\bbR)$$
is closed.

\begin{proposition}  \label{prop:levclosed}
Let $g_{i} \in  \widetilde \GL^+\!(2,\bbR)$ 
be a sequence  such that each $\mathsf{p}(g_{i})$  is  conjugate to an element of $A N$. If the sequence of conjugacy classes $\cC(g_{i})$ converges to $\cC(h) \in \cC \widetilde \GL^+\!(2,\bbR)$ then $\mathsf{p}(h)$ is conjugate to an element of  $AN$, and there exists $i_{0}$, such that for all $i \geq i_{0}$, $\lev g_{i} = \lev  h$. 
\end{proposition}
\begin{proof} The discriminant of the characteristic polynomial $\dis \mathsf{p}(g_{i})$ and the 
eigenvalues of $\mathsf{p}(g_{i})$ are continuous functions 
on the conjugacy classes. Therefore, $\mathsf{p}(h)$ is conjugate in  $\GL^+\!(2,\bbR)$ to an element of  $AN$. 
By assumption, all $\mathsf{p}(g_{i})$ are contained 
in $\cC AN$. 
Therefore, we have $g_{i} \in \tau^{2k_{i}} \cC(h_{i})$ with 
$h_{i} \in  \cC AN$. By Proposition \ref{prop:ind_cov1}, 
the group generated by $\tau^2$ acts properly discontinuously
on $\cC \widetilde \GL^+\!(2,\bbR)$.
In particular, there exists a neighbourhood $\cC U$ of $\cC(h)$ 
such that $\cC(g_{i}) \in \cC U$ implies $\lev g_{i} = \lev h$.
\end{proof}
%
Incidentally, the assertion of Proposition \ref{prop:ind_cov1} fails 
to be true when considering the situation for
the covering $$ \mathrm{P}\mathsf{p}: \widetilde \GL(2,\bbR)  \ra \PGL(2,\bbR) = \GL(2,\bbR) / \{ \pm \mathrm{E}_{2} \} \; . $$
\begin{example} \label{ex:CPGL}
We consider the induced map 
\begin{equation} \label{eq:cPGL}
\cC \widetilde \GL(2,\bbR)  \lra \cC \PGL(2,\bbR)
\end{equation} on conjugacy 
classes. It is the quotient map of $\cC \widetilde  \GL^+\!(2,\bbR)$ with 
respect to the action of the central subgroup $\cZ = \langle \tau \rangle$, generated by the element $\tau$.
Then the  $\widetilde \GL(2,\bbR)$-conjugacy class $\cC(g_{a})$, where $g_{a} \in \widetilde \GL^+\!(2,\bbR)$ is a lift of 
$$  \bar g_{a } = \begin{matrix}{cc} 0 & a \\ -a & 0 \end{matrix}  \in  \GL^+\!(2,\bbR) \; , \; a \neq 0 , $$
is fixed by translation with $\tau$. Indeed,  $\tau \cC(g_{a}) =
\cC(g_{-a}) =  \cC(g_{a})$. Therefore, the map \eqref{eq:cPGL} cannot be
a covering. It is \emph{not even a locally injective map}:
Indeed, in every neighborhood of $g_{a}$ there exist elements $g_{a,\epsilon}$ projecting to matrices of the form 
$$  \bar g_{a,\epsilon } = \begin{matrix}{rc} \epsilon  & a \\ -a  & \epsilon \end{matrix}  \in  \GL^+\!(2,\bbR) \; , \; \epsilon \neq 0 .$$
Then for the conjugacy classes in $\GL(2,\bbR)$, we have  $\cC \, \mathsf{p}(g_{a,\epsilon}) = \cC \, \mathsf{p}(g_{-a,\epsilon})$ and therefore 
$\cC\,  \mathrm{P}\mathsf{p}(g_{a,\epsilon}) = \cC\,  \mathrm{P}\mathsf{p} ( \tau g_{-a, \epsilon}) =  \cC\,  \mathrm{P}\mathsf{p} ( g_{ a,- \epsilon})$.  But clearly, $g_{a,\epsilon}$ and $g_{a,- \epsilon}$ are not
conjugate in $\widetilde \GL(2,\bbR)$ unless $\epsilon = 0$. 
Therefore,  \emph{the map \eqref{eq:cPGL}
is a twofold branched covering near $g_{a}$}.
\end{example}

\subsection{Conjugacy classes of homomorphisms} 
Let $G$ be a Lie group. 
Recall that the evaluation map on the generators 
 $$ \Hom(\bbZ^2,  G) \ra G \times G \, , \;   \rho \mapsto \left(\rho(e_{1}), \rho(e_{2}) \right)$$
identifies the space $\Hom(\bbZ^2,  G)$ of all homomorphisms 
$\bbZ^2 \ra G$ homeomorphically with an analytic subvariety of $G \times G$. 
With respect to this map,  the orbits of the conjugation action
of $G$ on $\Hom(\bbZ^2,  G)$ correspond to sets of the form 
$$ C(g_{1}, g_{2}) =  \{ (g g_{1} g^{-1}, g g_{2} g^{-1}) \mid g \in G\} \subset G \times G \; .$$
We put 
$$ {\mathcal X}(\bbZ^2,  G) = \Hom(\bbZ^2, G) / G $$
for the space of conjugacy  classes of homomorphisms
$\bbZ^2 \ra G$ (also called the character variety). 
Given a covering homomorphism $\mathsf{p}: G' \ra G$  there is a natural induced surjective map 
\begin{equation} \label{eq:indmap}
{\mathcal X}(\bbZ^2,  G') \lra   {\mathcal X}(\bbZ^2,  G) \; . 
\end{equation}
Returning to our specific context we introduce the following extension of Proposition \ref{prop:ind_cov1}:

\begin{proposition} \label{prop:ind_cov2}
The induced map on conjugacy classes of homomorphisms
$$  \Hom(\bbZ^2,  \widetilde \GL^+\!(2,\bbR)) / 
\widetilde \GL^+\!(2,\bbR) \lra  \Hom(\bbZ^2,   \GL^+\!(2,\bbR)) / 
 \GL^+\!(2,\bbR)$$
 is a covering map.
\end{proposition}
\begin{proof} 
Let $\Z(G)$ denote the center of $G$.
This representation of  $\Hom(\bbZ^2,  G)$ as a subset of
$G \times G$ gives rise 
to an action
of $\Z(G) \times \Z(G)$ on $\Hom(\bbZ^2,  G)$ which is determined by 
$$   \left( (z_{1}, z_{2})\cdot \rho\right) \, (e_{i}) \, =  \; z_{i}  \rho(e_{i})\; , $$ where $z_{i} \in \Z(G)$.  
Moreover, 
it factors to an action of $\Z(G) \times \Z(G)$ 
on the space of conjugacy classes 
${\mathcal X}(\bbZ^2,  G)$.  

Let  $\kappa \leq Z(G')$ denote the kernel of $\mathsf{p}: G' \ra G$. By the above, $\kappa \times \kappa$ acts on $\Hom(\bbZ^2, G')$, and the action factors to an action on the space of conjugacy classes ${\mathcal X}(\bbZ^2,  G')$, and, as is easily verified, the natural map 
$$ {\mathcal X}(\bbZ^2,  G')/ \kappa \times \kappa \lra {\mathcal X}(\bbZ^2,  G) \;  $$ which is induced on the quotient 
is a homeomorphism. Therefore, \ref{eq:indmap}
 is a covering if and only if 
 $\kappa \times \kappa$ acts discontinuously and freely on 
 $ {\mathcal X}(\bbZ^2,  G')$. 

Here  we consider only the case
$G'= \widetilde \GL^+\!(2,\bbR)$ and $G=
\GL^+(2,\bbR)$. 
For any $(g_{1}, g_{2}) \in G' \times G'$, choose
open neighborhoods $U_{\epsilon}(g_{i})$ as 
in the proof of Proposition \ref{prop:ind_cov1}.
As follows from this previous proof,  the open neighborhood 
$$ U_{\epsilon}(g_{1},g_{2})  = U_{\epsilon}(g_{1}) \times U_{\epsilon}(g_{2}) \;  $$ 
projects to a fundamental neighborhood $ \mathcal C  U_{\epsilon}$ for  the action of $\kappa \times \kappa$ on the set of all $G'$-orbits. 
Hence, $\kappa \times \kappa$ acts discontinuously on $G'$-orbits. 
\end{proof}

\begin{corollary} \label{cor:ind_cov2}
The natural map on conjugacy classes of homomorphisms
$$  \Hom(\bbZ^2,  \widetilde \GL^+\!(2,\bbR)) / 
\widetilde \GL\!(2,\bbR) \lra  \Hom(\bbZ^2,   \GL^+\!(2,\bbR)) / 
 \GL\!(2,\bbR)$$
 is a local homeomorphism.
\end{corollary}

\section{Example of a two-dimensional geometry where
$hol$ is not a local homeomorphism}

Let $(X,G)$ be the homogeneous geometry which is defined by
the natural action of $\PGL(2,\bbR) =  \GL(2,\bbR)/ \{ \pm 1 \}$ 
on the space
$$ X = \PbbAo = \bbAo / \{ \pm 1 \} \; , $$ that is, $X$ is the quotient space
of $\bbR^2 - \{0 \}$ by the action of the center $\{ \E_{2}, -\E_{2}\}$ of
$\SL(2,\bbR)$. The natural map $$ (\bbAo, \GL(2,\bbR)) \; \lra  \; 
 ( \PbbAo, \PGL(2,\bbR))$$ is a covering of geometries in the
sense of Definition \ref{def:subgeometry}.
By Lemma \ref{lem:covering_geoms}, the induced map 
on deformation spaces 
\begin{equation}  \label{eq:indmap_p}
\Def(T^2, \bbAo) \; \lra \;  \Def(T^2, \PbbAo) 
\end{equation}
is a homeomorphism. \\

We claim that the \emph{holonomy 
for the deformation space $ \Def(T^2, \PbbAo)$ is \emph{not} 
a local homeomorphism}.
For this we recall first the $ (\bbAo, \GL(2,\bbR))$-manifolds  
$\cH_{\lambda, {\pi \over 2},k}$ constructed in Example \ref{ex:finquothopftori}  (finite quotients  of Hopf tori).
Then we observe:  \index{holonomy map!is not a local homeomorphism} 

\begin{proposition}   \label{prop:branchedhol}
The holonomy map 
$$  \Def(T^2, \PbbAo) \, \stackrel{hol}{\lra} \, 
\Hom(\bbZ^2, \PGL(2,\bbR))/ \PGL(2,\bbR)$$
is a twofold branched covering near the image 
of a homogeneous flat affine torus $\cH_{\lambda, {\pi \over 2},k}$
under the map \eqref{eq:indmap_p}.  \end{proposition}
\begin{proof}  The commutative diagram \eqref{eq:subgeometry1} for the subgeometry takes the form
\begin{align*} 
 \xymatrix{
\; \; \Def(T^2, \bbAo) \; \; \ar[d]_{\approx} \ar[r]^(0.35){hol}
& \; \;\Hom(\bbZ^2, \GL^+\!(2,\bbR))/ \GL\! (2,\bbR) \; \; \ar[d]  \\
\; \; \Def(T^2, \PbbAo) \; \; \ar[r]^(0.35){hol}  &  \; \; 
 \Hom(\bbZ^2, \PGL(2,\bbR))/  \PGL(2,\bbR)  \,  .}  
\end{align*}
Note that, by Corollary \ref{cor:bbaotop}, the top horizontal map
is a local homeomorphism.  Furthermore, by Example \ref{ex:CPGL},
the right vertical map is a twofold branched covering near the
holonomy homomorphism of every flat affine 
torus $\cH_{\lambda, {\pi \over 2},k}$. 
We deduce that the bottom
map $hol$ for $ \Def(T^2, \PbbAo) $ 
is locally a twofold branched covering at the
images of $\cH_{\lambda, {\pi \over 2},k}$.
\end{proof}

\end{appendix}

\printindex 
\end{document}